\documentclass{amsart}

\usepackage[letterpaper,hmargin=1in,vmargin=0.8in]{geometry}
\usepackage{amsmath}
\usepackage{amsthm}
\usepackage{amssymb}
\usepackage{mathtools}
\usepackage{mathbbol}
\usepackage{soul}
\usepackage{tikz}
\usetikzlibrary{matrix,arrows,cd,calc}
\usepackage{comment}
\usepackage[hidelinks]{hyperref}
\usepackage{xcolor}
\definecolor{darkgreen}{rgb}{0,0.45,0}
\hypersetup{colorlinks,urlcolor=blue,citecolor=darkgreen,linkcolor=darkgreen,linktocpage}
\usepackage[capitalize]{cleveref}
\usepackage{xspace}
\usepackage{caption}
\usepackage{subcaption}
\usepackage{enumerate}
\usepackage{centernot}
\usepackage{textcomp}
\usepackage{thmtools}
\usepackage{thm-restate}
\usepackage[normalem]{ulem} 

\def\noteson{%
    \gdef\luis##1{\noindent{\color{blue}[Luis: ##1]}}%
    \gdef\steve##1{\noindent{\color{red}[Steve: ##1]}}%
    \gdef\magnus##1{\noindent{\color{magenta}[Magnus: ##1]}}%
    \gdef\steffen##1{\noindent{\color{orange}[Steffen: ##1]}}%
    \gdef\todo##1{\noindent{\color{violet}[to do: ##1]}}
    \gdef\revcom##1{\noindent{\color{red}[addresses comment ##1]}}%
    }

\noteson

\newtheorem{theorem}{Theorem}[section]
\newtheorem{lemma}[theorem]{Lemma}
\newtheorem{conjecture}[theorem]{Conjecture}
\newtheorem{proposition}[theorem]{Proposition}
\newtheorem{corollary}[theorem]{Corollary}

\newtheorem{resultx}{Result}

\theoremstyle{definition}
\newtheorem{definition}[theorem]{Definition}
\newtheorem{example}[theorem]{Example}

\newtheorem{remark}[theorem]{Remark}
\newtheorem{notation}[theorem]{Notation}

\renewcommand{\colon}{:}

\definecolor{light-gray}{gray}{0.8}

\newcommand{\hook}{%
{\,\begin{tikzpicture}[scale=0.05]
   \draw[-] (0,0) -- (2,0);
   \draw[-] (0,0) -- (0,2);
   \draw[-] (2,0) -- (2,1);
   \draw[-] (0,2) -- (1,2);
   \draw[-] (1,1) -- (1,2);
   \draw[-] (1,1) -- (2,1);
\end{tikzpicture}}}

\newcommand{\rectangle}{%
{\,\begin{tikzpicture}[scale=0.05]
   \draw[-] (0,0) -- (2,0);
   \draw[-] (0,0) -- (0,2);
   \draw[-] (2,0) -- (2,2);
   \draw[-] (0,2) -- (2,2);
\end{tikzpicture}}}

\renewcommand{\sp}{\textup{\texttt{+}}}
\newcommand{\sm}{\textup{\texttt{-}}}
\newcommand{\spm}{{\textup{\texttt{\textpm}}}}

\renewcommand{\to}{\xrightarrow{\;\;\;}}
\renewcommand{\nrightarrow}{\centernot{\xrightarrow{\;\;\;\;}}}

\renewcommand{\mod}{\mathrm{mod}\xspace}

\newcommand{\define}[1]{\textit{#1}}

\renewcommand{\epsilon}{\varepsilon}
\renewcommand{\phi}{\varphi}

\newcommand{\Erank}{\mathpzc{rk}}
\newcommand{\Eupset}{\mathpzc{lim}}

\DeclareMathOperator{\gldim}{gldim}
\DeclareMathOperator{\pdim}{pdim}
\newcommand{\gldimrank}{\gldim^\Erank}
\newcommand{\pdimrank}{\pdim^\Erank}
\newcommand{\gldimupset}{\gldim^\Eupset}
\newcommand{\pdimupset}{\pdim^\Eupset}
\newcommand{\gldimEcal}{\gldim^\Ecal}

\let\ker\undefined
\DeclareMathOperator{\ker}{ker}
\DeclareMathOperator{\coker}{coker}
\DeclareMathOperator{\Hom}{Hom}
\DeclareMathOperator{\Tor}{Tor}
\DeclareMathOperator{\Ext}{Ext}
\DeclareMathOperator{\Lan}{Lan}
\DeclareMathOperator{\End}{End}
\DeclareMathOperator{\add}{add}
\DeclareMathOperator{\Span}{Span}
\DeclareMathOperator{\proj}{proj}

\newcommand{\coim}{\mathsf{coim}}
\newcommand{\im}{\mathsf{im}}
\newcommand{\rk}{\mathsf{rk}}

\newcommand{\mrd}{\sbarc^{\rectangle}}
\newcommand{\mrdh}{\sbarc^{\hook}}
\newcommand{\bettibars}{\mathsf{b}}
\newcommand{\bettibarsrank}{\mathsf{b}^\Erank}
\newcommand{\decbars}{\mathsf{b}^{\rectangle}}
\newcommand{\barc}{\Bcal}
\newcommand{\sbarc}{\barc}
\newcommand{\barcEcal}{\sbarc^\Ecal}
\newcommand{\barcrank}{\sbarc^\Erank}
\newcommand{\bettirank}{\beta^\Erank}

\newcommand{\vect}{\mathrm{vec}}
\newcommand{\Vect}{\mathrm{Vec}}
\newcommand{\Ab}{\mathrm{Ab}}
\newcommand{\Modcat}{\mathrm{Mod}}
\newcommand{\modcat}{\mathrm{mod}}

\newcommand{\fpp}{\textit{fp}\xspace}

\DeclareMathAlphabet{\mathpzc}{OT1}{pzc}{m}{it}
\usepackage[mathscr]{euscript}

\makeatletter
\newcommand\DEFINEALPHABETLOOP[3]{%
  \ifx\relax#3\expandafter\@gobble\else\expandafter\@firstofone\fi
  {\expandafter\newcommand\expandafter*\csname#3#1\endcsname{#2{#3}}%
   \DEFINEALPHABETLOOP{#1}{#2}}%
}%
\newcommand\Definealphabet[2]{%
  \DEFINEALPHABETLOOP{#1}{#2}abcdefghijklmnopqrstuvwxyzABCDEFGHIJKLMNOPQRSTUVWXYZ\relax
}%
\makeatother
\Definealphabet{bb}{\mathbb}
\Definealphabet{cal}{\mathcal}
\Definealphabet{bf}{\mathbf}
\Definealphabet{sf}{\mathsf}
\Definealphabet{rm}{\mathrm}
\Definealphabet{frak}{\mathfrak}
\Definealphabet{scr}{\mathscr}

\title[Bottleneck stability of rank decompositions]{On the bottleneck stability of rank decompositions\\of multi-parameter persistence modules}
\author{Magnus Bakke Botnan}
\address{
Vrije Universiteit,
Department of Mathematics,
Faculteit der Exacte Wetenschappen,
Amsterdam}
\author{Steffen Oppermann}
\address{
Department of Mathematical Sciences,
NTNU, Trondheim, Norway}
\author{Steve Oudot}
\address{
Inria and \'Ecole Polytechnique,
Palaiseau, France}
\author{Luis Scoccola}
\address{
Mathematical Institute,
University of Oxford, United Kingdom}

\begin{document}

\maketitle

\begin{abstract}
    A significant part of modern topological data analysis is concerned with the design and study of algebraic invariants of poset representations---often referred to as persistence modules.
    One such invariant is the minimal rank decomposition, which 
    encodes the ranks of all the structure morphisms of the persistence module by a single ordered pair of rectangle-decomposable modules, interpreted as a signed barcode.
    This signed barcode generalizes the concept of persistence barcode from one-parameter persistence to any number of parameters, raising the question of its bottleneck stability.
    We show in this paper that the minimal rank decomposition is not stable under the natural notion of signed bottleneck matching between signed barcodes.
    We remedy this by turning our focus to the rank exact decomposition, a related signed barcode induced by the minimal projective resolution of the module relative to the so-called rank exact structure, which we prove to be bottleneck stable under signed matchings.
    As part of our proof, we obtain two intermediate results of independent interest: we compute the global dimension of the rank exact structure on the category of finitely presentable multi-parameter persistence modules, and we prove a bottleneck stability result for hook-decomposable modules.
    We also give a bound for the size of the rank exact decomposition that is polynomial in the size of the usual minimal projective resolution, we prove a universality result for the dissimilarity function induced by the notion of signed matching, and we compute, in the two-parameter case, the global dimension of a different exact structure related to the upsets of the indexing poset.
    This set of results combines concepts from topological data analysis and from the representation theory of posets, and we believe is relevant to both areas.
\end{abstract}

\medskip

\setcounter{tocdepth}{1}
\tableofcontents



\medskip


\medskip

\section{Introduction}

\subsection{Context and objectives}

The design and study of discrete invariants for persistence modules---poset representations, in the language of representation theory---lies at the heart of topological data analysis (TDA), which uses these invariants as features, regularizers, or estimators in applications. The discrete nature of the invariants is paramount, enabling their accurate construction and comparison on a computer.  In the one-parameter setting---the setting in which the indexing poset is a linear order---the persistence barcodes, or equivalently the persistence diagrams, have emerged as the canonical choice of discrete invariant, notably because they are provably complete and comparatively easy to interpret.

The situation in the multi-parameter setting---the setting in which the indexing poset is not a linear order and is typically a product of linear orders---is much less clear-cut: 
even fairly simple non-linear posets, including sufficiently large 2-dimensional grids, are of wild representation type \cite{Nazarova} and thus a full classification of the indecomposable representations of these posets is generally believed to be a hopelessly difficult task;
for manifestations of this phenomenon in the case of multi-parameter persistence, see \cite{carlsson-zomorodian,bauer-scoccola}.
Research in multi-parameter TDA within the past decade has therefore considered alternative, incomplete, discrete invariants, most notably: the Hilbert function and its associated Hilbert series~\cite{harrington-otter-schenck-tillmann}; the multigraded Betti numbers~\cite{lesnick-wright,berkouk2019stable}; the fibered barcode~\cite{Landi,lesnick-wright}; the rank invariant~\cite{carlsson-zomorodian} and its generalized version~\cite{kim-memoli}; the generalized persistence diagrams obtained through M\"obius inversion~\cite{kim-memoli,patel,clause-kim-memoli}, compressed multiplicities~\cite{asashiba2019approximation}, and minimal rank decompositions~\cite{botnan-oppermann-oudot}, which have strong ties with one another; the rank decompositions coming from rank exact resolutions~\cite{botnan-oppermann-oudot}; the persistence contours and their associated simple noise systems~\cite{gafvert2017stable,scolamiero2017multidimensional}; and the stable rank and its associated shift dimension~\cite{chacholski2021shift,gafvert2017stable}. The significant increase in the number and variety of these invariants over time calls for a general effort to classify and sort them according to their respective qualities. This, in turn, raises the question of defining what it means to be a {\em good} discrete invariant.

To the end user in data sciences, the most natural measure of quality for a discrete invariant would be its performances in machine learning or data mining applications. But in order to explain these performances, prior foundational developments are required, including notably the study of the stability properties of the invariant under perturbations of the persistence module it describes:
for example, in the one-parameter case, the stability of barcodes is a main ingredient in the proof of convergence of stochastic gradient descent for persistence-based optimization \cite{carriere}, as well as in proofs of consistency of persistence-based statistics (see, e.g., \cite{chazal2021introduction} for an overview).
For this, a suitable choice of distance or dissimilarity measure between invariants must be made, with the following trade-off in mind: on the one hand, stability bounds are sought for;
on the other hand, the dissimilarity should be as discriminative as possible. Meanwhile, for practical purposes, the dissimilarity should be computationally tractable. In the one-parameter setting, the bottleneck distance is the metric of choice on persistence barcodes or diagrams, as it is computationally tractable, easily interpretable (being a best-matching distance), and provably equivalent (via an isometry theorem) to the universal interleaving distance between the persistence modules themselves.

Again, the situation in the multi-parameter setting is more contrasted. First, bottleneck-type stability results for discrete invariants in this context are scarce; second, 
the computation of available metrics, such as for instance the matching distance between fibered barcodes~\cite{bjerkevik-kerber,kerber2018exact,kerber_et_al:LIPIcs:2020:12211}, poses some serious challenges.
A recent development~\cite{oudot-scoccola} aimed at resolving these issues by extending the classical bottleneck distance on persistence barcodes to a signed version designed for the comparison of multigraded Betti numbers.
While not a pseudo-metric (the triangle inequality may be infringed), this dissimilarity does satisfy a universality property, it is just as easy to compute as its unsigned counterpart, and it makes the multigraded Betti numbers stable.

It follows from the rank-nullity theorem from linear algebra that the Hilbert function (i.e., the pointwise dimension) of a multi-parameter persistence module can be seen as the Euler characteristic with respect to the multigraded Betti numbers, meaning that the alternating sum of the pointwise dimension function of the projectives represented by the multigraded Betti numbers equals the pointwise dimension of the module \cite{oudot-scoccola}.
Our focus here is on the study of an invariant through the lens of the signed bottleneck dissimilarity.
This invariant is given by Betti numbers relative to an exact structure other than the standard one, namely the rank exact structure, which was introduced in~\cite{botnan-oppermann-oudot} for the purpose of proving the existence of decompositions of the rank invariant.
Part of the motivation comes from the fact that the Euler characteristic with respect to the Betti numbers relative to the rank exact structure yields the rank invariant, an invariant finer than the Hilbert function.
Our primary objective is to obtain bottleneck stability guarantees for these relative multigraded Betti numbers.
%
%
%

Our approach is part of a general program which aims to study multi-parameter persistence modules via relative homological algebra \cite{hochschild,auslander-solberg,enochs-jenda}.
The basic idea of the program is to associate an exact structure to a given invariant of persistence modules (such as the Hilbert function or the rank invariant), which allows one to decompose the value of the invariant on a certain module as a formal sum of simpler modules, namely, projective modules relative to the exact structure.
This formal sum can be interpreted as a signed barcode and gives a discrete representation of the invariant.
This program was initiated in \cite{botnan-oppermann-oudot} in the case of the rank invariant and was then extended to other invariants in \cite{blanchette-brustle-hanson}.
Recent papers following this line of work include
\cite{asashiba-escolar-nakashima-yoshiwaki}, which focuses on the exact structure whose projective modules are the interval decomposable modules; and \cite{cacholski-guidolin-ren-scolamiero-tombari}---made public after the first version of this paper was made public---which focuses on the effective computation of relative Betti numbers of modules over posets using Koszul resolutions.
The present work takes the program one step further by bringing stability questions into the picture, considering the stability of relative projective resolutions of poset representations in terms of interleavings and matchings.

\subsection{Mathematical framework and contributions.}
For details about the notions used here, the reader may refer to \cref{section:background}.
Let \(\Pscr\) be a poset and let \(\vect\) be the category of finite dimensional vector spaces over a fixed field \(\kbb\).
A (pointwise finite dimensional) \define{\(\Pscr\)-persistence module} is a functor \(\Pscr \to \vect\).
A \define{\(\Pscr\)-barcode} is a multiset\footnote{
For the formal notion of multiset used in this article, see \cref{section:multisets}.}
of isomorphism classes of indecomposable \(\Pscr\)-persistence modules, usually denoted by \(\Ccal\), and a \define{signed \(\Pscr\)-barcode} is an ordered pair \((\Ccal_\sp, \Ccal_\sm)\) of \(\Pscr\)-barcodes, for which we use the shorthand \(\Ccal_\spm\).
By the main result of \cite{botnan-crawleybovey},
and the Krull--Remak--Schmidt--Azumaya theorem, any \(\Pscr\)-persistence module can be decomposed as a direct sum of indecomposable persistence modules in an essentially unique way (see \cref{subsection:barcodes}).
The \define{barcode} of a \(\Pscr\)-persistence module \(M\), denoted \(\barc(M)\), is the multiset of isomorphism classes of indecomposables of \(M\), counted with multiplicity.
A \(\Pscr\)-persistence module is \define{finitely presentable} (\fpp) if it is the cokernel of a morphism between persistence modules that are finite direct sums of indecomposable projectives.
The subcategory spanned by \fpp \(\Pscr\)-persistence modules is denoted by \(\vect^\Pscr_\fpp\).
Of particular interest is the \(n\)-fold product poset \(\Rscr^n\), where \(\Rscr\) is the poset of real numbers with its standard linear order, and \(n\) is a fixed natural number.
Modules of the form \(\Rscr^n \to \vect\) are known as \define{\(n\)-parameter persistence modules} and are usually compared using the \define{interleaving distance} \cite{lesnick,chazal-silva-glisse-oudot}, denoted \(d_I\).
The interleaving distance is based on the notion of \define{\(\epsilon\)-interleaving} between \(\Rscr^n\)-persistence modules, which can be interpreted as an approximate version of the notion of isomorphism.

The \define{rank invariant} of a \(\Pscr\)-persistence module \(M\) is the function \(\rk(M) \colon \{(a,b) \in \Pscr \times \Pscr \colon a \leq b\} \to \Zbb\) where \(\rk(M)(a,b)\) is the rank of the structure morphism \(M(a) \to M(b)\).
A \define{rank decomposition} of a \(\Pscr\)-persistence module \(M\) is a signed \(\Pscr\)-barcode \(\Ccal_\spm\) such that \(\rk(M) = \sum_{A \in \Ccal_\sp} \rk(A) - \sum_{B \in \Ccal_\sm} \rk(B)\) as \(\Zbb\)-valued functions.
It is shown in \cite{botnan-oppermann-oudot} that every \fpp \(\Rscr^n\)-persistence module admits a rank decomposition \(\Ccal_\spm\) with the property that the modules in \(\Ccal_\sp\) and in \(\Ccal_\sm\) are (right open) rectangle modules.
Moreover, if one requires the multisets \(\Ccal_\sp\) and \(\Ccal_\sm\) to be disjoint, then the signed barcode \(\Ccal_\spm\) is uniquely determined, and is called the \define{minimal rank decomposition by rectangles} of \(M\), denoted \(\mrd_\spm(M)\).


The starting point of this work is the following bottleneck instability result for the minimal rank decomposition by rectangles, alluded to, but not proven, in \cite{oudot-scoccola}.
In order to state the result, we use the following notion of bottleneck matching for signed barcodes, from~\cite{oudot-scoccola}.
An \define{\(\epsilon\)-matching} between barcodes \(\Ccal\) and \(\Dcal\) is a partial bijection \(\Ccal \nrightarrow \Dcal\) with the property that matched indecomposables are \(\epsilon\)-interleaved, and unmatched ones are \(\epsilon\)-interleaved with the zero module.
An \(\epsilon\)-matching between signed barcodes \(\Ccal_\spm\) and \(\Dcal_\spm\) is an \(\epsilon\)-matching between \(\Ccal_\sp \cup \Dcal_\sm\) and \(\Dcal_\sp \cup \Ccal_\sm\).
The \define{bottleneck distance} is then defined by \(d_B(\Ccal, \Dcal) = \inf\{\epsilon \colon \exists \text{ \(\epsilon\)-matching } \Ccal\nrightarrow \Dcal \}\), and the \define{signed bottleneck dissimilarity} by
\[
    \widehat{d_B}\big(\Ccal_\spm, \Dcal_\spm\big) = d_B(\Ccal_\sp \cup \Dcal_\sm,\Dcal_\sp \cup \Ccal_\sm).
\]

\begin{resultx}[\cref{proposition:instability-mrd-rectangles}]
    \label{proposition:instability-mrd-rectangles-intro}
    Let \(n \geq 2 \in \Nbb\).
    There is no function \(f \colon \Rbb_{\geq 0} \to \Rbb_{\geq 0}\) with \(f(r) \xrightarrow{r \to 0} 0\) such that the following inequality holds for all \fpp modules \(M,N \colon \Rscr^n \to \vect\) at finite interleaving distance:
    \[
        \widehat{d_B}(\,\mrd_\spm(M)\,,\,\mrd_\spm(N)\,) \leq f\left(d_I(M,N)\right).
    \]
\end{resultx}

Considering minimal rank decompositions by shapes other than right open rectangles does not in general resolve 
the issue of instability: in \cref{remark:instability-mrdh}, we explain how a result analogous to \cref{proposition:instability-mrd-rectangles-intro} can be proven for the minimal rank decomposition by hooks, also introduced in \cite{botnan-oppermann-oudot} and recalled in \cref{section:rank-decs}.

We thus turn our focus to a different type of  rank decomposition, which may not be minimal in the above sense, but which, for this very reason, can be proven to be bottleneck stable.
We call this rank decomposition the \emph{rank exact decomposition}.
There are two main differences between the minimal rank decomposition by rectangles and the rank exact decomposition:
\begin{enumerate}[(i)]
    \item Instead of using rectangle modules to decompose the rank invariant, we use hook modules, whose definition we recall below.
    \item The notion of minimality of the minimal rank decomposition by rectangles---which requires the positive and negative barcodes to be disjoint---is replaced by the requirement that the signed barcode comes from a minimal projective resolution in the so-called rank exact structure, also described below.
\end{enumerate}

\begin{figure}[t]
    \centering
    \begin{subfigure}{0.75\linewidth}
        \centering
        \includegraphics[height=4.9em]{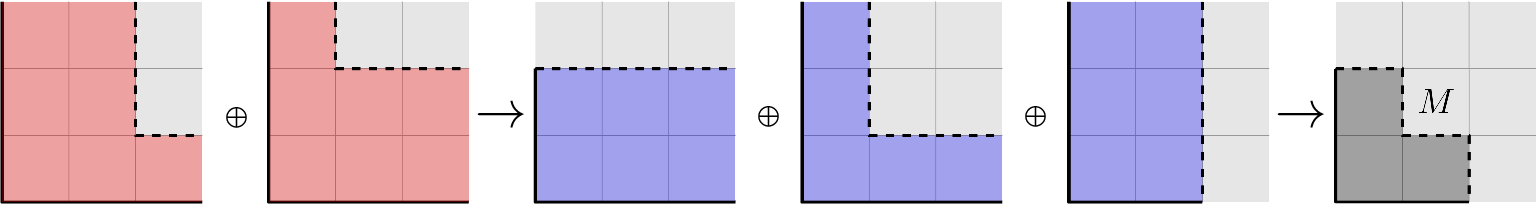}
        \caption{A minimal rank projective resolution of a persistence module \(M\).} 
        \label{figure:example-resolution-a}
    \end{subfigure}
    \hfill
    \begin{subfigure}{0.2\linewidth}
        \centering
        \includegraphics[height=4.9em]{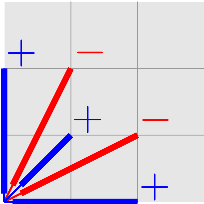}
        \caption{\(\barcrank_\spm(M)\).}
        \label{figure:example-resolution-b}
    \end{subfigure}
    \caption{
        The minimal rank projective resolution of \cref{figure:example-resolution-a} is used to define the rank exact decomposition of the \(\Rscr^2\)-persistence module \(M\), which, in this case, is an interval module.
        The rank exact decomposition of \(M\) is shown in \cref{figure:example-resolution-b} by representing an indecomposable summand \(\Lsf_{i,j}\) by a segment joining \(i\) and \(j\).
        Summands in even homological degrees correspond to positive bars; summands in odd homological degrees correspond to negative bars.
        For details, see \cref{example:rank-exact-resolution}.}
    \label{figure:example-resolution}
\end{figure}

In the case of the rank invariant, the right notion of resolution is that of a rank projective resolution~\cite{botnan-oppermann-oudot}.
A short exact sequence of \(\Rscr^n\)-persistence modules \(0 \to L \to M \to N \to 0\) is \define{rank exact} if \(\rk(M) = \rk(N) + \rk(L)\), and a long exact sequence is rank exact if it can be broken up into short exact sequences that are rank exact (see \cref{section:exact-structures} for a precise definition).
An \(\Rscr^n\)-persistence module \(P\) is \define{rank projective} if \(\Hom(P,-)\) maps rank exact sequences to exact sequences.
A \define{rank projective resolution} of an \(\Rscr^n\)-persistence module \(M\) is a resolution of \(M\) by rank projective modules, which, as a long exact sequence, is rank exact.
It is shown in \cite{botnan-oppermann-oudot} that the rank projective modules are precisely the hook-decomposable modules, that is, the modules that decompose as direct sums of hook modules, where the \define{hook module} \(\Lsf_{i,j}\) with \(i < j \in \Rscr^n \cup \{\infty\}\) is the interval module with support \(\{a \in \Rscr^n \colon a \geq i, a \ngeq j\}\).
Note that, in particular, hook modules are a generalization of the (right open) one-dimensional interval modules.
Let \(M\) be a \fpp \(\Rscr^n\)-persistence module and let \(P_\bullet \to M\) be a minimal rank projective resolution (the fact that this resolution exists and has finite length follows from \cite{botnan-oppermann-oudot}); see \cref{figure:example-resolution-a} for an illustration.
Let \(\bettirank_k(M) \coloneqq \barc(P_k)\), which is independent of the choice of minimal rank projective resolution.
The \define{rank exact decomposition} of \(M\) is the signed barcode
\[
    \barcrank_\spm(M) = \left( \bettirank_{2\Nbb}(M),\bettirank_{2\Nbb+1}(M)\right) = \left(\; \bigcup_{k \text{ even}} \bettirank_k(M)\;,\; \bigcup_{k \text{ odd}} \bettirank_k(M)\;\right);
\]
see \cref{figure:example-resolution-b} for an illustration.

\medskip 

Our first positive result is the following bottleneck stability result for the rank exact decomposition.
\cref{figure:example-stability} illustrates the kind of bottleneck matching that \cref{theorem:stability-rank-intro} guarantees to exist.
It is interesting to note that, when \(n=1\), \cref{theorem:stability-rank-intro} specializes to the well-known algebraic stability of one-parameter barcodes \cite{chazal-cohen-steiner-glisse-guibas-oudot,chazal-silva-glisse-oudot,bauer-lesnick}.

\begin{resultx}[\cref{theorem:stability-rank}]
    \label{theorem:stability-rank-intro}
    Let \(n \geq 1 \in \Nbb\).
    For all \fpp modules \(M,N \colon \Rscr^n \to \vect\), we have
    \[
        \widehat{d_B}(\, \barcrank_\spm(M)\,,\, \barcrank_\spm(N)\,) \leq (2n-1)^2 \cdot d_I(M,N).
    \]
    In other words, for all \fpp modules \(M,N \colon \Rscr^n \to \vect\), we have
    \[
        d_B\big(\;\bettirank_{2\Nbb}(M) \cup \bettirank_{2\Nbb+1}(N)\;,\;\bettirank_{2\Nbb}(N) \cup \bettirank_{2\Nbb+1}(M)\;\big) \leq (2n-1)^2 \cdot d_I(M,N).
    \]
\end{resultx}

We prove \cref{theorem:stability-rank-intro} by following \cite[Theorem~8.4]{oudot-scoccola}, which is a general stability result for signed barcodes that arise from a given exact structure.
Informally, the result says that the bottleneck stability of such a signed barcode follows from two main ingredients:
\begin{enumerate}
    \item A bottleneck stability result for the projectives relative to the exact structure.
    \item A finite upper bound for the global dimension of the exact structure.
\end{enumerate}
Most of our work goes into checking these two facts in the case of the rank exact structure, which yields two intermediate results of independent interest stated as  \cref{proposition-stability-hooks-intro} and \cref{theorem:gldim-rank-Rn-intro} below. 
In the proofs of \cref{theorem:stability-rank-intro} and \cite[Theorem~8.4]{oudot-scoccola}, ingredients (1) and (2) are combined using a persistent version of Schanuel's lemma (\cref{lemma:relative-stability-resolutions}), which applies, under mild assumptions, to any exact structure on the category of \(\Rscr^n\)-persistence modules.

\begin{figure}[t]
    \centering
    \begin{subfigure}{0.49\linewidth}
        \centering
        \includegraphics[height=6em]{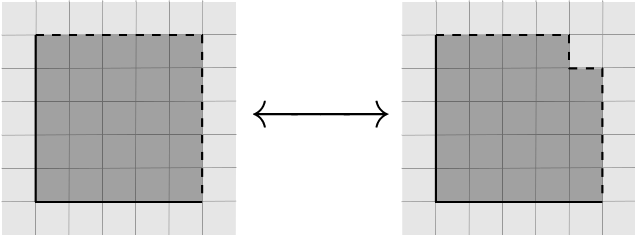}
    \end{subfigure}
    \begin{subfigure}{0.49\linewidth}
        \centering
        \includegraphics[height=6em]{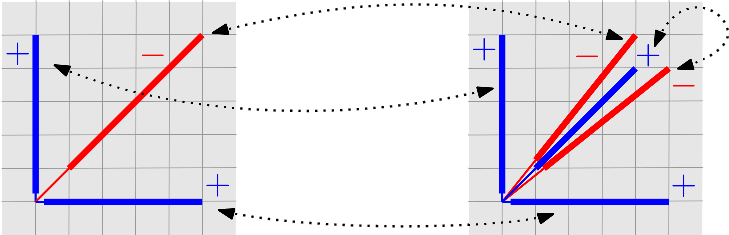}
    \end{subfigure}
    \caption{
        An example of a low-cost matching (right) between the rank exact decompositions of two \fpp \(\Rscr^2\)-persistence modules at small interleaving distance (left).
        For details, see \cref{example:stability}.
        }
    \label{figure:example-stability}
\end{figure}

\cref{proposition-stability-hooks-intro} adapts, from rectangles to hooks, Bjerkevik's bottleneck stability result for rectangle-decomposable modules \cite[Theorem~4.3]{bjerkevik}, and follows from a generalization of Bjerkevik's result, which we state as \cref{bjerkevik-thm-1}:

\begin{resultx}[\cref{proposition-stability-hooks}]
    \label{proposition-stability-hooks-intro}
    Let \(P,Q \colon \Rscr^n \to \vect\) decompose as countable direct sums of hook modules and let \(\epsilon \geq 0\).
    If \(P\) and \(Q\) are \(\epsilon\)-interleaved, then there exists a \((2n-1)\epsilon\)-matching between \(\barc(P)\) and \(\barc(Q)\).
\end{resultx}

We explain in \cref{remark:projective-stability-tight} how we can also adapt \cite[Example~5.2]{bjerkevik} to show that the bound in \cref{proposition-stability-hooks-intro} is tight when \(n=2\).

\cref{theorem:gldim-rank-Rn-intro}
gives the global dimension of the rank exact structure, denoted $\gldimrank$, on the category of \fpp \(\Rscr^n\)-persistence modules.
The result follows from \cref{theorem:gldim-rank-poset}, which provides upper and lower bounds for the global dimension of the rank exact structure on the category of persistence modules over a finite lattice.

\begin{resultx}[\cref{theorem:gldim-rank-Rn}]
    \label{theorem:gldim-rank-Rn-intro}
    Let $n \geq 1 \in \Nbb$.
    We have \(\gldimrank\left(\vect^{\Rscr^n}_{\textup{\fpp}}\right) = 2n-2\).
\end{resultx}


When $n=1$, \cref{theorem:gldim-rank-Rn-intro} says that every \fpp module is rank projective.
It follows that every \fpp module decomposes as a direct sum of hook modules, which, in this setting, are interval modules supported on intervals of the form $[i,j)$ in~$\Rscr$.
Hence, when $n=1$ the result recovers the well-known decomposition theorem for \fpp one-parameter persistence modules
(see, e.g., \cite[Section~5.2]{carlsson-zomorodian-collins-guibas}).

\medskip

For computation purposes, one is not only interested in the maximum length of a minimal rank projective resolution, but also in the number of bars in the rank exact decomposition.
Our next result gives an upper bound on the size of the rank exact decomposition of a module, which, for any fixed number of parameters \(n\), is polynomial in the size of the (usual) multigraded Betti numbers of the module.
It also gives a quadratic lower bound.
While the gap between the upper and lower bounds is quite significant, the result is a first step towards establishing how practical the rank exact decomposition is.
Recall that, for a \fpp \(M \colon \Rscr^n \to \vect\), the \(k\)th multigraded Betti numbers of \(M\) are given by \(\beta_k(M) \coloneqq \barc(P_k)\), where \(P_\bullet \to M\) is a (usual) minimal projective resolution of \(M\).

\begin{resultx}[\cref{proposition:polynomial-bound}]
    \label{proposition:polynomial-bound-intro}
    Let $n \geq 1 \in \Nbb$.
    Given \(M \colon \Rscr^n \to \vect\) \fpp, let \(\bettibars(M) = \sum_i |\beta_i(M)|\).
    We have
    \[
        |\barcrank_\sp(M)| + |\barcrank_\sm(M)|\; \in \; O\Big( \bettibars(M)^{e_n}\Big), \text{ where $e_n = {(2n+1)^{2n-1}}$},
    \]
    in the sense that the quantity 
    $|\barcrank_\sp(M)| + |\barcrank_\sm(M)|$ can be bounded above by $\bettibars(M)^{e_n}$ times a constant that only depends on $n$.
    Also, there exists a family of \fpp modules $M : \Rscr^2 \to \vect$ for which $\bettibars(M)$ can be made arbitrarily large, and such that $|\barcrank_\sp(M)| + |\barcrank_\sm(M)| \geq \bettibars(M)^2/48$.
\end{resultx}

In particular, there is no linear upper bound for \(|\barcrank_\sp(M)| + |\barcrank_\sm(M)|\) in terms \(\bettibars(M)\) that holds for all \fpp \(M \colon \Rscr^2 \to \vect\).
We suspect that, when $n=2$, there is a cubic upper bound; see \cref{conjecture:cubic-bound}.

\medskip

Our last result regarding the rank exact decomposition establishes a universality property for \(\widehat{d_B}\), and, as a consequence, a no-go result for distances on signed barcodes.
This result suggests that it may be difficult to find a non-trivial dissimilarity on signed barcodes that
makes the rank exact decomposition stable, that satisfies the triangle inequality,
and whose computation reduces to the computation of a distance between unsigned barcodes.
We first note in \cref{remark:no-triangle-inequality} that \(\widehat{d_B}\) does in general not satisfy the triangle inequality, even when restricted to signed barcodes that are the rank exact decomposition of a \fpp \(\Rscr^n\)-persistence module.
Thus, \(\widehat{d_B}\) is not an extended pseudo distance, but just a dissimilarity.
In order to state the results, we introduce two properties of a dissimilarity on finite signed barcodes \(d\).
We say that the dissimilarity \(d\) is \define{\(\barcrank_{\spm}\)-stable} if it satisfies \(d(\barcrank_\spm(M), \barcrank_\spm(N)) \leq d_I(M,N)\) for all \fpp \(\Rscr^n\)-persistence modules \(M\) and \(N\), and that it is \define{balanced} if it satisfies
\[
    d\big((\Ccal_\sp, \Ccal_\sm \cup \Acal), (\Dcal_\sp , \Dcal_\sm)\big) = d\big((\Ccal_\sp , \Ccal_\sm), (\Dcal_\sp \cup \Acal, \Dcal_\sm)\big),
\]
for all finite signed \(\Rscr^n\)-barcodes \(\Ccal_\spm\) and \(\Dcal_\spm\), and finite (unsigned) \(\Rscr^n\)-barcode \(\Acal\).
The motivation for considering balanced dissimilarities on signed barcodes---besides the fact that \(\widehat{d_B}\) is balanced---is that their computation reduces to the computation of a distance between unsigned barcodes, in the sense that for any balanced dissimilarity \(d\) on finite signed barcodes there exists a dissimilarity \(d'\) on finite (unsigned) barcodes such that \(d\big(\Ccal_\spm, \Dcal_\spm \big) = d'\big(\Ccal_\sp \cup \Dcal_\sm, \Dcal_\sp \cup \Ccal_\sm\big)\) for all finite signed barcodes \(\Ccal_\spm\) and \(\Dcal_\spm\).

\begin{resultx}[\cref{proposition:universality,prop:no-stable-tri}]
    \label{proposition:universality-intro}
    Let $n \geq 1 \in \Nbb$.
    The collection of \(\barcrank_{\spm}\)-stable and balanced dissimilarity functions on finite signed \(\Rscr^n\)-barcodes has a maximum with respect to the pointwise order, denoted \(d_\msf\).
    When restricted to signed barcodes containing only hook modules, $\widehat{d_B}$ and $d_\msf$ satisfy the following bi-Lipschitz equivalence:
    \[
        \widehat{d_B}/(2n-1)^2 \leq d_\msf \leq \widehat{d_B}.
    \]
    In particular, if \(d\) is a \(\barcrank_{\spm}\)-stable and balanced dissimilarity function on finite signed \(\Rscr^n\)-barcodes that satisfies the triangle inequality, then, for all \fpp \(M,N \colon \Rscr^n \to \vect\) such that \(d_I(M,N) < \infty\), we have \(d(\barcrank_\spm(M), \barcrank_\spm(N)) = 0\).
\end{resultx}

\medskip

We conclude the paper with a discussion about other exact structures on categories of persistence modules.
In particular, using an approach similar to the one used to prove \cref{theorem:gldim-rank-Rn-intro}, we compute the global dimension of the limit exact structure on persistence modules over a finite, \(2\)-dimensional grid, and on the category of \fpp two-parameter persistence modules.
When it exists, the \define{limit exact structure} on a category of functors \(\Pscr \to \vect\), denoted \(\Eupset\), is defined to be the exact structure whose exact sequences are the sequences \(0 \to L \to M \to N \to 0\) such that, for all upward closed subsets \(U \subseteq \Pscr\), the sequence of vector spaces \(0 \to \varprojlim_U L \to \varprojlim_U M \to \varprojlim_U N \to 0\) is exact.

\begin{resultx}[\cref{theorem:global-dimension-upset,corollary:infty-dim-upset-exact}]
    \label{theorem:global-dimension-upset-intro}
    We have \(\gldimupset\left(\vect^{[m]^2}\right) = m-2\), where \(m \geq 3\) and \([m]^2\) is the poset given by an \(m\)-by-\(m\) \(2\)-dimensional grid.
    As a consequence, we have \(\gldimupset\left(\vect^{\Rscr_{\geq 0}^2}_{\textup{\fpp}}\right) = \infty\).
\end{resultx}

In the result above, the poset $\Rscr_{\geq 0}^n$ is used instead of $\Rscr^n$ so that there are enough relative projectives; see \cref{sec:other-exact-structures} for details.
We also discuss in \cref{remark:connection-miller} the relationship between \(\Eupset\)-projective resolutions and the upset resolutions of \cite{miller}.

\medskip


\subsection{Acknowledgements}
L.S.~thanks Ignacio Darago and Pedro Tamaroff for inspiring conversations about homological algebra.
We thank H\aa{}vard Bjerkevik for suggesting several improvements to \cref{matchings-section}, and the reviewers for many comments that have improved this manuscript.
L.S.~was partially supported by the National Science Foundation through grant CCF-2006661
and CAREER award DMS-1943758, as well as by EPSRC grant ``New Approaches to Data Science: Application Driven Topological Data Analysis'', EP/R018472/1.
For the purpose of Open Access, the authors have applied a CC BY public copyright licence to any Author Accepted Manuscript (AAM) version arising from this submission.

\section{Background and notation}
\label{section:background}
We assume familiarity with basic category theory and homological algebra.
Throughout the paper, we fix a field \(\kbb\) and let \(\vect\) denote the category of finite dimensional \(\kbb\)-vector spaces.
We also fix a natural number \(n \geq 1 \in \Nbb\).

\subsection{Distances and dissimilarities}
A \define{dissimilarity} on a set or proper class \(A\) is a function \(d \colon A \times A \to \Rbb \cup \{\infty\}\) that is non-negative and satisfies \(d(x,y) = d(y,x)\) for all \(x,y \in A\), and \(d(x,x) = 0\) for all \(x \in A\).
An \define{extended pseudo distance} on \(A\) is a dissimilarity that satisfies \(d(x,z) \leq d(x,y) + d(y,z)\) for all \(x,y,z \in A\).

\subsection{Posets}
We consider the poset of real numbers \(\Rscr\) with its standard total order, and the product poset \(\Rscr^n\).
For \(m \geq 1 \in \Nbb\), we also consider the finite poset \([m] = \{0, \dots, m-1\}\) with its standard order, and the product poset \([m]^n\).
We write \(0 = (0,\dots,0) \in \Rscr^n\) and \(0 = (0,\dots,0) \in [m]^n\).

\subsection{Persistence modules}
Let \(\Pscr\) be a poset.
By an abuse of notation we also denote by \(\Pscr\) the category that has as objects the elements of \(\Pscr\), and exactly one morphism from \(i\) to \(j\) whenever \(i \leq j \in \Pscr\).
A \define{pointwise finite dimensional \(\Pscr\)-persistence module} is a functor \(\Pscr \to \vect\).
For simplicity, we refer to pointwise finite dimensional \(\Pscr\)-persistence modules simply as \define{\(\Pscr\)-persistence modules}.

Let \(M \colon \Pscr \to \vect\).
For \(a \leq b \in \Pscr\) we let \(\phi^M_{a,b} \colon M(a) \to M(b)\) denote the corresponding structure morphism.

\subsection{Persistence modules as modules}
Given a poset \(\Pscr\), we let \(\kbb \Pscr\) denote the \(\kbb\)-algebra that is freely generated as a \(\kbb\)-vector space by pairs \([i,j] \in \Pscr \times \Pscr\) with \(i \leq j\), and with multiplication given by linearly extending the rule
\[
    [i,j] \cdot [k,l] =
    \begin{cases}
        [i,l], & \text{if \(j=k\),}  \\
        0,     & \text{otherwise.}
    \end{cases}
\]

If \(\Lambda\) is a \(\kbb\)-algebra, we denote the category of right \(\Lambda\)-modules by $\Modcat_{\Lambda}$ and the category of finite dimensional \(\Lambda\)-modules by \(\modcat_{\Lambda}\).
A proof of the next result is in \cref{section:appendix}.



\begin{restatable}{lemma}{persistencemodulesaremodules}
    \label{lemma:persistence-modules-are-modules}
    There is an additive, fully faithful functor \(\Vect^\Pscr \to \Modcat_{\kbb \Pscr}\).
    The functor is given by mapping \(M \colon \Pscr \to \Vect\) to the module with underlying vector space \(\bigoplus_{i \in \Pscr} M(i)\), and with action given by linearly extending the rule \(\left(\sum_{i \in \Pscr} m_i\right) \cdot [j,k] = \phi^M_{j,k}(m_j)\), where \(m_i \in M(i) \subseteq \bigoplus_{i \in \Pscr} M(i)\).
    If \(\Pscr\) is finite, then this functor restricts to an equivalence of categories \(\vect^\Pscr \simeq \modcat_{\kbb \Pscr}\).
\end{restatable}

In view of \cref{lemma:persistence-modules-are-modules}, we abuse notation and treat \(\Pscr\)-persistence modules as functors \(\Pscr \to \vect\) or as right \(\kbb \Pscr\)-modules, particularly in the context of finite posets.

\begin{remark}
    If \(\Pscr\) is a finite poset, then \(\kbb \Pscr\) is an Artin $\kbb$-algebra (i.e., a $\kbb$-algebra of finite dimension as a $\kbb$-vector space) and is in fact isomorphic to the incidence \(\kbb\)-algebra of the poset \(\Pscr\), as defined in \cite{doubilet-rota-stanley}.
    The fact that $\kbb \Pscr$ is an Artin algebra lets us rely on the well-developed theory of representations of Artin algebras, and in particular on the relative homological algebra developed in \cite{auslander-solberg}.
\end{remark}

\subsection{Interleavings}
\label{section:interleavings}
We recall the notion of interleaving between \(\Rscr^n\)-persistence modules of \cite{lesnick,lesnick-thesis}.
Let \(M \colon \Rscr^n \to \vect\) and \(\epsilon \geq 0 \in \Rscr^n\).
The \define{\(\epsilon\)-shift} of \(M\), denoted \(M[\epsilon] \colon \Rscr^n \to \vect\), is such that \(M[\epsilon](r) = M(r + \epsilon)\) and has structure morphism corresponding to \(r \leq s \in \Rscr^n\) given by \(\phi_{r+\epsilon, s+\epsilon}^M\).
Shifting gives a functor \((-)[\epsilon] \colon \vect^{\Rscr^n} \to \vect^{\Rscr^n}\), and, for every persistence module \(M\), there is a natural transformation \(\eta^M_\epsilon \colon M \to M[\epsilon]\) with \(r\)-component given by \(\phi^M_{r,r+\epsilon} \colon M(r) \to M(r+\epsilon)\).
For \(M,N \colon \Rscr^n \to \vect\) and \(\epsilon, \delta \geq 0 \in \Rscr^n\), an \define{\((\epsilon;\delta)\)-interleaving} between \(M\) and \(N\) consists of morphisms \(f \colon M \to N[\epsilon]\) and \(g \colon N \to M[\delta]\) such that \(g[\epsilon] \circ f = \eta^M_{\epsilon+\delta}\) and \(f[\delta] \circ g = \eta^N_{\delta+\epsilon}\).
When \(\epsilon \geq 0 \in \Rscr\), an \define{\(\epsilon\)-interleaving} is defined to be an \((\epsilon, \dots, \epsilon; \epsilon, \dots, \epsilon)\)-interleaving.
The \define{interleaving distance} between \(M, N \colon \Rscr^n \to \vect\) is
\[
    d_I(M,N) = \inf\left( \{ \; \epsilon \geq 0 \in \Rscr \; \colon \; \text{there exists an \(\epsilon\)-interleaving between \(M\) and \(N\) } \} \cup \{\infty\}\right).
\]
By composing interleavings, one shows that the interleaving distance is an extended pseudo distance on the collection of \(\Rscr^n\)-persistence modules.

Let \(\epsilon \geq 0 \in \Rscr\).
A persistence module \( M \colon \Rscr^n \to \vect \) is \define{\(\epsilon\)-trivial} if there exists an \(\epsilon/2\)-interleaving between \(M\) and the zero module.
If \(M\) is not \(\epsilon\)-trivial, we say that \(M\) is \define{\(\epsilon\)-significant}; this is equivalent to the existence of \(r \in \Rscr^n\) such that \(\phi_{r,r+\epsilon}^M\) is not the zero morphism.

\subsection{Interval modules}
\label{section:interval-modules}
Let \(\Pscr\) be a poset.
An \define{interval} of \(\Pscr\) consists of a non-empty subset \(\Ical \subseteq \Pscr\) satisfying the following two properties: if \(a,b \in \Ical\), then \(c \in \Ical\) for all \(a \leq c \leq b\); and for all \(a,b \in \Ical\), there exists a finite sequence \(c_1, \dots, c_\ell \in \Ical\) such that \(a = c_1\), \(b = c_\ell\), and \(c_k\) and \(c_{k+1}\) are comparable for all \(1 \leq k \leq \ell-1\).

If \(\Ical \subseteq \Pscr\) is an interval, the \define{interval module} with support \(\Ical\) is the persistence module \(\kbb_\Ical \colon \Pscr \to \vect\) that takes the value \(\kbb\) on the elements of \(\Ical\) and the value \(0\) elsewhere, and is such that all the morphisms that are not constrained to be zero are the identity of \(\kbb\).

The \emph{indecomposable projective} corresponding to \(i \in \Pscr\), denoted \(\Psf_i \colon \Pscr \to \vect\), is the interval module with support \(\{j \in \Pscr \colon i \leq j\}\).
The name comes from the fact that any indecomposable projective module $\Pscr \to \mathrm{Vec}$ is isomorphic to $\Psf_i$ for some $i \in \Pscr$.
The \emph{simple module} corresponding to \(i \in \Pscr\), denoted \(\Ssf_i \colon \Pscr \to \vect\), is the interval module with support \(\{i\}\).
The name comes from the fact that any module $\Pscr \to \vect$ with no non-zero proper submodules is isomorphic to $\Ssf_i$ for some $i \in \Pscr$.
If \(i < j \in \Pscr \cup \{\infty\}\), we let \(\Lsf_{i,j} \colon \Pscr \to \vect\) denote the interval module with support \(\{k \in \Pscr \colon i \leq k \text{ and } j \nleq k\}\).
We refer to modules of the form \(\Lsf_{i,j}\) as \define{hook modules}, and to modules which are isomorphic to a direct sum of hook modules as \define{hook-decomposable} modules.
We choose the notation \(\Lsf_{i,j}\) and the name ``hook module'' since they are reminiscent of the shape of the support of these modules in the case \(\Pscr = \Rscr^2\).
Note that, for every \(i \in \Pscr\), we have \(\Psf_i = \Lsf_{i,\infty}\).
If \(i \leq j \in \Pscr\), we let \(\Rsf_{i,j} \colon \Pscr \to \vect\) denote the interval module with support \(\{k \in \Pscr \colon i \leq k \leq j\}\).
We refer to modules of the form \(\Rsf_{i,j}\) as \define{rectangle modules}, taking inspiration from their shape when the indexing poset is \(\Rscr^2\) or a \(2\)-dimensional grid.

\subsection{Multisets}
\label{section:multisets}
Let \(A\) be a set.
A finite multiset of elements of \(A\) is often defined to be a function \(\alpha \colon A \to \Nbb\), thought of as a multiplicity function, with the property that \(\{a \in A \colon \alpha(a) > 0\}\) is finite.
Using this definition, equality of multisets is simply equality of functions, but other relationships between multisets, such as functions or partial matchings between them, are harder to manipulate.
For this reason, we use an equivalent notion of multiset that is more convenient for our purposes.

An \define{indexed multiset of elements of \(A\)} consists of a pair \((I,f)\), with \(I\) a set, called the \define{indexing set}, and \(f\) a function \(I \to A\), called the \define{indexing function}.
An indexed multiset of elements of \(A\) is \define{finite} if its indexing set is finite.
Fix \((I,f)\) and \((J,g)\) indexed multisets of elements of \(A\).
A \define{bijection} \(h \colon (I,f) \to (J,g)\) is simply a bijection \(h \colon I \to J\) between the indexing sets, and an \define{isomorphism} between \((I,f)\) and \((J,g)\) is a bijection \(h \colon (I,f) \to (J,g)\) such that \(f = g \circ h\).
When \((I,f)\) and \((J,g)\) are isomorphic, we write \((I,f) = (J,g)\).
The union \((I,f) \cup (J,g)\) is the indexed multiset of elements of \(A\) given by \((I \amalg J, f \amalg g)\).
We say that \((I,f)\) and \((J,g)\) are \define{disjoint} if the images of \(f\) and \(g\) are disjoint.
The \define{multiplicity} of $a \in A$ in \((I,f)\) is the cardinality of the preimage \(f^{-1}(a) \subseteq I\).
The \emph{cardinality} of a multiset $(I,f)$, denoted $|I|$ when the indexing function is omitted, is simply the cardinality of the indexing set.

When there is no risk of confusion, we refer to indexed multisets simply as multisets, leave the indexing function \(f\) of a multiset \((I,f)\) implicit, and do not distinguish an element \(i \in I\) from \(f(i) \in A\).
When the indexing is relevant, we often write \(\{a_i\}_{i \in I}\) for an indexed multiset of elements of \(A\).

\subsection{Barcodes and signed barcodes}
\label{subsection:barcodes}
Let \(\Pscr\) be a poset.
A \define{\(\Pscr\)-barcode} consists of a multiset of isomorphism classes of indecomposable \(\Pscr\)-persistence modules.
A \define{signed \(\Pscr\)-barcode} consists of an ordered pair of \(\Pscr\)-barcodes.
When there is no risk of confusion, we keep the poset \(\Pscr\) implicit and refer to \(\Pscr\)-barcodes as barcodes and to signed \(\Pscr\)-barcodes as signed barcodes.

We denote the isomorphism class of a persistence module \(M\) by \([M]\).
By \cite[Theorem~1]{botnan-crawleybovey} and the Krull--Remak--Schmidt--Azumaya theorem \cite{azumaya}, any \(\Pscr\)-persistence module \(M\) decomposes as a direct sum of indecomposables \(M \cong \bigoplus_{i \in I} M_i\) for a unique (up to isomorphism of multisets) multiset of isomorphism classes of indecomposables \(\{[M]_i\}_{i \in I}\); we denote \(\barc(M) = \{ [M_i] \}_{i \in I}\), which is independent of the choice of decomposition of \(M\) as direct sum of indecomposables.

\subsection{Matchings, bottleneck distance, and signed bottleneck dissimilarity}
\label{section:matchings-bottleneck-distance-bottleneck-dissimilarity}

A \define{partial matching} \(h \colon (I,f) \nrightarrow (J,g)\) between multisets \((I,f)\) and \((J,g)\) consists of subsets \(\coim(h) \subseteq I\) and \(\im(h) \subseteq J\) and a bijection \(h \colon \coim(h) \to \im(h)\).
We often refer to a partial matching \(h \colon (I,f) \nrightarrow (J,g)\) simply as a matching.
Let \(\Ccal = \{[M_i]\}_{i \in I}\) and \(\Dcal = \{[N_j]\}_{j \in J}\) be \(\Rscr^n\)-barcodes and let \(\epsilon \geq 0 \in \Rscr\).
An \define{\(\epsilon\)-matching} between \(\Ccal\) and \(\Dcal\) is a matching \(h \colon \Ccal \nrightarrow \Dcal\) such that
\begin{itemize}
    \item for all \(i \in \coim(h)\), the modules \(M_i\) and \(N_{h(i)}\) are \(\epsilon\)-interleaved,
    \item for all \(i \in I \setminus \coim(h)\), the module \(M_i\) is \(\epsilon\)-interleaved with the zero module,
    \item for all \(j \in J \setminus \im(h)\), the module \(N_j\) is \(\epsilon\)-interleaved with the zero module.
\end{itemize}

Although the interleaving distance between indecomposable modules may be hard to compute in general, in \cref{lemma:formulas-interleavings} we give an explicit formula for computing the interleaving distance between hook modules, which is of particular relevance to this work.

The \define{bottleneck distance} between \(\Rscr^n\)-barcodes \(\Ccal\) and \(\Dcal\) is
\[
    d_B(\Ccal,\Dcal) = \inf\left( \{ \; \epsilon \geq 0 \in \Rscr \; \colon \; \text{there exists an \(\epsilon\)-matching between \(\Ccal\) and \(\Dcal\) } \} \cup \{\infty\}\right).
\]
By composing matchings, one shows that the bottleneck distance is an extended pseudo distance on the collection of \(\Rscr^n\)-barcodes.
Note that, for \(M,N \colon \Rscr^n \to \vect\), we always have \(d_I(M,N) \leq d_B(\barc(M),\barc(N))\).

The \define{signed bottleneck dissimilarity} between signed \(\Rscr^n\)-barcodes \(\Ccal_\spm\) and \(\Dcal_\spm\) is
\[
    \widehat{d_B}(\Ccal_\spm, \Dcal_\spm) = d_B(\Ccal_\sp \cup \Dcal_\sm, \Dcal_\sp \cup \Ccal_\sm).
\]

\subsection{Rank decompositions}
\label{section:rank-decs}

Let \(\Pscr\) be a poset.
The \define{rank invariant} \cite{carlsson-zomorodian} of a module \(M \colon \Pscr \to \vect\) is the function \(\rk(M) \colon \{(a,b) \in \Pscr \times \Pscr \colon a \leq b\} \to \Zbb\) where \(\rk(M)(a,b)\) is the rank of the structure morphism \(\phi^M_{a,b} \colon M(a) \to M(b)\).
A \define{rank decomposition} of a \(\Pscr\)-persistence module \(M\) is a signed barcode \(\Ccal_\spm\) such that \(\rk(M) = \sum_{A \in \Ccal_\sp} \rk(A) - \sum_{B \in \Ccal_\sm} \rk(B)\) as \(\Zbb\)-valued functions.
We recall the existence of the minimal rank decompositions by rectangles and by hooks from \cite{botnan-oppermann-oudot}.

A \define{right open rectangle module} over \(\Rscr^n\) is an interval \(\Rscr^n\)-persistence module with support \([a_1, b_1) \times \dots \times [a_n, b_n) \subseteq \Rscr^n\) for some \(a < b \in (\Rscr \cup \{\infty\})^n\), with \(a_k < b_k\) for all \(1 \leq k \leq n\).
We denote such an interval module by \(\Rsf^o_{a,b}\).

\begin{theorem}[{\cite[Corollary~5.6]{botnan-oppermann-oudot}}]
    Let \(M \colon \Rscr^n \to \vect\) be \fpp.
    There exists a unique signed \(\Rscr^n\)-barcode \(\mrd_\spm(M) = (\mrd_\sp(M), \mrd_\sm(M))\) with the following properties:
    \begin{itemize}
        \item the (isomorphism classes of) modules appearing in \(\mrd_\sp(M)\) and \(\mrd_\sm(M)\) are right open rectangle modules;
        \item the multisets \(\mrd_\sp(M)\) and \(\mrd_\sm(M)\) are disjoint;
        \item \(\mrd_\spm(M)\) is a rank decomposition of \(M\).
    \end{itemize}
    Similarly, there exists a unique signed \(\Rscr^n\)-barcode \(\mrdh_\spm(M) = (\mrdh_\sp(M), \mrdh_\sm(M))\) with the following properties:
    \begin{itemize}
        \item the (isomorphism classes of) modules appearing in \(\mrdh_\sp(M)\) and \(\mrdh_\sm(M)\) are hook modules;
        \item the multisets \(\mrdh_\sp(M)\) and \(\mrdh_\sm(M)\) are disjoint;
        \item \(\mrdh_\spm(M)\) is a rank decomposition of \(M\).
    \end{itemize}
\end{theorem}
If \(M \colon \Rscr^n \to \vect\) is \fpp, we refer to \(\mrd_\spm(M)\) and \(\mrdh_\spm(M)\) as the \define{minimal rank decomposition by rectangles} of \(M\) and the \define{minimal rank decomposition by hooks} of \(M\), respectively.

It is worth mentioning that, when \(n=1\), we have \(\mrd_\spm(M) = \mrdh_\spm(M)\), with the negative part being empty, and with the positive part being equal to the standard one-parameter barcode.

\subsection{Basic homological algebra of persistence modules}

Let \(\Pscr\) be a poset and let \(M \colon \Pscr \to \vect\).
We say that \(M\) is \define{finitely presentable} (abbreviated \fpp) if it is isomorphic to the cokernel of a morphism \(P \to Q\) such that \(P\) and \(Q\) are isomorphic to finite direct sums of indecomposable projective modules (i.e., modules of the form \(\Psf_i\));
recall that these modules are defined in \cref{section:interval-modules}.
The subcategory of \(\vect^\Pscr\) spanned by \fpp \(\Pscr\)-persistence modules is denoted by \(\vect^\Pscr_\fpp\).
Note that \(\vect^\Pscr_\fpp\) is an Abelian category when $\Pscr$ is either finite poset or $\Rscr^n$, which are the cases that are relevant to this paper.
If \(\Pscr\) is finite, then every \(\Pscr\)-persistence module is finitely presentable, but this need not be the case if \(\Pscr\) is infinite.
The following is well known (see, e.g., \cite[Chapter~9,~Corollary~7.3]{mitchell}).

\begin{lemma}
    \label{lemma:proj-and-inj-modules}
    Let \(\Pscr\) be a finite poset.
    A module \(M \in \modcat_{\kbb \Pscr}\) is \(\kbb \Pscr\)-projective if and only if \(M\) is isomorphic to a direct sum \(\bigoplus_{i \in I} \Psf_i\) for some finite multiset \(I\) of elements of \(\Pscr\).
    If \(\Pscr\) has a minimum \(0 \in \Pscr\), a module \(M \in \modcat_{\kbb \Pscr}\) is \(\kbb \Pscr\)-injective if and only if it is isomorphic to a direct sum \(\bigoplus_{i \in I} \Rsf_{0,i}\) for some finite multiset \(I\) of elements of \(\Pscr\).
    \qed
\end{lemma}

The following is well known and follows from, e.g., \cite[Chapter~9,~Corollary~10.3]{mitchell}.

\begin{lemma}
    \label{lemma:finite-poset-finite-resolution}
    If \(\Pscr\) is a finite poset, then every module \(M \in \modcat_{\kbb \Pscr}\) admits a finite \(\kbb \Pscr\)-projective resolution and a finite \(\kbb \Pscr\)-injective resolution.\qed
\end{lemma}

\subsection{Exact structures}
\label{section:exact-structures}
We start with a standard definition to fix notation.
A sequence \(0 \to A \xrightarrow{\;f\;} B \xrightarrow{\;g\;} C \to 0\) in an Abelian category is \define{exact} if \(f\) is the kernel of \(g\) and \(g\) is the cokernel of \(f\).
We refer to exact sequences of the form \(0 \to A \to B \to C \to 0\) as \define{short exact} sequences.

Let \(\Ecal\) be a collection of short exact sequences of an Abelian category \(\Cscr\).
If \(0 \to A \xrightarrow{\;i\;} B \xrightarrow{\;d\;} C \to 0\) is in \(\Ecal\), we say that \(i\) is an \define{inflation} and that \(d\) is a \define{deflation}.
The collection \(\Ecal\) is an \define{exact structure} on \(\Cscr\) if it satisfies the following properties:
\begin{itemize}
    \item \(\Ecal\) is closed under isomorphisms of short exact sequences;
    \item the identity morphism on the zero object is a deflation;
    \item the composition of two deflations is a deflation;
    \item the pullback of a deflation along any morphism is a deflation;
    \item the pushout of an inflation along any morphism is an inflation.
\end{itemize}

\begin{example}
The collection of all short exact sequences of an Abelian category \(\Cscr\) forms an exact structure on \(\Cscr\).
\end{example}

For a thorough introduction to the theory of exact structures as well as some context and history, we refer the reader to \cite{draxler-reiten-smalo-solberg}.
Note that the notion of exact structure is self-dual.

Let \(\Ecal\) be an exact structure on an Abelian category \(\Cscr\).
A short exact sequence of \(\Cscr\) is \define{\(\Ecal\)-exact} if it belongs to the exact structure \(\Ecal\).
A sequence
\[
    \cdots \to A_{k+1} \xrightarrow{f_{k+1}} A_{k} \xrightarrow{f_k} A_{k-1} \to \cdots \to A_0 \to 0.
\]
is \define{\(\Ecal\)-exact} if it is exact and for every \(k \in \Nbb\) we have that the short exact sequence \(0 \to \ker(f_{k}) \to A_k \to \coker(f_{k+1}) \to 0\) is \(\Ecal\)-exact.
An object \(X \in \Cscr\) is \define{\(\Ecal\)-projective} if \(\Hom_\Cscr(X,-) \colon \Cscr \to \Ab\) maps every element of \(\Ecal\) to a short exact sequence of Abelian groups.
The notion of \define{\(\Ecal\)-injective} is dual.
An \define{\(\Ecal\)-projective resolution} of \(Y \in \Cscr\) consists of an \(\Ecal\)-exact sequence \( \cdots \to X_{k+1} \to X_k \to \cdots \to X_0 \to Y \to 0\) where \(X_k\) is \(\Ecal\)-projective for all \(k \in \Nbb\).
The notion of \define{\(\Ecal\)-injective resolution} is dual.
The \emph{length} of an $\Ecal$-projective resolution $X_\bullet \to Y$ is the maximum over $i \in \Nbb$ such that $X_i \not\cong 0$, with the convention that the length is $\infty$ if no such $i$ exists.

An \emph{$\Ecal$-projective cover} of an object $Y \in \Cscr$ is a deflation $p : X \to Y$ with $X$ an $\Ecal$-projective object, and such that any endomorphism $q : X \to X$ such that $p = p \circ q$ is necessarily an isomorphism.
An $\Ecal$-projective resolution $X_\bullet \to Y$ is \emph{minimal} if $X_0 \to Y$ and $X_k \to \ker(X_{k-1} \to X_{k-2})$ for all $k \geq 1$ are $\Ecal$-projective covers (here, by convention $X_{-1} = Y$).
It is straightforward to see that, if $p : X \to Y$ and $p' : X' \to Y$ are $\Ecal$-projective covers, then there exists an isomorphism $r : X \to X'$ such that $p = p' \circ r$.
An easy inductive argument then shows that, if $Y$ admits a minimal $\Ecal$-projective resolution, then any other minimal $\Ecal$-projective resolution is isomorphic to it.

The \emph{$\Ecal$-projective dimension} of $Y \in \Ccal$, denoted $\pdim^\Ecal(Y)$, is the infimum, over all $\Ecal$-projective resolutions, of the length of the resolution;
note that this number could be infinity.
It is easy to verify that, if $Y$ admits a minimal $\Ecal$-projective resolution, then the length of this resolution equals the $\Ecal$-projective dimension of $Y$.

The \emph{global dimension} of $\Ecal$, denoted $\gldim^\Ecal(\Ccal)$, is the supremum, over all $Y \in \Ccal$, of the $\Ecal$-projective dimension of $Y$.

%

\subsection{Signed barcodes from exact structures}
Let \(\Pscr\) be a poset and let \(\Ecal\) be an exact structure on the category of \(\Pscr\)-persistence modules.
Assume that \(M \colon \Pscr \to \vect\) admits a minimal \(\Ecal\)-projective resolution \(P_\bullet \to M\).
In that case, we define the \(k\)th \define{\(\Ecal\)-relative Betti numbers} of \(M\) by \(\beta^\Ecal_k(M) \coloneqq \barc(P_k)\), which are independent of the particular choice of minimal resolution for \(M\) by the uniqueness up to isomorphism of minimal resolutions.
We also define the \define{\(\Ecal\)-signed barcode} of \(M\) as the signed barcode
\[
    \barcEcal_\spm(M) \coloneqq \left(\;\bigcup_{k \text{ even}} \beta_k^\Ecal(M)\;,\; \bigcup_{k \text{ odd}} \beta_k^\Ecal(M)\;\right).
\]

The prototypical example comes from the usual exact structure on the category of \fpp \(\Pscr\)-persistence modules, which has as exact sequences the usual exact sequences.
The \(k\)th Betti numbers of a \fpp \(\Pscr\)-persistence module \(M\) relative to the usual exact structure are called the \(k\)th \define{(multigraded) Betti numbers} of \(M\), and are denoted by \(\beta_k(M)\).

\subsection{The rank exact structure}
Let \(\Pscr\) be a poset.
A short exact sequence \(0 \to A \to B \to C \to 0\) in \(\vect^\Pscr\) is \define{rank exact} \cite{botnan-oppermann-oudot} if \(\rk(B) = \rk(A) + \rk(C)\).
When it exists, the exact structure whose exact sequences are the rank exact sequences is called the \emph{rank exact structure}, and is denoted by \(\Erank\).

\begin{restatable}[{cf.~\cite{botnan-oppermann-oudot}}]{theorem}{rankexactstructurefacts}
    \label{theorem:rank-exact-structure-facts}
    Let $n \in \Nbb$ and let $\Pscr$ be a finite poset.
    \begin{enumerate}
        \item Let $\Mcal$ be either \(\vect_{\fpp}^{\Rscr^n}\) or $\vect^\Pscr$.
        The collection of rank exact sequences forms an exact structure on the category $\Mcal$.
        The projectives of this exact structure are the hook-decomposable modules.
        Every module of $\Mcal$ admits a finite minimal rank projective resolution.
        \item Let \(M \colon \Rscr^n \to \vect\) be \fpp.
            A minimal rank projective resolution of \(M\) can be obtained as the left Kan extension of a minimal rank projective resolution of a module \(M' \colon [m_1] \times \cdots \times [m_n] \to \vect\) along an injective monotonic map \(f \colon [m_1] \times \cdots \times [m_n] \to \Rscr^n\) with image the finite lattice generated by the grades involved in a minimal projective presentation of $M$.
    \end{enumerate}
\end{restatable}
Since \cref{theorem:rank-exact-structure-facts} is not stated in exactly that way in \cite{botnan-oppermann-oudot}, we include a proof of the result in \cref{section:appendix}, for completeness.

By \cref{theorem:rank-exact-structure-facts}, any \fpp \(\Rscr^n\)-persistence module admits Betti numbers relative to the rank exact structure.
We can thus instantiate the definition of relative signed barcode to the rank exact structure on \fpp persistence modules, which we now do.
The \define{rank exact decomposition} of a \fpp persistence module \(M \colon \Rscr^n \to \vect\) is the signed barcode
\[
    \barcrank_\spm(M) = \left(\; \bigcup_{k \text{ even}} \bettirank_k(M)\;,\; \bigcup_{k \text{ odd}} \bettirank_k(M)\;\right).
\]

\begin{example}
    \label{example:rank-exact-resolution}
    For an illustration of this example, see \cref{figure:example-resolution}.
    Let $M \coloneqq \kbb_{\Ical} : \Rscr^2 \to \vect$, with $\Ical = \big([0,2)\times [0,1)\big) \cup \big([0,1)\times [0,2)\big)$.
    Then, the following is a minimal rank projective resolution of $M$:
    \begin{equation}
        \label{equation:minimal-rank-exact-resolution}
        0 \to \Lsf_{(0,0),(2,1)} \oplus \Lsf_{(0,0),(1,2)}
            \xrightarrow{\scriptsize\begin{pmatrix}0 & 1 \\ 1 & -1 \\ -1 & 0\end{pmatrix}}
        \Lsf_{(0,0),(0,2)} \oplus \Lsf_{(0,0),(1,1)}\oplus \Lsf_{(0,0),(2,0)}
            \xrightarrow{(1,1,1)}
        M
    \end{equation}
    Here, we are letting $1$ denote the morphism between interval modules $\kbb_\Ical \to \kbb_\Jcal$ with component $\mathsf{id} : \kbb \to \kbb$ in the intersection $\Ical \cap \Jcal$ and zero otherwise, when this morphism is well defined.

    To see that \cref{equation:minimal-rank-exact-resolution}
    is indeed a minimal rank projective resolution of $M$, one can proceed as follows.
    First, one checks that the sequence is exact, which is easily done since it is enough to do it pointwise at each element of $\Rbb^2$, and it is enough to do it for points on the grid $\{0,1,2\}^2 \subseteq \Rscr^2$.
    Then, one checks that any morphism $\Lsf_{a,b} \to M$ factors through
    $\Lsf_{(0,0),(0,2)} \oplus \Lsf_{(0,0),(1,1)}\oplus \Lsf_{(0,0),(2,0)} \to M$.
    This is a simple case analysis: if $0 \leq b_y < 1$, the morphism factors through
    $1 : \Lsf_{(0,0),(0,2)} \to M$;
    if $1 \leq b_y < 2$, it factors through
    $1 : \Lsf_{(0,0),(1,1)} \to M$;
    and if $2 \leq b_y$, it factors through
    $1 : \Lsf_{(0,0),(2,0)} \to M$.

    The two considerations above imply that \cref{equation:minimal-rank-exact-resolution} is a rank exact sequence.
    So, to conclude, we need to prove that the right-most morphism is a rank projective cover, and for this we need to prove that any endomorphism of
    $X = \Lsf_{(0,0),(0,2)} \oplus \Lsf_{(0,0),(1,1)}\oplus \Lsf_{(0,0),(2,0)}$
    commuting with the right-most morphism is necessarily an isomorphism.
    Note that, when written in matrix form, any endomorphism of $X$ is necessarily diagonal since there are no non-zero morphism between distinct summands of $X$.
    This forces any endomorphism commuting with the right-most morphism to be the identity, concluding the proof.
\end{example}

\subsection{Stability of relative projective resolutions}
\label{stability-projective-resolutions-section}

An exact structure \(\Ecal\) on \(\vect^{\Rscr^n}\) is \define{shift-invariant} if, for all \(\epsilon \geq 0 \in \Rscr\) and \(\Ecal\)-exact sequence \(0 \to A \to B \to C \to 0\), we have that \(0 \to A[\epsilon] \to B[\epsilon] \to C[\epsilon] \to 0\) is \(\Ecal\)-exact.
The following result is a straightforward generalization of the stability of projective resolutions proven in \cite[Section~3]{oudot-scoccola}; we include its proof in \cref{section:appendix}, for completeness.

\begin{restatable}{proposition}{relativestabilityresolutions}
    \label{lemma:relative-stability-resolutions}
    Let \(\Ecal\) be a shift-invariant exact structure on \(\vect^{\Rscr^n}\).
    Let \(M,N \colon \Rscr^n \to \vect\), and let \(P_\bullet \to M\) and \(Q_\bullet \to N\) be \(\Ecal\)-projective resolutions of finite length.
    If \(M\) and \(N\) are \(\epsilon\)-interleaved, then so are
    \[
        P_0 \oplus Q_1[\epsilon] \oplus P_2[2\epsilon] \oplus Q_3[3\epsilon] \oplus \cdots \;\;\text{and}\;\; Q_0 \oplus P_1[\epsilon] \oplus Q_2[2\epsilon] \oplus P_3[3\epsilon] \oplus \cdots
    \]
\end{restatable}

\section{Bottleneck instability of minimal rank decompositions}

In this section, we prove \cref{proposition:instability-mrd-rectangles}, the instability result for minimal rank decompositions by rectangles in the signed bottleneck dissimilarity.
We also explain in \cref{remark:instability-mrdh} how a similar example shows that minimal rank decompositions by hooks are also not stable in the signed bottleneck dissimilarity.

\begin{proposition}
    \label{proposition:instability-mrd-rectangles}
    Let \(n \geq 2 \in \Nbb\).
    There is no function \(f \colon \Rbb_{\geq 0} \to \Rbb_{\geq 0}\) with \(f(r) \xrightarrow{r \to 0} 0\) such that the following inequality holds for all \fpp modules \(M,N \colon \Rscr^n \to \vect\) at finite interleaving distance:
    \[
        \widehat{d_B}(\,\mrd_\spm(M)\,,\,\mrd_\spm(N)\,) \leq f\left(d_I(M,N)\right).
    \]
\end{proposition}

\begin{proof}
    We prove the claim in the case \(n=2\).
    Examples in dimension \(n\geq 3\) can be constructed by taking the external tensor product of each of the modules in the example we give with \(\Psf_0 \colon \Rscr^{n-2} \to \vect\), where the external tensor product of \(X \colon \Rscr^k \to \vect\) and \(Y \colon \Rscr^\ell \to \vect\) is the persistence module \(X\widehat{\otimes} Y \colon \Rscr^{k + \ell} \to \vect\) with \((X\widehat{\otimes}Y)(r) = X(r_1, \dots,r_k) \otimes Y(r_{k+1}, \dots, r_{k+\ell})\).

    Let \(b = 10\) and \(a = 2\).
    Let \(H_{a,b}, V_{a,b}, T_{a} \colon \Rscr^2 \to \vect\) be the right open rectangle modules with support \([0,b) \times [0,a)\), \([0,a) \times [0,b)\), and \([0,a)\times [0,a)\), respectively.
    Let \(M_{a,b}\) be the interval module whose support is the union of the supports of \(H_{a,b}\) and \(V_{a,b}\).
    In particular, we have \(\mrd_\spm(M_{a,b}) = \left(\{[V_{a,b}], [H_{a,b}]\}, \{[T_{a}]\}\right)\).
    Given a natural number \(k \in \Nbb\), let \(\epsilon = 1/(k+1)\) and consider the following two persistence modules
    \begin{align*}
        A & = (M_{a,b} \oplus T_{a}) \oplus (M_{a+\epsilon,b} \oplus T_{a+\epsilon}) \oplus (M_{a+2\epsilon,b} \oplus T_{a+2\epsilon}) \oplus \cdots \oplus (M_{a+k\epsilon,b} \oplus T_{a+k\epsilon})                \\
        B & = (M_{a,b} \oplus T_{a+\epsilon}) \oplus (M_{a+\epsilon,b} \oplus T_{a+2\epsilon}) \oplus (M_{a+2\epsilon,b} \oplus T_{a+3\epsilon}) \oplus \cdots \oplus (M_{a+k\epsilon,b} \oplus T_{a+(k+1)\epsilon}).
    \end{align*}
    Note that $A$ and $B$ are \(\epsilon/2\)-interleaved, and that their minimal rank decompositions by rectangles are as follows (see \cref{figure:instability-mrd-rectangles}):
    \begin{align*}
        \mrd_\spm(A) & = \left( \{[V_{a,b}], [H_{a,b}]\} \cup \{[V_{a+\epsilon,b}],  [H_{a+\epsilon,b}]\} \cup \cdots \cup \{[V_{a+k\epsilon,b}],  [H_{a+k\epsilon, b}]\}, \varnothing \right)                               \\
        \mrd_\spm(B) & = \left( \{[V_{a,b}], [H_{a,b}]\} \cup \{[V_{a+\epsilon,b}], [H_{a+\epsilon,b}]\} \cup \cdots \cup \{[V_{a+k\epsilon,b}], [H_{a+k\epsilon, b}]\} \cup \{[T_{a+(k+1)\epsilon}]\},  \{[T_a]\}\right).
    \end{align*}
    Any matching between \(\mrd_\spm(A)\) and \(\mrd_\spm(B)\) has cost at least \(1\), since \([T_{a}] \in \mrd_\sm(B)\) is either unmatched, which has cost \(1\), matched with \([T_{a+(k+1)\epsilon}] = [T_{a+1}] \in \mrd_\sp(B)\), which has cost \(1\), or matched with another bar in \(\mrd_\sp(B)\), which has cost larger than \(1\).
    By letting \(k\to \infty\) we get a sequence of pairs of modules \(A\) and \(B\) whose interleaving distance converges to \(0\), and such that an optimal matching between their minimal rank decomposition by rectangles has cost at least \(1\).
\end{proof}

\begin{figure}
    \begin{center}
        \includegraphics[width=0.8\linewidth]{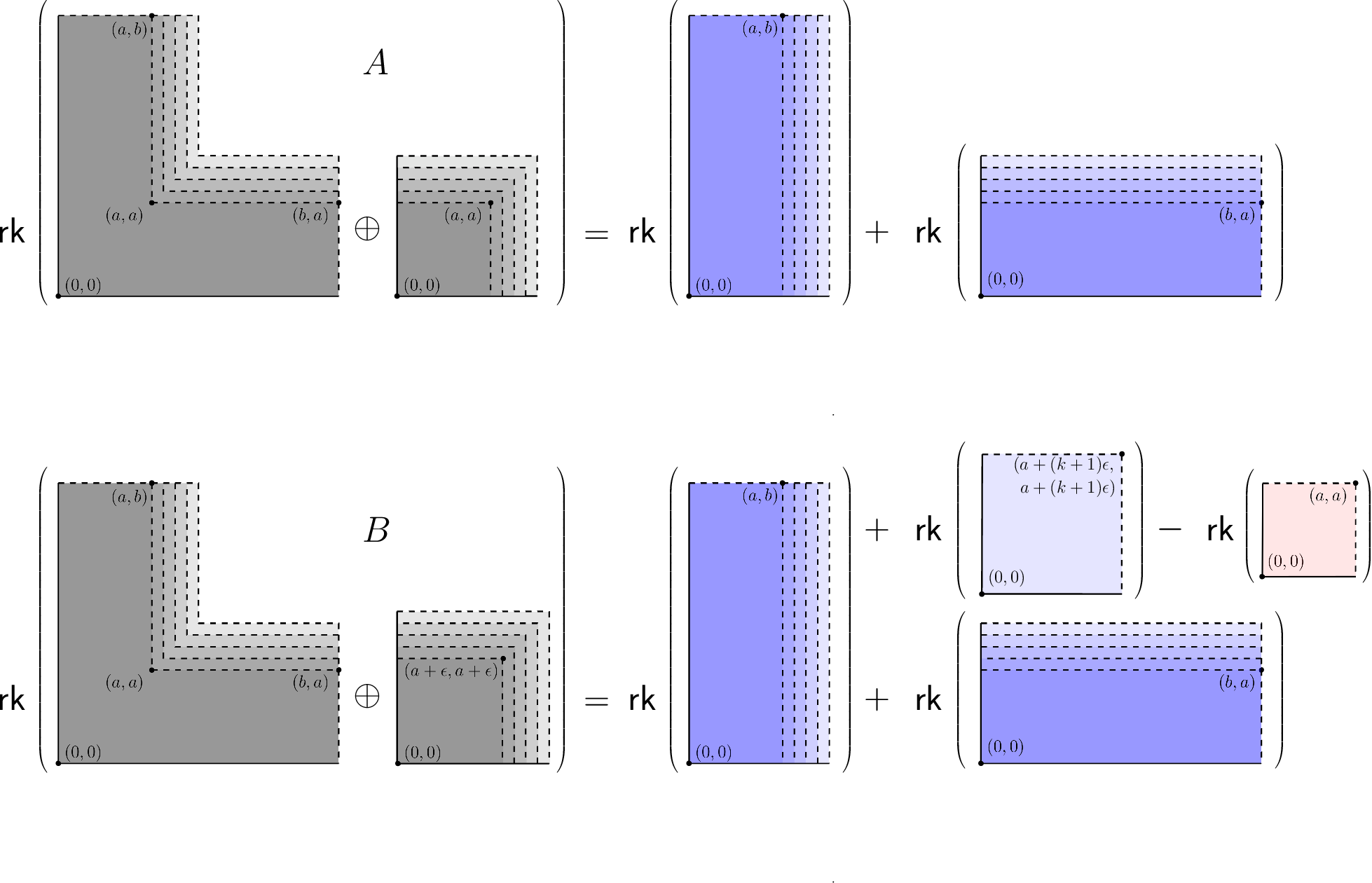}
    \end{center}
    \caption{The minimal rank decompositions of the modules \(A\) and \(B\) of the proof of \cref{proposition:instability-mrd-rectangles}, when \(k=4\).}
    \label{figure:instability-mrd-rectangles}
\end{figure}

\begin{remark}
    \label{remark:instability-mrdh}
    An example analogous to that in the proof of \cref{proposition:instability-mrd-rectangles} shows that the minimal rank decomposition by hooks (\cref{section:rank-decs}) is also not stable in the signed bottleneck dissimilarity when \(n\geq 2\).
    To build the example, one replaces the modules \(M_{a+i\epsilon,b}\) by the modules \(T_{a+i\epsilon}\), the modules \(T_{a+i\epsilon}\) by the modules \(\Lsf_{0,(a+i\epsilon,a+i\epsilon)}\), the modules \(H_{a+i\epsilon,b}\) by the modules \(\Lsf_{0,(0,a+i\epsilon)}\), and the modules \(V_{a+i\epsilon,b}\) by the modules \(\Lsf_{0,(a+i\epsilon,0)}\).
\end{remark}

\section{Bottleneck stability of hook-decomposable modules}
\label{matchings-section}

In \cref{section:matchings}, we prove \cref{proposition-stability-hooks}, the bottleneck stability result for hook-decomposable modules.
In \cref{remark:projective-stability-tight}, we show that the stability result of \cref{section:matchings} is tight when \(n=2\).
In order to do this, we repurpose an example from \cite{bjerkevik} originally designed for rectangle-decomposable modules.
Note however that,
in general, right open rectangle modules and hook modules can have quite different behaviors under interleavings or matchings, therefore we cannot directly apply Bjerkevik's stability result for rectangle-decomposable modules \cite[Theorem~4.3]{bjerkevik} to prove the stability of hook-decomposable modules.

\subsection{Bottleneck stability}
\label{section:matchings}

We start with some definitions that let us state simple generalizations of the main constructions in \cite{bjerkevik}.

Let \(\Pcal\) be a class of indecomposable persistence modules \(\Rscr^n \to \vect\).
A \define{matching constant} for \(\Pcal\) is any number \(c \geq 1 \in \Rbb\) such that, whenever \(P,Q \colon \Rscr^n \to \vect\) are \(\epsilon\)-interleaved countably \(\Pcal\)-decomposable modules, there exists a \(c\epsilon\)-matching between \(\barc(P)\) and \(\barc(Q)\) and thus, in particular, we have \(d_B(\barc(P),\barc(Q)) \leq c \cdot d_I(P,Q)\).

We use the following version of Bjerkevik's stability result \cite[Theorem~4.3]{bjerkevik},

\begin{theorem}[Bjerkevik]
    \label{bjerkevik-thm-1}
    Let \(\Pcal\) be a class of isomorphism classes of interval modules that is closed under shifts, such that the Hom-spaces between its elements are of dimension~$\leq 1$, and for which there exists a total pre-order \(\preceq\) (i.e., a pre-order such that every pair of elements is comparable) and a constant \(c \geq 1\) satisfying the following two conditions for all \(\epsilon > 0\) and all \(R,S,T \in \Pcal\) with \(R \preceq T\):
    \begin{enumerate}
        \item If there are non-zero morphisms \(\kbb_R \to \kbb_S[\epsilon]\) and \(\kbb_S \to \kbb_T[\epsilon]\), then \(S\) is \(c\epsilon\)-interleaved with \(R\) or \(T\).
        \item Assume \(R\) and \(T\) are \(2c\epsilon\)-significant (i.e., not $2c\epsilon$-interleaved with the zero module).
              If there are non-zero morphisms \(\kbb_R \to \kbb_S[\epsilon]\) and \(\kbb_S \to \kbb_T[\epsilon]\), then their composite \(\kbb_R \to \kbb_T[2\epsilon]\) is non-zero.
    \end{enumerate}
    Then \(c\) is a matching constant for \(\Pcal\).
\end{theorem}


\begin{proof}[Proof sketch]
  The proof of this result is completely analogous to the proof of \cite[Theorem~4.3]{bjerkevik}, which follows directly from \cite[Lemma~4.9]{bjerkevik} and \cite[Theorem~4.10]{bjerkevik} via a purely combinatorial argument that we outline now. 
   
Let \(P,Q \colon \Rscr^n \to \vect\) decompose as countable direct sums of indecomposables in \(\Pcal\), and let us assume that \(P\) and \(Q\) are \(\epsilon\)-interleaved. For any element $P_i\in \barc(P)$, let
  \[ \mu(P_i) = \left\{Q_j\in \barc(Q) \mid P_i\ \text{and}\ Q_j\ \text{are $c\epsilon$-interleaved}\right\}, \]
  where $\mu(P_i)$ is viewed as a multisubset of $\barc(Q)$, containing all the elements of $\barc(Q)$ with which $P_i$ can potentially be matched. More generally, for any multisubset $A\subseteq \barc(P)$, let
  \[ \mu(A) = \bigcup_{P_i\in A} \mu(P_i), \]
  where, again, $\mu(A)$ is viewed as a multisubset of~$\barc(Q)$. Now, we build the (possibly infinite) undirected bipartite graph $G$ on the multisets
  $\barc(P)\sqcup\barc(Q)$ with an edge between $P_i\in\barc(P)$ and $Q_j\in\barc(Q)$ whenever $Q_j\in\mu(P_i)$. Then, \cite[Lemma~4.9]{bjerkevik} states that $|A|\leq |\mu(A)|$ for any $A\subseteq \barc(P)$ that contains only $2c\epsilon$-significant elements. This implies that any subset $A$ of the nodes of~$G$ corresponding to $2c\epsilon$-significant  elements of~$\barc(P)$ is connected to at least $|A|$ other nodes of~$G$. As a consequence, Hall's Marriage theorem~\cite[Theorem~4.10]{bjerkevik}---applied to the subgraph of $G$ induced by the $2c\epsilon$-significant elements of~$\barc(P)$ and their neighbors in~$G$---ensures the existence of a (possibly not perfect) matching~$\sigma$  in~$G$ that covers all the nodes corresponding to  $2c\epsilon$-significant  elements of~$\barc(P)$. Symmetrically, there exists a matching~$\tau$  in~$G$ that covers all the nodes corresponding to  $2c\epsilon$-significant  elements of~$\barc(Q)$. From the union of the edges in $\sigma$ and $\tau$, one can then extract a matching in~$G$ that covers all the   $2c\epsilon$-significant  elements of~$\barc(P)\sqcup\barc(Q)$, by a combinatorial argument similar to the one used in the Cantor--Bernstein theorem \cite[pp.~110-111]{aigner1999proofs}.

  As seen from this outline, the only part of the proof that depends on the class~$\Pcal$ of indecomposables is \cite[Lemma~4.9]{bjerkevik}.  The proof of this lemma itself is quite modular,  based on a matrix rank argument, and the only places where it makes any assumptions about the intervals in \( \Pcal \)---making it specific to the case of rectangles---are:
  \begin{itemize}
\item in the construction of the matrices from the morphisms of the $\epsilon$-interleaving between $P$ and $Q$: there, the restrictions of the interleaving morphisms to individual summands of $P$ or $Q$ are identified with $\kbb$-endomorphisms, which is possible in the case of rectangle summands because the Hom-spaces between them are of dimension $\leq 1$; here we make the same assumption on the Hom-spaces between the indecomposables in~$\Pcal$;   
\item when  \cite[Lemma~4.7]{bjerkevik} is invoked; here the lemma is replaced by our assumption~(1);
\item when  \cite[Lemma~4.8]{bjerkevik} is invoked; here the lemma is replaced by our assumption~(2);
\item when $(2n - 1)$ is used, it should be replaced by our matching constant $c$. 
\end{itemize}
As a result, the proof of \cite[Lemma~4.9]{bjerkevik}, and so the one of \cite[Theorem~4.3]{bjerkevik} as well, extend to any class~$\Pcal$ that satisfies our assumptions.
\end{proof}

The following lemma is also implicit in Bjerkevik's work.

\begin{lemma}
    \label{lemma:matching-constant-union}
    Let \(\Pcal\) and \(\Pcal'\) be isomorphism classes of interval modules, closed under shifts, with matching constants \(c\) and \(c'\) respectively.
    Suppose that, every time \(\kbb_R\) and \(\kbb_T\) belong to one of the two classes and \(\kbb_S\) belongs to the other class, we have that any composite morphism \(\kbb_R \to \kbb_S \to \kbb_T\) is zero.
    Then, \(\max(c,c')\) is a matching constant for \(\Pcal \cup \Pcal'\).
\end{lemma}
\begin{proof}
    Suppose that \(P,Q \colon \Rscr^n \to \vect\) are \(\epsilon\)-interleaved, countably \((\Pcal \cup \Pcal')\)-decomposable modules.
    Let \(P = P_1 \oplus P_2\) and \(Q = Q_1 \oplus Q_2\), with \(P_1\) and \(Q_1\) being \(\Pcal\)-decomposable and \(P_2\) and \(Q_2\) being \(\Pcal'\)-decomposable.
    Let \(i \neq j \in \{1,2\}\).
    The interleaving between \(P\) and \(Q\)
consists of morphisms \(f : P \to Q[\epsilon]\) and \(g : Q \to P[\epsilon]\),
which we can write in matrix form as
\[
    \begin{pmatrix}
        f_{1,1} & f_{1,2}\\
        f_{2,1} & f_{2,2}
    \end{pmatrix}
    : P_1 \oplus P_2 \to Q_1[\epsilon] \oplus Q_2[\epsilon]
    \;\; \text{ and } \;\;
    \begin{pmatrix}
        g_{1,1} & g_{1,2}\\
        g_{2,1} & g_{2,2}
    \end{pmatrix}
    : Q_1 \oplus Q_2 \to P_1[\epsilon] \oplus P_2[\epsilon].
\]
We claim that $f_{1,1}$ and $g_{1,1}$ form an $\epsilon$-interleaving between $P_1$ and $Q_1$, and that $f_{2,2}$ and $g_{2,2}$ form an $\epsilon$-interleaving between $P_2$ and $Q_2$.
Proving this is sufficient since then, by assumption, there exists a \(c\epsilon\)-matching between \(\Bcal(P_1)\) and \(\Bcal(Q_1)\) and a \(c'\epsilon\)-matching between \(\Bcal(P_2)\) and \(\Bcal(Q_2)\), and the result follows by combining these two matchings.

Using that $f$ and $g$ form an $\epsilon$-interleaving, and standard matrix multiplication, we get
\[
    \begin{pmatrix}
        g_{1,1}[\epsilon] & g_{1,2}[\epsilon]\\
        g_{2,1}[\epsilon] & g_{2,2}[\epsilon]
    \end{pmatrix}
    \begin{pmatrix}
        f_{1,1} & f_{1,2}\\
        f_{2,1} & f_{2,2}
    \end{pmatrix}
    =
    \begin{pmatrix}
        g_{1,1}[\epsilon]\circ f_{1,1} + 
        g_{1,2}[\epsilon]\circ f_{2,1} &
        g_{1,1}[\epsilon]\circ f_{1,2} + 
        g_{1,2}[\epsilon]\circ f_{2,2} \\
        g_{2,1}[\epsilon]\circ f_{1,1} + 
        g_{2,2}[\epsilon]\circ f_{2,1} &
        g_{2,1}[\epsilon]\circ f_{1,2} + 
        g_{2,2}[\epsilon]\circ f_{2,2}
    \end{pmatrix}
    =
    \begin{pmatrix}
        \eta_{2\epsilon}^{P_1} & 0\\
        0 & \eta_{2\epsilon}^{P_2}
    \end{pmatrix},
\]
so $g_{1,1}[\epsilon]\circ f_{1,1} + g_{1,2}[\epsilon]\circ f_{2,1} = \eta_{2\epsilon}^{P_1}$ and $g_{2,1}[\epsilon]\circ f_{1,2} + g_{2,2}[\epsilon]\circ f_{2,2} = \eta_{2\epsilon}^{P_2}$.
But, by assumption, $g_{1,2}[\epsilon]\circ f_{2,1} = 0$ and $g_{2,1}[\epsilon]\circ f_{1,2} = 0$, since $P_1$ and $Q_2$ belong to different classes, and $P_2$ and $Q_1$ belong to different classes as well.
This implies that $g_{1,1}[\epsilon]\circ f_{1,1} = \eta_{2\epsilon}^{P_1}$ and $g_{2,2}[\epsilon]\circ f_{2,2} = \eta_{2\epsilon}^{P_2}$.
A symmetric argument shows that 
$f_{1,1}[\epsilon]\circ g_{1,1} = \eta_{2\epsilon}^{Q_1}$ and $f_{2,2}[\epsilon]\circ g_{2,2} = \eta_{2\epsilon}^{Q_2}$, finishing the proof.
\end{proof}

We are now ready to prove the main result of this section.

\begin{proposition}
    \label{proposition-stability-hooks}
    Let \(P,Q \colon \Rscr^n \to \vect\) decompose as countable direct sums of hook modules and let \(\epsilon \geq 0\).
    If \(P\) and \(Q\) are \(\epsilon\)-interleaved, then there exists a \((2n-1)\epsilon\)-matching between \(\barc(P)\) and \(\barc(Q)\).
\end{proposition}
\begin{proof}
    We use \cref{lemma:matching-constant-union} with one class consisting of modules of the form \(\Lsf_{a,b}\) with \(a < b \in \Rscr^n\), and the other class consisting of modules of the form \(\Lsf_{a,\infty} = \Psf_a\) with \(a \in \Rscr^n\), noticing that any morphism \(\Lsf_{a,b} \to \Psf_c\) with \(b < \infty\) must be zero.

   A matching constant for the second class is \(\max(1,n-1)\), by the bottleneck stability theorem for one-parameter barcodes (case $n=1$) and by \cite[Theorem~4.12]{bjerkevik} (case $n\geq 2$).
    It thus suffices to show that \(2n-1\) is a matching constant for the first class.
    To prove this, we use \cref{bjerkevik-thm-1}, so we need to verify conditions \((1)\) and \((2)\) in the case of hooks.
    We start by defining the pre-order \(\preceq\), which we do by following Bjerkevik's arguments.
    For \(i \in \Rscr^n\), let \(\alpha(i) = \sum_{m = 1}^n i_{m}\), where $i_m$ denotes the $m$-th coordinate of~$i$.
    Now, we let \(\Lsf_{i,j} \preceq \Lsf_{i',j'}\) if and only if \(\alpha(i) + \alpha(j) \leq \alpha(i') + \alpha(j')\).

    To verify condition \((1)\), we follow the general outline of the proof of \cite[Lemma~4.7]{bjerkevik}.
    Suppose that there exist non-zero morphisms \(f : \Lsf_{i,j} \to \Lsf_{a,b}[\epsilon]\) and \(g : \Lsf_{a,b} \to \Lsf_{c,d}[\epsilon]\), and that $\Lsf_{i,j} \preceq \Lsf_{c,d}$.
    Since the preorder $\preceq$ is linear, we have either $\Lsf_{i,j} \preceq \Lsf_{a,b}$ or $\Lsf_{a,b} \preceq \Lsf_{c,d}$.
    We assume $\Lsf_{i,j} \preceq \Lsf_{a,b}$, the other case being analogous.
    Since $f \neq 0$, it must be the case that
    \begin{enumerate}
        \item[(I)] $a \leq i + \epsilon$;
        \item[(II)] $b \leq j + \epsilon$.
    \end{enumerate}
    It is thus enough to prove that 
    \begin{itemize}
        \item $i \leq a + (2n - 1)\epsilon$;
        \item $j \leq b + (2n - 1)\epsilon$,
    \end{itemize}
    since this implies that the supports are close (see also \cref{lemma:formulas-interleavings}).
    We prove that $i_m \leq a_m + (2n-1)\epsilon$ for every $1 \leq m \leq n$ by contradiction; the case of $j$ and $b$ is analogous.
    Suppose $i_m > a_m + (2n-1)\epsilon$, using this, (I), and (II) we get:
    \begin{align*}
        \alpha(a) + \alpha(b)
            &= \sum_{k=1}^n a_k + \sum_{k=1}^n b_k\\
            &< i_m - (2n -1)\epsilon + \sum_{k\neq m} (i_k + \epsilon) + \sum_{k=1}^n (j_k + \epsilon)\\
            &= \sum_{k = 1}^n i_k + \sum_{k = 1} j_k = \alpha(i) + \alpha(j),
    \end{align*}
    which contradicts the fact that $\Lsf_{i,j} \preceq \Lsf_{a,b}$.
   

    To verify condition \((2)\), we follow the general outline of the proof of \cite[Lemma~4.8]{bjerkevik}.
    Suppose that there exist non-zero morphisms \(f \colon \Lsf_{i,j} \to \Lsf_{a,b}[\epsilon]\) and \(g \colon \Lsf_{a,b} \to \Lsf_{c,d}[\epsilon]\), that \(\Lsf_{i,j} \preceq \Lsf_{c,d}\), and that \(\Lsf_{i,j}\) and \(\Lsf_{c,d}\) are \(2(2n-1)\epsilon\)-significant.
    We must prove that \(g[\epsilon] \circ f\) is non-zero, and we proceed by contradiction.
    If we assume \(g[\epsilon] \circ f = 0\), then, we claim we have the following:
    \begin{enumerate}[(i)]
        \item \(i + 2\epsilon \geq c\);
        \item \(i + 2\epsilon \geq d\);
        \item \(j + 2\epsilon \geq d\);
        \item \(i_{k} > c_{k} + (4n-4)\epsilon\) for some \(1 \leq k \leq n\);
        \item \(j_{k'} > d_{k'} + (4n-4)\epsilon\) for some \(1 \leq k' \leq n\).
    \end{enumerate}
    The first and third claims follow from the fact that \(f\) and \(g\) are non-zero.
    The second claim follows from the fact that, if \(i_l + 2\epsilon < d_l\) for some \(1 \leq l \leq n\), then \(\Lsf_{i,j}\) and \(\Lsf_{c,d}[2\epsilon]\) would intersect, and \(g[\epsilon] \circ f\) would be non-zero.
    The fourth claim follows from the second claim and the fact that \(\Lsf_{c,d}\) is \((4n-2)\epsilon\)-significant, while the fifth claim follows from the second one and the fact that \(\Lsf_{i,j}\) is \((4n-2)\epsilon\)-significant.
    We can then compute as follows:
    \begin{align*}
        \alpha(i) + \alpha(j) & = \sum_{m = 1}^n i_m + \sum_{m' = 1}^n j_{m'}                                                                                   \\
                              & > c_{k} + (4n-4)\epsilon + d_{k'} + (4n-4)\epsilon + \sum_{m \neq k} (c_m - 2\epsilon) + \sum_{m' \neq k'} (d_{m'} - 2\epsilon) \\
                              & = \alpha(c) + \alpha(d) + (4n-4)\epsilon \geq \alpha(c) + \alpha(d),
    \end{align*}
    which contradicts the fact that \(\Lsf_{i,j} \preceq \Lsf_{c,d}\).
\end{proof}

\subsection{On the relation between hooks and rectangles}
\label{remark:projective-stability-tight}

In this section, we let \(a \lessdot b \in (\Rscr \cup \{\infty\})^n\) denote the fact that \(a_k < b_k \in \Rscr \cup \{\infty\}\) for all \(1 \leq k \leq n\).
Note that a pair \((a,b)\) with \(a \lessdot b \in \Rscr^n\) determines both the right open rectangle module \(\Rsf^o_{a,b}\) and the hook module \(\Lsf_{a,b}\).
There thus seems to be a kind of duality between these two types of modules, which we now explore.

We say that a right open rectangle module has \define{finite diagonal} if it is isomorphic to a module of the form 
\( \Rsf_{a,b}^o \) with \(a \lessdot b \in \Rscr^n\).
Similarly, a hook module has \define{finite diagonal} if it is isomorphic to a module of the form \( \Lsf_{a,b} \) with \(a \lessdot b \in \Rscr^n\).
Let \(I\) be multiset of pairs \((a,b)\) with \(a \lessdot b \in \Rscr^n\), and let
\(M = \bigoplus_{(a,b)\in I} \Rsf_{a,b}^o\) and \(N = \bigoplus_{(a,b)\in I} \Lsf_{a,b}\).
We say that \(M\) is the rectangle-decomposable module \define{corresponding} to \(N\), and that \(N\) is the hook-decomposable module \define{corresponding} to \(M\).
In particular, any rectangle-decomposable module having indecomposable summands with finite diagonal has a corresponding hook-decomposable module, well defined up to isomorphism, and any hook-decomposable module having indecomposable summands with finite diagonal has a corresponding rectangle-decomposable module, well defined up to isomorphism.



The following result gives explicit formulas for the interleaving distance between a rectangle or hook module and the zero module, and for the interleaving distance between two rectangle modules or two hook modules.
The formulas for the interleaving distances to the zero module are a consequence of the following observation: for a rectangle to be close to empty it is sufficient for one of its sides to be small, whereas in the case of a hook all sides (i.e., differences in coordinates between the generator and the relation) must be small. 
The formula for the distance between two rectangle modules or two hook modules is a case analysis: two rectangle modules or two hook modules can be close either because their endpoints are close, or because both are close to the zero module.

\begin{lemma}
    \label{lemma:formulas-interleavings}
Using the conventions \(|\infty - \infty| = 0\) and \(|a - \infty| = |\infty - a| = \infty\) for all \(a \in \Rscr\), we have the following.
If \(a \lessdot b \in (\Rscr \cup \{\infty\})^n\) and \(c \lessdot d \in (\Rscr \cup \{\infty\})^n\),
then:
\begin{align*}
    d_I(\Rsf_{a,b}^o, 0) &= \min_{1 \leq k \leq n} |b_k - a_k|/2\\
    d_I(\Rsf_{a,b}^o, \Rsf_{c,d}^o) &= \min\left(\; \max(\|c-a\|_\infty, \|d-b\|_\infty)\;,\; \max(d_I(\Rsf_{a,b}^o, 0),d_I(\Rsf_{c,d}^o, 0))\;\right)
\end{align*}
If \(a < b \in \Rscr^n \cup \{\infty\}\) and \(c < d \in \Rscr^n \cup \{\infty\}\), then:
\begin{align*}
    d_I(\Lsf_{a,b}, 0) &= \max_{1 \leq k \leq n} |b_k - a_k|/2 = \|b-a\|_\infty/2\\
    d_I(\Lsf_{a,b}, \Lsf_{c,d}) &= \min\left(\; \max(\|c-a\|_\infty, \|d-b\|_\infty)\;,\; \max(d_I(\Lsf_{a,b}, 0),d_I(\Lsf_{c,d}, 0))\;\right).
\end{align*}
Moreover, the infimum that defines the interleaving distances in the above formulas is attained by an interleaving.
\end{lemma}

In the next two lemmas, we collect properties relating rectangle modules and hook modules.

\begin{lemma}
    \label{lemma:matching-hook-to-matching-rect}
    An \(\epsilon\)-matching between barcodes consisting of hook modules with finite diagonal induces an \(\epsilon\)-matching between the barcodes consisting of the corresponding right open rectangle modules.
    The converse statement is not true in general.
\end{lemma}
\begin{proof}
    The first claim follows from the facts that 
    \( d_I(\Rsf_{a,b}^o, 0) \leq d_I(\Lsf_{a,b}, 0) \) and \(d_I(\Rsf_{a,b}^o, \Rsf_{c,d}^o) \leq d_I(\Lsf_{a,b}, \Lsf_{c,d})\), for all \(a \lessdot b \in \Rscr^n\) and \(c \lessdot d \in \Rscr^n \), by \cref{lemma:formulas-interleavings}.
    For the second claim, note that for \(x >  y > 0 \in \Rscr\), we have \(d_I(\Rsf^o_{(0,0),(x,y)}, 0) = y/2\), while \(d_I(\Lsf_{(0,0),(x,y)}, 0) = x/2\).
    In particular, there exists a \(y/2\)-matching between the barcode \(\{\Rsf^o_{(0,0),(x,y)}\}\) and the empty barcode but not between the barcode \(\{\Lsf_{(0,0),(x,y)}\}\) and the empty barcode.
\end{proof}

\begin{lemma}
    \label{lemma:int-rect-to-int-hook}
    Let \(\epsilon > 0\) and suppose that \(M\) and \(N\) are right open rectangle-decomposable modules with the property that all of their indecomposable summands have finite diagonal and are not \(\epsilon\)-interleaved with the zero module.
    If \(M\) and \(N\) are \(\epsilon\)-interleaved, then so are their corresponding hook-decomposable modules.
\end{lemma}

It is not clear to us whether the analogous statement where the roles of rectangles and hooks are swapped holds true.

\begin{proof}
    Let \(M'\) and \(N'\) be the hook-decomposable modules corresponding to \(M\) and \(N\), respectively.
    Note that any non-zero morphism between rectangle modules with finite diagonal \(\Rsf_{a,b}^o \to \Rsf_{c,d}^o\) is given by multiplication by a non-zero scalar and induces a non-zero morphism between the corresponding hook modules \(\Lsf_{a,b} \to \Lsf_{c,d}\) given by multiplication by the same scalar.
    In particular, given an \(\epsilon\)-interleaving between \(M\) and \(N\), we get a pair of morphism \(M' \to N'[\epsilon]\) and \(N' \to M'[\epsilon]\).
    Since all of the summands of \(M\) and \(N\) are \(2\epsilon\)-significant, we see that this last pair of morphisms is in fact an \(\epsilon\)-interleaving between \(M'\) and \(N'\).
\end{proof}

Combining \cref{lemma:matching-hook-to-matching-rect,lemma:int-rect-to-int-hook} with \cite[Example~5.2]{bjerkevik} we see that, when \(n=2\), the constant \(2n-1 = 3\) of \cref{proposition-stability-hooks} is the smallest possible.

\section{Global dimension of the rank exact structure and size of the rank exact decomposition}
\label{section:bounds-section}

In \cref{section:global-dimension-rank-exact-section}, we give lower and upper bounds for the global dimension of the rank exact structure on the category of modules over a finite lattice, and use these to prove \cref{theorem:gldim-rank-Rn}, which gives the global dimension of the rank exact structure on the category of \fpp \(\Rscr^n\)-persistence modules.
In \cref{section:polynomial-bound}, we prove \cref{proposition:polynomial-bound}, which gives lower and upper bounds for the size of the rank exact decomposition of a \fpp module \(M\) in terms of the size of the (usual) Betti numbers of \(M\).

\subsection{Global dimension}
\label{section:global-dimension-rank-exact-section}
The aim of this subsection is to determine the global dimension of categories of persistence modules with respect to the rank exact structure (recall that global dimension is defined in \cref{section:exact-structures}).
More precisely, we show that the global dimension of \( n \)-parameter persistence modules is precisely \( 2n - 2 \) (\cref{proposition:global-dimension-rank-grid,theorem:gldim-rank-Rn}).
We also give more general bounds for the global dimension of the rank exact structure on any finite lattice \( \Sscr \) (Theorem~\ref{theorem:gldim-rank-poset}).

The way towards these results is a bit roundabout: First we introduce a new lattice \( \Pscr \), consisting, in spirit, of compatible pairs of elements of \( \Sscr \) (see Notation~\ref{notation:P_from_S} for the precise construction).
The key point is that (a certain quotient of) \( \Pscr \) is closely related to the rank projectives over \( \Sscr \) (Lemma~\ref{lemma.algebra_isomorphism}), giving us a way to translate rank projective dimensions to usual projective dimensions (Lemma~\ref{lemma.equivalence_X_and_L}).
This type of argument is familiar from the representation theory of finite dimensional algebras, and is sometimes referred to as ``projectivization'' \cite[Section~2,~Chapter~2]{auslander-reiten-smalo}.

As a first step, we look at (usual) projective resolutions over a lattice \( \Pscr \). (We will use this result for the specific \( \Pscr \) alluded to above eventually.)
%
%
%
We construct a Koszul-like resolution for interval modules whose support is an upset.
We use this resolution in \cref{corollary:projective-dimension-rectangle} to bound the projective dimension of rectangle modules, which in particular include the (usual) injective modules.

\begin{proposition}
    \label{proposition:Koszul-for-upper-set}
    Let \(\Pscr\) be a lattice. 
    Let \(m \in \Nbb\) and let \(f \colon [m] \to \Pscr\) be a function, not assumed to be monotonic.
    Let \(\Ical = \{i \in \Pscr \colon \exists k \in [m], f(k) \leq i\} \subseteq \Pscr\) be the upset generated by \(f([m])\).
    There exists a \(\kbb \Pscr\)-projective resolution of \(\kbb_\Ical\) of length \(m-1\), as follows:
    \[
        0 \to \Psf_{\vee f([m])} \to \bigoplus_{\substack{A \subseteq [m] \\ |A| = m-1}} \Psf_{\vee f(A)}
        \to \cdots
        \to \bigoplus_{\substack{A \subseteq [m] \\ |A| = 2}} \Psf_{\vee f(A)} \to \bigoplus_{k \in [m]} \Psf_{f(k)} \to \kbb_{\Ical}.
    \]
    If for every proper subset \(A \subsetneq [m]\) we have \(\vee f(A) < \vee f([m])\), then the above resolution is a shortest possible projective resolution.
\end{proposition}


\begin{proof}
    We start by specifying the morphisms in the resolution. Note that if \( [m] \supseteq A \supseteq B \) then \( \vee f(A) \geq \vee f(B) \), and so we have a natural inclusion of \( \iota_{A, B} \colon \Psf_{\vee f(A)} \to \Psf_{\vee f(B)} \). The components of the morphisms in the resolution above are given by
    \[
        \Psf_{\vee f(A)} \to \Psf_{\vee f(B)} 
        \colon
        \begin{cases}
            (-1)^{| \{ b \in B : b \leq i \} |} \cdot \iota_{A, B} & \text{if } A = B \sqcup \{ i \}   \\ 0 & \text{else.}
        \end{cases}
    \]
    Now we look at these maps pointwise. So fix \( p \in \Pscr \). Since
    \[ \Psf_{\vee f(A)}(p) = \begin{cases} \kbb & \text{if $p \geq f(a)$ for all $a \in A$} \\ 0 & \text{else,} \end{cases} \]
    we observe that only subsets of \( [m]^{\leq p} := \{ i : f(i) \leq p \} \subseteq [m] \) contribute for this \( p \). Thus, when evaluating at $p \in \Pscr$, our sequence of maps degenerates to
    \[ 0 \to \kbb \to \bigoplus_{\substack{A \subseteq [m]^{\leq p}\\|A| = |[m]^{\leq p}|-1}} \kbb \to \cdots \to \bigoplus_{\substack{A \subseteq [m]^{\leq p}\\|A| = 2}} \kbb \to \bigoplus_{a \in [m]^{\leq p}} \kbb. \]
    With the signs explained above, this is exactly the complex calculating simplicial homology of a (filled) \( (|[m]^{\leq p}|-1) \)-simplex. If \( [m]^{\leq p} = \varnothing \), then the entire complex vanishes; otherwise this homology is concentrated in degree \( 0 \), where it is \( \kbb \).
    Thus, we do have a sequence which is exact everywhere except at its right end, and the cokernel is \( \kbb \) whenever there is at least one \( a \in [m] \) such that \( p \geq f(a) \), in other words if \( p \in \Ical \). This establishes the resolution of the proposition.

    Finally, if we additionally assume that \( \vee f(A) < \vee f([m]) \) for all proper subsets \( A \) of \( [m] \), then the first morphism (on the left) of the resolution is not a split monomorphism, so the resolution is a shortest possible projective resolution.
\end{proof}

\begin{notation}
Let \(\Pscr\) be a finite poset.
If \(i \in \Pscr\), let \(i^\wedge \subseteq \Pscr\) be the set of minimal elements of \(\{a \in \Pscr \colon a \nleq i\}\), and let \(i^+ = \{a \in \Pscr \colon i < a, \nexists i < j < a\} \subseteq \Pscr\) be the set of elements directly above \(i\) in the Hasse diagram of~\(\Pscr\).
\end{notation}

\begin{corollary}
    \label{corollary:projective-dimension-rectangle}
    Let \(\Pscr\) be a finite lattice, let \(0 \in \Pscr\) denote its minimum, and let \(i \in \Pscr\).
    Then \(\pdim_{\kbb \Pscr}(\Rsf_{0,i}) \leq | i^\wedge |\).
    If for any proper subset \( A \subsetneq i^\wedge \) we have \( \vee A < \vee i^\wedge \), then \( \pdim_{\kbb \Pscr}(\Rsf_{0,i}) = | i^\wedge | \).
\end{corollary}
\begin{proof}
    This follows by combining \cref{proposition:Koszul-for-upper-set} with the short exact sequence \(0 \to \kbb_{\Ical} \to \Psf_0 \to \Rsf_{0,i} \to 0\), where \(\Ical\) is the upset generated by \(i^\wedge\).
\end{proof}




Our overall aim in this subsection is to relate the rank global dimension (and certain rank projective dimensions) of the exact structure on one given lattice \( \Sscr \) to the usual global dimension (and usual projective dimensions) on a different lattice \( \Pscr \).
However, we need to consider a quotient of the algebra \( \kbb \Pscr \). Our next aim is thus to introduce this type of quotient, and then work towards \cref{lemma:proj-dim-kP/Q,prop.gldim_for_kP/e} which, in essence, say that the quotient construction does not influence our homological dimensions too much.


\begin{definition}
    If \( \Qscr \subseteq \Pscr \) are finite posets, let \(e = \sum_{i \in \Qscr} [i,i] \in \kbb\Pscr\) be the idempotent of \(\kbb\Pscr\) corresponding to \(\Qscr\).
    We define the quotient algebra \(\kbb \Pscr/\Qscr \coloneqq \kbb \Pscr / (e)\), where \((e)\) is the two-sided ideal of \(\kbb \Pscr\) generated by \(e\).
\end{definition}

\begin{lemma}
    \label{lemma:kP/Q-description}
    Let \( \Qscr \subseteq \Pscr \) be finite posets.
    The \(\kbb\)-algebra \(\kbb \Pscr/\Qscr\) is isomorphic to the following \(\kbb\)-algebra \(\Abb\).
    As a vector space, \(\Abb\) is freely generated by pairs \([i,j]\) such that \(i \leq j \in \Pscr\) and such that there does not exist \(x \in \Qscr\) with \(i \leq x \leq j\).
    The multiplication of \(\Abb\) is given by linearly extending the rule
    \[
        [i,j] \cdot [k,l] =
        \begin{cases}
            [i,l], & \text{if \(j=k\), and \(\nexists x \in \Qscr\) such that \(i \leq x \leq l\).} \\
            0,     & \text{otherwise.}
        \end{cases}
    \]
\end{lemma}
\begin{proof}
    Consider the \(\kbb\)-linear map \(\kbb \Pscr \to \Abb\) which maps \([i,j]\) to \([i,j]\) if there does not exist \(x \in \Qscr\) with \( i \leq x \leq j \), and to \(0\) otherwise.
    It is clear that this linear map is a \(\kbb\)-algebra morphism.
    Its kernel is generated, as a \(\kbb \Pscr\)-bimodule, by the elements \([x,x]\) with \(x \in \Qscr\).
    Thus, its kernel is the two-sided ideal \((e)\), where \(e \in \kbb \Pscr\) is the idempotent corresponding to \(\Qscr\).
\end{proof}

The first technical step towards relating the homological algebra over \( \kbb \Pscr/\Qscr \) back to \( \kbb \Pscr \) is the following \( \Tor \)-vanishing result.
Recall that $S_k$ denotes the simple corresponding to $k \in \Pscr$ (defined in \cref{section:interval-modules}).

\begin{lemma}
    \label{lemma.no_tor}
    Let \( \Qscr \subseteq \Pscr \) be finite posets such that any subset of \( \Qscr \) has a join in \( \Pscr \) which lies in \( \Qscr \).
    Then, for \( k \not\in \Qscr \) we have
    \(\Tor^{\kbb \Pscr}_d(\Ssf_k, \kbb \Pscr/\Qscr) = 0\), for all \(d > 0 \in \Nbb\).
\end{lemma}

\begin{proof}
    Denote by \( e \in \kbb \Pscr \) the idempotent of \( \kbb \Pscr \) corresponding to \( \Qscr \). We use the short exact sequence
    \[ 0 \to (e) \to \kbb \Pscr \to \kbb \Pscr / \Qscr \to 0 \]
    to calculate the desired \( \Tor \)-groups. Since the middle term is projective, we have \(\Tor^{\kbb \Pscr}_d(\Ssf_k, \kbb \Pscr/\Qscr) \cong \Tor^{\kbb \Pscr}_{d-1}(\Ssf_k, (e))\) whenever \( d \geq 2 \), and
    \(\Tor^{\kbb \Pscr}_d(\Ssf_k, \kbb \Pscr/\Qscr) \) isomorphic to a subgroup of \(\Tor^{\kbb \Pscr}_{d-1}(\Ssf_k, (e))\) if \( d = 1 \). In particular, it suffices to show that \( \Tor^{\kbb \Pscr}_{d-1}(\Ssf_k, (e)) = 0 \) for all \( d \geq 1 \).

    To that end, note that \( (e) = \Span \{ [i, j] \colon \exists s \in \Qscr \text{ with } i \leq s \leq j \} \).
    Let us first consider the case \( d = 1 \), where we look at \( \Tor^{\kbb \Pscr}_{0}(\Ssf_k, (e)) = \Ssf_k \otimes_{\kbb \Pscr} (e) \). Clearly an elementary tensor \( x \otimes [i,j] = x [i,i] \otimes [i,j] \) vanishes if \( i \neq k \). If \( i = k \) then we can pick \( s \) as in the description of \( (e) \) above, and note that \( s \neq k \). Thus, also in this case the elementary tensor \( x \otimes [i,j] = x [i,s] \otimes [s,y] \) vanishes.

    Now we consider \( d > 1 \). Here it suffices to show that \( (e) \) is projective as a left \( \kbb \Pscr \)-module. To that end, note that \( (e) = \bigoplus_{j \in \Pscr} \Span \{ [i, j] \colon \exists s \in \Qscr \text{ with } i \leq s \leq j \} \) as left \( \kbb \Pscr \)-modules. Moreover, for a given \( j \) we have
    \[ \Span \{ [i, j] \colon \exists s \in \Qscr \text{ with } i \leq s \leq j \} = \Span \{ [i, j] \colon i \leq \vee \{ s \in S \colon s \leq j \} \} \cong \Span \{ [i, s] \colon i \leq \vee \{ s \in S \colon s \leq j \} \}, \]
    where the last isomorphism is given by mapping \( [i,j] \) to \( [i,s] \). Note that the final module is the projective left module associated to the vertex \( \vee \{ s \in S \colon s \leq j \} \), which is in particular it is a projective left module.
    This concludes the proof.
\end{proof}

\begin{lemma}
    \label{lemma:proj-dim-kP/Q}
    If \(\Qscr \subseteq \Pscr\) are finite posets such that any subset of \(\Qscr\) has a join in \(\Pscr\) which lies in \(\Qscr\), then \(\gldim(\kbb \Pscr/\Qscr) \leq \gldim(\kbb \Pscr)\).
\end{lemma}
\begin{proof}
    It is a standard result (see, e.g., \cite[Corollary~11]{auslander}) that the global dimension of a finite dimensional algebra is the maximum of the projective dimensions of its simple modules (i.e., modules with no non-zero proper submodules). Thus, it is enough to show that, for every simple module \(M \in \modcat_{\kbb \Pscr/\Qscr}\), we have \(\pdim_{\kbb \Pscr/\Qscr}(M) \leq \gldim(\kbb \Pscr)\).
    We claim that any simple \(\kbb \Pscr/\Qscr\)-module \(M\) is isomorphic to \(\Ssf_i \otimes_{\kbb \Pscr} \kbb\Pscr/\Qscr\) for some \(i \in \Pscr \setminus \Qscr\).
    To see this, note that, as a \(\kbb \Pscr\)-module, the module \(M\) must also be simple.
    This implies that, as a \(\kbb \Pscr\)-module, \(M\) must be isomorphic to \(\Ssf_i\) for some \(i \in \Pscr\), so that, as a \(\kbb \Pscr/\Qscr\)-module, \(M\) must be isomorphic to \(\Ssf_i \otimes_{\kbb \Pscr} \kbb \Pscr/\Qscr\), with \(i \notin \Qscr\) since \(M\) is non-zero.

    We can now take a \(\kbb \Pscr\)-projective resolution of \(\Ssf_i\) of length at most \(\gldim(\kbb \Pscr)\) and tensor it on the right with \(\kbb \Pscr/\Qscr\), which, by \cref{lemma.no_tor}, gives us a \(\kbb \Pscr/\Qscr\)-projective resolution of \(\Ssf_i \otimes_{\kbb \Pscr} \kbb \Pscr/\Qscr \cong M\) of length at most \(\gldim(\kbb \Pscr)\), as required.
\end{proof}
%

\begin{lemma}
    \label{prop.gldim_for_kP/e}
    Let \(\Pscr\) be a finite lattice and let \(\Qscr \subseteq \Pscr\) be closed under joins and meets.
    For any \( i \not\in \Qscr \), we have that \( \pdim_{\kbb \Pscr / \Qscr}\left( \Ssf_i \otimes_{\kbb \Pscr} \kbb \Pscr / \Qscr \right) \leq | i^+ | \).
    If, moreover, \( \vee i^+ \not\in \Qscr \) and for any proper subset \( A \subsetneq i^+ \) we have \( \vee A < \vee i^+ \), then \( \pdim_{\kbb \Pscr / \Qscr}\left( \Ssf_i \otimes_{\kbb \Pscr} \kbb \Pscr / \Qscr\right) = | i^+ | \).
\end{lemma}
\begin{proof}
    Let \(\Ical\) be the upset of elements that are greater than or equal to an element of \(i^+\).
    By splicing the short exact sequence \(0 \to \kbb_\Ical \to \Psf_i \to \Ssf_{i} \to 0\) with the resolution of \(\kbb_\Ical\) in \cref{proposition:Koszul-for-upper-set},
    we get a projective resolution of \(\Ssf_i\) of length \(|i^+|\).
    Tensor this resolution with \( \kbb \Pscr / \Qscr \).
    By \cref{lemma.no_tor} we obtain a projective resolution of \( \Ssf_i \otimes_{\kbb \Pscr} \kbb \Pscr/\Qscr \) over \( \kbb \Pscr / \Qscr \), which proves the first claim.
    To prove the second claim, note that, under the additional assumptions of the claim, the first map (on the left) in the resolution we constructed is not a split monomorphism.
\end{proof}
%

Now we turn back to our original goal of studying the rank global dimension over a lattice \( \Sscr \). We do so by relating it to the usual global dimension of an algebra \( \kbb \Pscr / \Qscr \) of the form we have just studied. First we fix the precise setup. 

\begin{notation}
\label{notation:P_from_S}
For the rest of this section, we let \( \Sscr \) be a finite lattice, we let \( \Sscr_{\infty} \coloneqq \Sscr \cup \{ \infty \} \), and we let \( \Pscr \coloneqq \{ (a, b) \in \Sscr_\infty^2 \colon a \leq b \} \), which is a lattice with the pointwise partial order.
Note that \( \Qscr = \{ (a, a) \colon a \in \Sscr_\infty \} \subseteq \Pscr \) is a sublattice of \(\Pscr\).
We write \( \Lbb = \bigoplus_{a < b \in \Sscr_\infty} \Lsf_{a, b} \in \modcat_{\kbb \Sscr}\).
We denote the projective dimension relative to the rank exact structure of a module \(M \in \mod_{\kbb \Sscr}\) by \(\pdimrank_{\kbb \Sscr}(M)\).
\end{notation}

\begin{lemma} \label{lemma.algebra_isomorphism}
    We have an isomorphism of \(\kbb\)-algebras \(\End_{\kbb \Sscr}( \Lbb ) \cong \kbb\Pscr / \Qscr\).
\end{lemma}

\begin{proof}
    We start by noting that
    \[  \Hom_{\kbb \Sscr}( \Lsf_{a, b}, \Lsf_{c,d}) = \begin{cases} \kbb & \text{if } a \geq c, b \geq d, a \not\geq d \\ 0 & \text{otherwise.} \end{cases} \]
    Thus, as a vector space, \(\End_{\kbb \Sscr}(\Lbb)\) is generated by the morphisms \(f_{(a,b),(c,d)} \colon \Lsf_{c,d} \to \Lsf_{a,b}\), only defined for pairs \( (a, b) \leq (c, d) \) such that there does not exist \(x \in \Sscr\) with \( (a, b) \leq (x,x) \leq (c, d) \), which map \(1 \in \kbb = \Lsf_{a,b}(a)\) to \(1 \in \kbb = \Lsf_{c,d}(a)\).
    Thus, \(\End_{\kbb \Sscr}(\Lbb)\) has a basis in bijection with the basis of the algebra \(\Abb\) of \cref{lemma:kP/Q-description}.
    It is straightforward to see that the algebra multiplication coincides with that of \(\Abb\), so the result follows from \cref{lemma:kP/Q-description}.
\end{proof}

Since \(\Lbb\) is an \((\End(\Lbb),\kbb \Sscr)\)-bimodule by \cref{lemma.algebra_isomorphism}, we have a functor
\[ 
    \Hom_{\kbb \Sscr}(\Lbb, -) \colon \mod_{\kbb \Sscr} \to \mod_{\kbb \Pscr/\Qscr}.
\]
Let \(\add \Lbb\) denote the full subcategory of \(\modcat_{\kbb \Sscr}\) spanned by modules that decompose as direct sums of modules of the form \(\Lsf_{a,b}\) with \(a<b \in \Sscr_{\infty}\), and let \(\proj \kbb \Pscr/\Qscr\) denote the full subcategory of \(\modcat_{\kbb \Pscr/\Qscr}\) spanned by the projective modules.

We now establish a fundamental property of the functor $\Hom_{\kbb \Sscr}(\Lbb, -) \colon \mod_{\kbb \Sscr} \to \mod_{\kbb \Pscr/\Qscr}$, which in particular lets us compute the projective dimension of a module $M$ relative to the rank exact structure by computing the (usual) projective dimension of $\Hom_{\kbb \Sscr}(\Lbb, M)$.

\begin{lemma} \label{lemma.equivalence_X_and_L}
    The functor \( \Hom_{\kbb \Sscr}(\Lbb, -) \colon \mod_{\kbb \Sscr} \to \mod_{\kbb \Pscr /\Qscr} \) restricts to an equivalence of categories \( \add \Lbb \simeq \proj \kbb \Pscr/\Qscr \).
    In particular, we have \( \pdimrank_{\kbb \Sscr}(M) = \pdim_{\kbb \Pscr/\Qscr}(\Hom(\Lbb, M)) \), for every \( M \in \mod_{\kbb \Sscr} \).
\end{lemma}

\begin{proof}
    We start by proving the first claim.
    By definition, \(\Hom_{\kbb \Sscr}(\Lbb, -) \) maps \(\Lbb\) to \(\End_{\kbb \Sscr}(\Lbb)\), which is isomorphic to \(\kbb \Pscr/\Qscr\), by \cref{lemma.algebra_isomorphism}.
    Moreover, we have \(\End_{\kbb \Sscr}(\Lbb) \cong \kbb \Pscr/\Qscr \cong \End_{\kbb \Pscr/\Qscr}(\kbb \Pscr/\Qscr)\), also induced by \(\Hom_{\kbb \Sscr}(\Lbb, -) \).
    This implies that \(\Hom_{\kbb \Sscr}(\Lbb, -) \) restricts to an equivalence of categories between the subcategory of \(\modcat_{\kbb \Sscr}\) spanned by the object \(\Lbb\) and the subcategory of \(\modcat_{\kbb \Pscr/\Qscr}\) spanned by the object \(\kbb \Pscr/\Qscr\).
    Since the functor \(\Hom_{\kbb \Sscr}(\Lbb, -) \) is additive, it also restricts to an equivalence \( \add \Lbb \simeq \proj \kbb \Pscr/\Qscr \) between the additive closures of the modules \(\Lbb\) and \(\kbb \Pscr/\Qscr\).

    To prove the equality between the projective dimensions, note that \( \Hom_{\kbb \Sscr}(\Lbb, -) \) maps minimal relative projective resolutions to minimal \( \kbb \Pscr/\Qscr \)-projective resolutions, where minimality is preserved by the first claim.
\end{proof}

In order to bound the global dimension of the rank exact structure, we need an explicit bound for the projective dimension relative to the rank exact structure of (usual) injective modules.
We compute this by constructing an explicit rank projective resolution for them.
For this, we need the following.

\begin{lemma}
    \label{sufficient-conditions-rank-exact}
    Let \(\Ebb \colon 0 \to A \to B \to C \to 0\) be a short exact sequence in \(\modcat_{\kbb \Sscr}\).
    Assume that \(A\) has property \((\ast)\): for all \(i \in \Sscr\), \(A(0) \to A(i)\) is surjective.
    Then \(\Ebb\) is rank exact.
    Similarly, a long exact sequence of modules with property \((\ast)\) is rank exact.
\end{lemma}
\begin{proof}
    We start with the first claim.
    Let \(i < j \in \Sscr\).
    Since \(A(i) \to A(j)\) is surjective, we have that \(0 \to \Hom(\Lsf_{i,j},A) \to A(i) \to A(j) \to 0\) is exact.
    We can then consider the following diagram:
    \[
        \begin{tikzpicture}
            \matrix (m) [matrix of math nodes,row sep=1em,column sep=1em,minimum width=2em,nodes={text height=1.75ex,text depth=0.25ex}]
            {      & 0                  & 0                  & 0                  &                          \\
                 0 & \Hom(\Lsf_{i,j},A) & \Hom(\Lsf_{i,j},B) & \Hom(\Lsf_{i,j},C) & 0                        \\
                 0 & A(i)               & B(i)               & C(i)               & 0                             \\
                 0 & A(j)               & B(j)               & C(j)               & 0                             \\
                   & 0                  &                    &                    & \\};
            \path[line width=0.75pt, -{>[width=8pt]}]
            (m-2-1) edge (m-2-2)
            (m-3-1) edge (m-3-2)
            (m-4-1) edge (m-4-2)
            (m-2-4) edge (m-2-5)
            (m-3-4) edge (m-3-5)
            (m-4-4) edge (m-4-5)
            (m-4-4) edge (m-4-5)
            (m-1-2) edge (m-2-2)
            (m-1-3) edge (m-2-3)
            (m-1-4) edge (m-2-4)
            (m-4-2) edge (m-5-2)
            (m-2-2) edge (m-2-3) (m-2-3) edge (m-2-4)
            (m-3-2) edge (m-3-3) (m-3-3) edge (m-3-4)
            (m-4-2) edge (m-4-3) (m-4-3) edge (m-4-4)
            (m-2-2) edge (m-3-2)
            (m-2-3) edge (m-3-3)
            (m-2-4) edge (m-3-4)
            (m-3-2) edge (m-4-2)
            (m-3-3) edge (m-4-3)
            (m-3-4) edge (m-4-4)
            ;
        \end{tikzpicture}
    \]
    The middle and bottom rows are exact by assumption and the first column is exact by the previous observation.
    By the snake lemma, the top row is exact, as required.

    For the second claim, let \(X_\bullet\) be a long exact sequence.
    We prove that in \( 0 \to \ker d_k \to X_k \to \coker d_{k+1} \to 0\) the module \(\ker d_k\) satisfies property \((\ast)\), and, since \(\ker d_k \cong \coker d_{k+2}\) by exactness, it is enough to check this for \(\coker d_k\) for all $k$.
    Let \(i \in \Sscr\).
    Since the composite \(X(0) \to X(i) \to (\coker d_k)(i)\) is surjective and factors through \((\coker d_k)(0) \to (\coker d_k)(i)\), we have that this last map is surjective, as needed.
\end{proof}

The proof of the following result is straightforward.

\begin{lemma}
    \label{lemma:functoriality-Koszul}
    The resolution of \cref{proposition:Koszul-for-upper-set} is functorial in \(f\), in the following sense.
    Let \(f, f' \colon [m] \to \Pscr\) be two not necessarily monotonic functions such that \(f'(\ell) \leq f(\ell)\) for all \(\ell \in [m]\).
    Let \(\Ical^f\) (resp.\ \(\Ical^{f'}\)) denote the upset generated by the image of \(f\) (resp.\ \(f'\)), and let \(R_f\) (resp.\ \(R_{f'}\)) denote the projective resolution of \(\kbb_{\Ical^f}\) (resp.\ \(\kbb_{\Ical^{f'}}\)) of \cref{proposition:Koszul-for-upper-set}.
    Then there is a monomorphism of resolutions \(R_f \to R_{f'}\) extending the natural monomorphism \(\kbb_{\Ical^f} \to \kbb_{\Ical^{f'}}\).
    In degree \(k\), the morphism is given by the natural monomorphism 
    \[
        \bigoplus_{\substack{A \subseteq [m] \\ |A| = k}} \Psf_{\vee f(A)} \to \bigoplus_{\substack{A \subseteq [m] \\ |A| = k}} \Psf_{\vee f'(A)}
    \]
    induced by the natural monomorphism in each summand.
    \qed
\end{lemma}

We can now bound the projective dimension relative to the rank exact structure of (usual) injective modules.

\begin{proposition}
    \label{proposition:injective-relative-pdim}
    Let \(M \in \modcat_{\kbb \Sscr}\) be \(\kbb \Sscr\)-injective.
    Then \(\pdimrank_{\kbb\Sscr}(M) \leq \sup_{i \in \Sscr} |i^\wedge| - 1\).
\end{proposition}
\begin{proof}
    Let \(0 \in \Sscr\) be the minimum element of \(\Sscr\).
    By \cref{lemma:proj-and-inj-modules}, the module \(M\) is a direct sum of modules of the form \(\Rsf_{0,i}\), with \(i \in \Sscr\), so it is enough to prove the result for one these modules.
    We construct a relative projective resolution of \(\Rsf_{0,i}\) by taking a quotient of a resolution of \(\Psf_0\).
    Let \(k = |i^\wedge|\) and let \(f \colon [k] \to i^\wedge \subseteq \Sscr\) be a bijection.
    Let \(f' \colon [k] \to \Sscr\) be constantly \(0\).
    Let \(\Ical' = \Sscr\), and let \(\Ical\) be the upset generated by \(i^\wedge\).
    Applying \cref{proposition:Koszul-for-upper-set} and \cref{lemma:functoriality-Koszul} to \(f\) and \(f'\), we get a monomorphism of resolutions \(R_f \to R_{f'}\) of length \(k-1\) extending the natural inclusion \(\kbb_\Ical \to \kbb_{\Ical'}\).
    By \cref{lemma:functoriality-Koszul}, the cokernel of \(R_f \to R_{f'}\) is a resolution of \(\Rsf_{0,i}\) by hook modules, which has length \(k-1 = |i^\wedge| - 1\).
    In order to see that this resolution is rank exact, we apply \cref{sufficient-conditions-rank-exact}, since all the hook modules appearing in the cokernel of \( \Rsf_f \to \Rsf_{f'} \) are generated at \(0\).
\end{proof}

Before bounding the global dimension of the rank exact structure, we state and prove an easy consequence the existence of the Koszul-like resolution of \cref{proposition:Koszul-for-upper-set} for the usual global dimension.
Note that similar results are already known (see, e.g., \cite[Theorem~2.6]{iyama-marczinzik}).

\begin{proposition}
    \label{proposition:gl-dim-lattice}
    If \(\Sscr\) is a finite lattice, then \(\gldim(\kbb\Sscr) \leq \sup_{i \in \Sscr} |i^\wedge|\).
\end{proposition}
\begin{proof}
    Recall that any \(M \in \modcat_{\kbb \Sscr}\) admits a finite injective resolution (\cref{lemma:finite-poset-finite-resolution}).
    Using induction and the fact that, for every short exact sequence of \(\kbb \Sscr\)-modules \(0 \to L \to M \to N \to 0\), we have \(\pdim(L) \leq \max\{\pdim(M),\pdim(N)\}\), we reduce the problem to proving that the projective dimension of all indecomposable \(\kbb\Sscr\)-injective modules is at most \(\sup_{i \in \Sscr} |i^\wedge|\).
    By \cref{lemma:proj-and-inj-modules}, the indecomposable injectives are isomorphic to modules of the form \(\Rsf_{0,i}\), where \(0 \in \Sscr\) is the minimum of \(\Sscr\), so the result follows from \cref{corollary:projective-dimension-rectangle}.
\end{proof}


We now give a lower and upper bound for the global dimension of the rank exact structure on an arbitrary finite lattice.

\begin{theorem}
    \label{theorem:gldim-rank-poset}
    For any finite lattice $\Sscr$, we have \(\gldim(\kbb \Pscr/\Qscr) - 2  \leq \gldimrank(\Sscr) \leq 2 \cdot \sup_{i \in \Sscr} |i^\wedge| - 2\).
\end{theorem}
\begin{proof}
    We start with the upper bound.
    In view of \cref{lemma.equivalence_X_and_L}, it is enough to show that, for every \( M \in \mod_{\kbb \Sscr} \), we have \( \pdim_{\kbb \Pscr/\Qscr}(\Hom(\Lbb,M)) \leq 2n-2 \).
    Since there are enough (usual) injectives, we may choose \( 0 \to M \to I^0 \to I^1 \) exact with \( I^0 \) and \( I^1 \) injective.
    Applying \( \Hom(\Lbb, - ) \) gives an exact sequence
    \[ 0 \to \Hom(\Lbb, M) \to \Hom(\Lbb, I^0) \to  \Hom(\Lbb, I^1). \]
    We denote the cokernel of the final map by \( C \). The resulting six-term exact sequence implies that
    \[
        \pdim_{\kbb \Pscr/\Qscr}(\Hom(\Lbb, M)) \leq \max \{ \pdim_{\kbb \Pscr/\Qscr}(\Hom(\Lbb, I^0)), \pdim_{\kbb \Pscr/\Qscr}(\Hom(\Lbb, I^1)) - 1, \pdim_{\kbb \Pscr/\Qscr}( C) - 2 \}.
    \]
    Let \(d = \sup_{i \in \Sscr} |i^\wedge|\).
    Since, by \cref{proposition:injective-relative-pdim}, injectives have relative projective resolutions of length at most \( d-1 \), we have \( \pdim_{\kbb \Pscr/\Qscr}(\Hom(\Lbb,M)) \leq \max\{d-1, d-2, \pdim_{\kbb \Pscr/\Qscr} (C) - 2\} \), by \cref{lemma.equivalence_X_and_L}.
    It is thus enough to show that \(\gldim(\kbb \Pscr/\Qscr) \leq 2d\),
    and \cref{lemma:proj-dim-kP/Q} reduces this to showing \(\gldim(\kbb \Pscr) \leq 2d\).
    By applying \cref{proposition:gl-dim-lattice} to the lattice \(\Pscr\), we see that it is enough to show that, for \(j \leq k \in \Sscr\), we have \(|(j,k)^\wedge| \leq |j^\wedge| + |k^\wedge|\).
    To show this, note that if \((a,b) \nleq (j,k)\), then \(a \nleq j\) or \(b \nleq k\).
    This implies that, if \((a,b) \in (j,k)^\wedge\), then \(a=0\) or \(b=0\), since, if \(a\neq 0\) and \(b \neq 0\), then \((a,0) \nleq (j,k)\) and \((a,0) < (a,b)\), so \((a,b)\) would not be minimal.
    To conclude, note that if \((a,0) \in (j,k)^\wedge\) then \(a \in j^\wedge\), and if \((0,b) \in (j,k)^\wedge\), then \(b \in k^\wedge\).
    
    We now prove the lower bound in part (2).
    Let \(c = \gldim(\kbb \Pscr/\Qscr)\) and let \(Q_\bullet \to M\) be a minimal projective resolution of length \(c\) of some \(\kbb \Pscr/\Qscr\)-module \(M\).
    By \cref{lemma.equivalence_X_and_L}, there exist a chain complex of relative projectives \( 0 \to P_c \to \cdots \to P_1 \to P_0 \) such that \( 0 \to \Hom(\Lbb, P_{c}) \to \cdots \to \Hom(\Lbb,P_1) \to \Hom(\Lbb,P_0) \) is isomorphic to \( 0 \to Q_{c} \to \cdots \to Q_1 \to Q_0 \).
    Since \( \Hom(\Lbb, - ) \) is left exact, we have that \( \Hom(\Lbb,\ker(P_1 \to P_0)) \cong \ker(Q_1 \to Q_0) \).
    It follows that \( 0 \to \Hom(\Lbb, P_{c}) \to \cdots \to \Hom(\Lbb,P_2) \to \Hom(\Lbb,\ker(P_1 \to P_0)) \) is a minimal projective resolution of length \( c - 2 \), and thus that \( \pdimrank_{\kbb \Sscr}( \ker(P_1 \to P_0) ) \geq c - 2 \), by \cref{lemma.equivalence_X_and_L}.
\end{proof}

As a consequence, we can compute the global dimension of the rank exact structure on a finite grid.

\begin{proposition}
    \label{proposition:global-dimension-rank-grid}
    Let \(m \geq 3\) and \(\Sscr = [m]^n\).
    Then \(\gldimrank(\modcat_{\kbb \Sscr}) = 2n - 2\).
\end{proposition}
\begin{proof}
    For \(\Sscr = [m]^n\), \(m \geq 3\), and \(i \in \Sscr\), we have \(|i^\wedge| \leq n\), so \(\gldimrank(\Sscr) \leq 2n - 2\) by the upper bound in \cref{theorem:gldim-rank-poset}.
    Denote \(1 = (1,\dots,1) \in \Sscr\).
    Using the notation of \cref{theorem:gldim-rank-poset}, consider the element \((0,1) \in \Pscr\).
    By \cref{prop.gldim_for_kP/e}, we have \(\pdim(\Ssf_{(0,1)} \otimes_{\kbb \Pscr} \kbb \Pscr/\Qscr) = |(0,1)^+| = 2n\), since \(\vee (0,1)^+ = (1,2) \notin \Qscr\),
    and \(\vee A < (1,2)\) for every \(A \subsetneq (0,1)^+\), where \(2 = (2, \dots, 2) \in \Sscr\).
    The lower bound in \cref{theorem:gldim-rank-poset} now finishes the proof.
\end{proof}

We conclude this section by proving \cref{theorem:gldim-rank-Rn}.

\begin{theorem}
    \label{theorem:gldim-rank-Rn}
    We have \(\gldimrank\left(\vect^{\Rscr^n}_{\textup{\fpp}}\right) = 2n-2\).
\end{theorem}
\begin{proof}
    This is a direct consequence of \cref{proposition:global-dimension-rank-grid} and \cref{theorem:rank-exact-structure-facts}(2).
\end{proof}

\subsection{Bounds on the size of rank exact decomposition}
\label{section:polynomial-bound}
We start by fixing notation.
Given a \fpp persistence module \(M \colon \Rscr^n \to \vect\), let
\begin{align*}
    \bettibars(M) &= |\beta_0(M) \cup \dots \cup \beta_n(M)|,\\
    \bettibarsrank(M) &= |\barcrank_\sp(M) \cup \barcrank_\sm(M)| = |\bettirank_0(M) \cup \dots \cup \bettirank_{2n-2}(M)|,\\
    \decbars(M) &= |\mrd_\sp(M) \cup \mrd_\sm(M)|.
\end{align*}

The following result is standard; see \cref{section:appendix} for a proof.

\begin{restatable}{lemma}{bounddimensionles}
    \label{lemma:bound-dimensions-les}
    Let \(0 \to X_\ell \to \cdots \to X_1 \to X_0 \to 0\) be an exact sequence of finite dimensional vector spaces.
    Let the sequence \([x,y,z]\) be a reordering of the sequence \([0,1,2]\).
    Then, we have \(\sum_{i \equiv x (\mod\, 3)} \dim(X_i) \leq \sum_{i \equiv y (\mod\, 3)} \dim(X_i) + \sum_{i \equiv z (\mod\, 3)} \dim(X_i)\).
\end{restatable}

We now prove a standard result which gives a characterization of the (usual) Betti numbers of a \fpp persistence module in terms of an \(\Ext\) functor, and use it to bound the size of the Betti numbers of a module in a short exact sequence by the sum of the sizes of the Betti numbers of the other two modules.

\begin{lemma}
    \label{lemma:betti-numbers-as-ext}
    Let $\Pscr$ be a finite poset, let \(M \colon \Pscr \to \vect\) be a persistence module, let \(a \in \Pscr\), and let \(i \in \Nbb\).
    Then, the multiplicity of \([\Psf_a]\) in \(\beta_i(M)\) is equal to the dimension of \(\Ext^i_{\kbb \Pscr}(M,\Ssf_a)\)
\end{lemma}
\begin{proof}
    Let \(P_\bullet \to M\) be a minimal resolution of \(M\) by (usual) projectives.
    Assume that \(j \in \Nbb\) and that \(\Psf_a\) is a direct summand of both \(P_{j+1}\) and \(P_j\).
    Then, the differential \(P_{j+1} \to P_j\), composed with the isomorphisms \(P_{j+1} \cong \Psf_a \oplus P_{j+1}'\) and \(P_j \cong \Psf_a \oplus P_j'\), must restrict to the zero morphism \(\Psf_a \to \Psf_a\), since, otherwise, the resolution \(P_\bullet \to M\) would not be minimal.
    This implies that the differentials of the chain complex \(\Hom_{\kbb \Pscr}(P_\bullet, \Ssf_a)\) are all zero, and thus \(\Ext^i_{\kbb \Pscr}(M,\Ssf_a) \cong \Hom_{\kbb \Pscr}(P_i, \Ssf_a)\).
    To conclude, note that \(\dim(\Hom_{\kbb \Pscr}(\Psf_b,\Ssf_a))\) is either \(0\) or \(1\), and is \(1\) if and only if \(a=b\).
\end{proof}

\begin{lemma}
    \label{lemma:betti-ses}
    Let \(0 \to M_0 \to M_1 \to M_2 \to 0\) be a short exact sequence of \fpp \(\Rscr^n\)-persistence modules and let \([x,y,z]\) be a reordering of \([0,1,2]\).
    Then \(\bettibars(M_x) \leq \bettibars(M_y) + \bettibars(M_z)\).
\end{lemma}
\begin{proof}
    A straightforward extension of \cref{theorem:rank-exact-structure-facts}(2) shows that we can assume that \(M_0\), \(M_1\), and \(M_2\) are left Kan extensions along a common monotonic function \([m]^n \to \Rscr^n\) for \(m \in \Nbb\), and thus that \(M_0\), \(M_1\), and \(M_2\) are \([m]^n\)-persistence modules.
    If \(N \colon [m]^n \to \vect\), we know from \cref{lemma:betti-numbers-as-ext} that for \(a \in [m]^n\) and \(i \in \Nbb\), the multiplicity of \([\Psf_a]\) in \(\beta_i(N)\) is equal to the dimension of \(\Ext^i_{\kbb [m]^n}(N,\Ssf_a)\), so that \(|\beta_i(N)| = \sum_{a \in [m]^n} \dim(\Ext^i_{\kbb [m]^n}(N,\Ssf_a)) = \dim(\Ext^i_{\kbb [m]^n}(N,\bigoplus_{a \in [m]^n}\Ssf_a))\).
    Then, the result follows by applying \cref{lemma:bound-dimensions-les} to the long exact sequence of \(\Ext^\bullet_{\kbb [m]^n}(-, \bigoplus_{a \in [m]^n}\Ssf_a)\) induced by \(0 \to M_0 \to M_1 \to M_2 \to 0\).
\end{proof}

As a consequence, we can bound the size of the Betti numbers of the first relative syzygy of a module in terms of the size of the absolute and relative Betti numbers of the module.

\begin{lemma}
    \label{lemma:inductive-step}
    Let \(M \colon \Rscr^n \to \vect\) be \fpp and not relative projective.
    Let \(P \to M\) be the first step in a minimal rank projective resolution of \(M\) and let \( 0 \to S \to P \to M \to 0\) be the induced short exact sequence.
    Then \(\bettibars(S) \leq 2 \cdot |\bettirank_0(M)| + \bettibars(M)\).
\end{lemma}
\begin{proof}
    Let \(P = \bigoplus_{(a,b) \in A} \Lsf_{a,b}\), where \(A\) is a multiset of elements of \(\Rscr^n \times (\Rscr^n \cup \{\infty\})\),
    and apply \cref{lemma:betti-ses}, using that \(|A| = |\bettirank_0(M)|\) and \(\bettibars\left(\bigoplus_{(a,b) \in A} \Lsf_{a,b}\right) \leq 2 \cdot |A|\).
\end{proof}

The following is a general construction that allows us to bound the size of the relative Betti numbers in terms of the size of the absolute Betti numbers, when we only have a bound for the size of $0$th relative Betti numbers.

\begin{lemma}
    \label{lemma:induction}
    Assume that there exists \(\alpha \geq 2 \in \Rbb\) such that, for all \fpp \(M \colon \Rscr^n \to \vect\), we have \(|\bettirank_0(M)| \leq \bettibars(M)^\alpha\).
    Then, for all \(i \in \Nbb\), we have \(\sum_{j=0}^i|\bettirank_j(M)| \leq \big(4\cdot \bettibars(M)\big)^{\alpha^{i+1}}\).
\end{lemma}

\begin{proof}
    We prove \(\sum_{j=0}^i|\bettirank_j(M)| \leq 4^{(2\alpha)^i} \cdot \bettibars(M)^{\alpha^{i+1}}\) for all $i \in \Nbb$.
    The result then follows from the fact that $(2\alpha)^i \leq \alpha^{i+1}$.

    We start with a few observations.
    To keep notation short, in this proof we let \(s_i = \sum_{j=0}^i|\bettirank_j(M)|\) and \(\bettibars = \bettibars(M)\). 
    Let \(P_\bullet \to M\) be a minimal rank projective resolution.
    Let \(M_0 = M\), let \(M_1\) be the kernel of \(P_0 \to M\), and let \(M_{i+1}\) be the kernel of \(P_i \to P_{i-1}\).
    In particular, we have short exact sequences \(0 \to M_{i+1} \to P_i \to M_i \to 0\).
    Using \cref{lemma:inductive-step} on the short exact sequence \(0 \to M_{i} \to P_{i-1} \to M_{i-1} \to 0\), we get \(\bettibars(M_i) \leq 2 \cdot |\bettirank_{i-1}(M)| + \bettibars(M_{i-1})\), since \(\bettirank_0(M_{i-1}) = \bettirank_{i-1}(M)\).
    An inductive argument then shows that
    \begin{equation}
        \label{eq:technical-lemma-1}
        \bettibars(M_i) \leq 2 s_{i-1} + \bettibars.
    \end{equation}
    We use this to bound \(s_i\) as follows
    \begin{align*}
        s_i &= s_{i-1} + |\bettirank_i(M)|\\
            &= s_{i-1} + |\bettirank_0(M_i)|\\
            &\leq s_{i-1} + \bettibars(M_i)^\alpha\\
            &\leq s_{i-1} + (2 s_{i-1} + \bettibars)^\alpha,
    \end{align*}
    where in the first inequality we used the hypothesis of this result and in the second inequality we used \cref{eq:technical-lemma-1}.
    We use this inequality to prove the result, by induction on \(i\).

    The base case \(i=0\) is true by assumption.
    For the inductive step, we have
    \begin{align*}
        s_i &\leq (2 \cdot s_{i-1} + \bettibars)^\alpha + s_{i-1},\\
            &\leq (3 \cdot s_{i-1} + \bettibars)^\alpha,\\
            &\leq \left( 3 \cdot 4^{(2\alpha)^{i-1}} \cdot \bettibars^{\alpha^i} + \bettibars\right)^\alpha\\
            &\leq \left( 3 \cdot 4^{(2\alpha)^{i-1}} \cdot \bettibars^{\alpha^i} + 4^{(2\alpha)^{i-1}}\cdot \bettibars^{\alpha^i}\right)^\alpha\\
            &\leq \left( 4 \cdot 4^{(2\alpha)^{i-1}} \cdot \bettibars^{\alpha^i}\right)^\alpha\\
            &= 4^{(2\alpha)^{i-1}\alpha + \alpha} \cdot \bettibars^{\alpha^{i+1}}
    \end{align*}
    where in the second inequality we used the inductive hypothesis, and in the rest of the inequalities we used simple bounds.
    To conclude, note that \((2\alpha)^{i-1}\alpha + \alpha \leq (2\alpha)^{i}\).
\end{proof}

In order to be able to use \cref{lemma:induction}, we need an explicit bound for the size of $0$th relative Betti numbers in terms of the absolute Betti numbers.

\begin{lemma}
    \label{lemma:explicit-bound-degree-0}
    Let \(M \colon \Rscr^n \to \vect\) be \fpp.
    Then \(|\bettirank_0(M)| \leq \bettibars(M)^{2n + 1}\).
\end{lemma}
\begin{proof}
    Let $\bigoplus_{j \in J} \Psf_j \to \bigoplus_{i \in I} \Psf_i \to M$ be a minimal projective presentation of $M$, so that \(\beta_0(M) = \{[\Psf_i]\}_{i \in I}\) and \(\beta_1(M) = \{[\Psf_j]\}_{j \in J}\), and \(I\) and \(J\) are multisets of elements of \(\Rscr^n\).
    Let \(\bar{I} \subseteq \Rscr^n\) be the subset of \(\Rscr^n\) consisting of distinct elements of \(I\), and define \(\bar{J}\) analogously.
    Let \(\bar{I}^{\vee 0} = \bar{I} \subseteq \Rscr^n\) and let \(\bar{I}^{\vee k+1} = \bar{I} \vee \bar{I}^{\vee k}\), where, for \(X,Y \subseteq \Rscr^n\), we let \(X\vee Y = \{x\vee y \colon x \in X, y \in Y\} \subseteq \Rscr^n\).
    Define \(\bar{J}^{\vee k}\) analogously.

    Let \(\bigoplus_{(a,b) \in A} \Lsf_{a,b} \to M\) be the first step in a minimal rank projective resolution of \(M\), where \(A\) is a multiset of elements of \(\Rscr^n \times (\Rscr^n \cup \{\infty\})\).
    We claim that, for each \((a,b) \in A\), we have \(a \in \bar{I}^{\vee n}\) and either \(b \in \bar{J}^{\vee n}\) or \(b = \infty\).
    This follows from two claims: (a) any morphism \(\Lsf_{a',b'} \to M\) with \(b' \neq \infty\) factors through a map \(\Lsf_{a'',b''} \to M\) with \(a'' \in \bar{I}^{\vee k}\) and \(b'' \in \bar{J}^{\vee k}\) for a sufficiently large \(k\); (b) \(\bar{I}^{\vee k} = \bar{I}^{\vee n}\) and \(\bar{J}^{\vee k} = \bar{J}^{\vee n}\) for all \(k \geq n\).
    The case \(b = \infty\) is analogous.

    To prove claim (b), observe that any element of $\Rscr^n$ that is written as a join of $k$ elements with $k \geq n$ can actually be written as a join of $n$ of those elements, since the join operation in $\Rscr^n$ is equivalently a componentwise maximum.
    To prove claim (a), note that such a morphism induces a commutative diagram 
    \[
        \begin{tikzpicture}
            \matrix (m) [matrix of math nodes,row sep=2em,column sep=4em,minimum width=2em,nodes={text height=1.75ex,text depth=0.25ex}]
            { \Psf_{b'} & {\bigoplus_{j \in J} \Psf_j}   \\
              \Psf_{a'} & \bigoplus_{i \in I}\Psf_i \\
              M & M. \\};
            \path[line width=0.75pt, -{>[width=8pt]}]
            (m-1-1) edge (m-1-2)
            (m-1-1) edge (m-2-1)
            (m-2-1) edge (m-2-2)
            (m-1-2) edge (m-2-2)
            (m-2-1) edge (m-3-1)
            (m-2-2) edge (m-3-2)
            (m-3-1) edge [double equal sign distance,-] (m-3-2)
            ;
        \end{tikzpicture}
    \]
    The morphism $\Psf_{a'} \to \bigoplus_{i \in I}\Psf_i$ clearly factors through $\Psf_{a''}$, for some $a'' \in \bar{I}^{\vee k}$ for large enough $k$, and, analogously, the morphism $\Psf_{b'} \to \bigoplus_{j \in J}\Psf_j$ factors through $\Psf_{b''}$, for some $b'' \in \bar{J}^{\vee k}$.
    The claim follows.

    If we let $\bar{A} \subseteq \Rscr^n \times (\Rscr^n \cup \{\infty\})$ be the set of distinct elements of $A$, we have proved that $|\bar{A}| \leq |\bar{I}^{\vee n}| + (|\bar{I}^{\vee n}| \cdot |\bar{J}^{\vee n}|)$.

    Now, note that \(\dim(\Hom(\Lsf_{a,b},M)) = \dim(\ker(M(a) \to M(b)))\) if $b \neq \infty$, and $\dim(\Hom(\Lsf_{a,b},M)) = \dim(M(a))$ if $b = \infty$.
    In either case, \(\dim(\Hom(\Lsf_{a,b},M)) \leq \dim(M(a)) \leq |I|\).
    Thus, if $(a,b) \in \bar{A}$, the multiplicity of $(a,b)$ as an element of $A$ is bounded above by $\dim(\Hom(\Lsf_{a,b},M))$, and thus by $|I|$.

    This lets us bound as follows
    \begin{align*}
        |\bettirank_0(M)|
        &= |A|
        = \sum_{(a,b) \in \bar{A}} \text{ (multiplicity of $(a,b)$ in $A$)} \leq |\bar{A}| \cdot |I|\\
        &\leq 
        \left(|\bar{I}^{\vee n}| + (|\bar{I}^{\vee n}| \cdot |\bar{J}^{\vee n}|) \right) |I| \leq |I|^{n+1} \cdot (1 + |J|^n) \leq \bettibars(M)^{2n + 1},
    \end{align*}
    as required.
\end{proof}

The results up to this point allow us to give an upper bound for the size of the relative Betti numbers in terms of the size of the absolute Betti numbers.
The next three results allow us to give a lower bound.

\begin{lemma}
    \label{lemma:koszul-upset-R2}
    Let \(\{g_1, \dots, g_k\} \subseteq \Rscr^2\) be a sequence of incomparable elements ordered by their abscissa and let \(\Ical = \{x \in \Rscr^2 \colon \exists a, x \geq g_a\}\).
    Then \(\beta_0(\kbb_\Ical) = \{[\Psf_{g_1}],\dots, [\Psf_{g_k}]\}\), \(\beta_1(\kbb_\Ical) = \{[\Psf_{g_1\vee g_2}],[\Psf_{g_2\vee g_3}], \dots, [\Psf_{g_{k-1} \vee g_{k}}]\}\), and \(\beta_2(\kbb_\Ical) = \varnothing\).
    In particular, \(\bettibars(\kbb_{\Ical}) = 2k-1\).
\end{lemma}
\begin{proof}
    It is easy to verify that we have minimal resolution
    \[
        0 \to \bigoplus_{\ell=2}^k \Psf_{g_{\ell-1} \vee g_\ell} \to \bigoplus_{\ell = 1}^k \Psf_{g_\ell} \to \kbb_{\Ical},
    \]
    by considering the first two stages of the resolution of \cref{proposition:Koszul-for-upper-set}.
\end{proof}

\begin{lemma}
    \label{lemma:lower-bound-bars-mrd}
    Let \(\{g_1, \dots, g_k\} \subseteq \Rscr^2\) be a sequence of incomparable elements of \(\Rscr^2\), ordered by their abscissa, and let \(\{c_1, \dots, c_\ell\} \subseteq \Rscr^2\) be a sequence of incomparable elements of \(\Rscr^2\), also ordered by their abscissa.
    Assume that \((g_k)_x < (c_1)_x\) and that \((g_1)_y < (c_\ell)_y\), so that, in particular, \(\Ical = \{x  \in \Rscr^2 \colon \exists a, g_a \leq x, \exists b, x \leq c_b\}\) is an interval of \(\Rscr^2\).
    We have \(\bettibars(\kbb_\Ical) \leq 2(k + \ell)\) and \(\decbars(\kbb_\Ical) \geq k\ell\).
\end{lemma}
\begin{proof}
    We start with the first claim.
    Let \(\Jcal = \{ x \in \Rscr^2 \colon \exists a, g_a \leq x\}\), \(B = \{ (((g_1)_x,(c_1)_y), c_1 \wedge c_2, c_2 \wedge c_3, \dots, c_{\ell - 2} \wedge c_{\ell - 1}, c_{\ell - 1} \wedge c_\ell, ((c_\ell)_x,(g_k)_y)\}\), and \(\Kcal = \{ x \in \Rscr^2 \colon \exists b \in B, b\leq x\}\).
    We have a short exact sequence \(0 \to \kbb_\Kcal \to \kbb_\Jcal \to \kbb_\Ical \to 0\).
    From \cref{lemma:betti-ses}, it follows that \(\bettibars(\kbb_\Ical) \leq \bettibars(\kbb_\Jcal) + \bettibars(\kbb_\Kcal)\).
    So that \(\bettibars(\kbb_\Ical) \leq 2(k+\ell)\) by \cref{lemma:koszul-upset-R2}.

    The second claim follows from \cite[Equation~5.1]{botnan-oppermann-oudot}, which implies that
    \begin{align*}
        \mrd_\sp(\kbb_\Ical) & = \left\{[\Rsf^o_{g_a,c_b}]\right\}_{\substack{1 \leq a \leq k                 \\1 \leq b \leq \ell}}\;\; \cup\;\; \left\{[\Rsf^o_{g_a\vee g_{a+1},c_{b}\wedge c_{b+1}}]\right\}_{\substack{1 \leq a \leq k-1\\1 \leq b \leq \ell-1}}\\
        \mrd_\sm(\kbb_\Ical) & = \left\{[\Rsf^o_{g_a,c_{b}\wedge c_{b+1}}]\right\}_{\substack{1 \leq a \leq k \\1 \leq b \leq \ell-1}}\;\; \cup\;\; \left\{[\Rsf^o_{g_a\vee g_{a+1},c_b}]\right\}_{\substack{1 \leq a \leq k-1\\1 \leq b \leq \ell}},
    \end{align*}
    and thus \(\decbars(\kbb_\Ical) \geq k\ell\).
\end{proof}

\begin{figure}
    \begin{center}
        \includegraphics[width=0.7\linewidth]{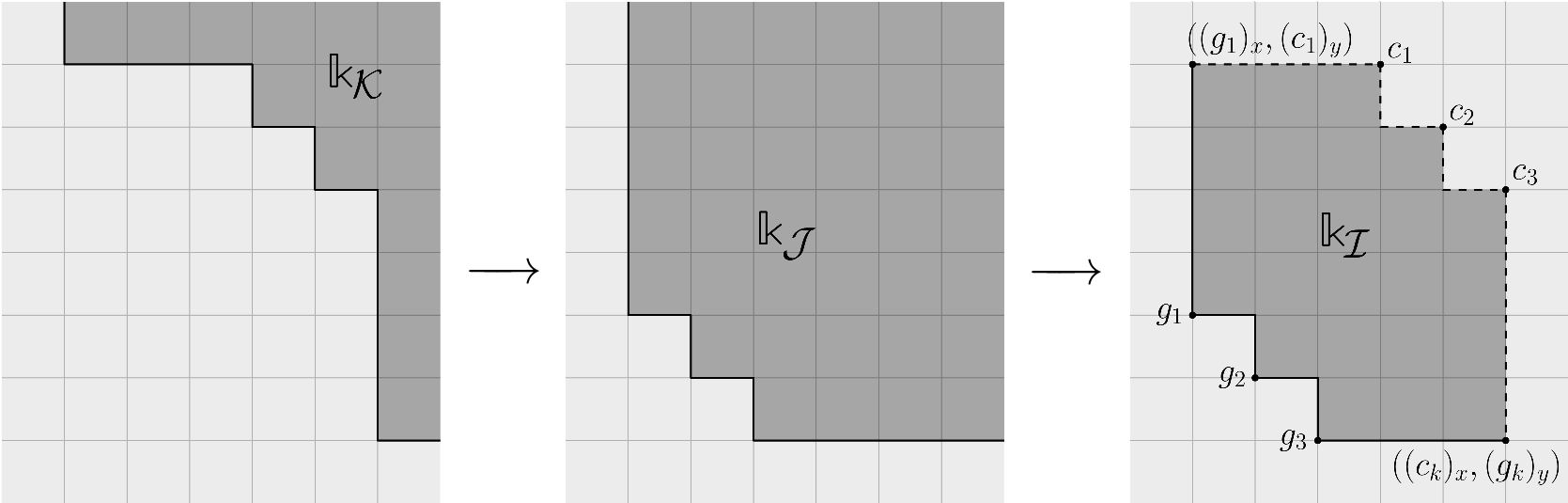}
    \end{center}
    \caption{The short exact sequence in the proof of \cref{lemma:lower-bound-bars-mrd}, when \(k = \ell = 3\).}
    \label{figure:lower-bound-bars-mrd}
\end{figure}

\begin{lemma}
    \label{lemma:mrd-hook}
    Let \(a < b \in \Rscr^n \cup \{\infty\}\) and let \(\Lsf_{a,b} \colon \Rscr^n \to \vect\) be the corresponding hook module.
    Then \(\decbars(\Lsf_{a,b}) \leq 2^n - 1\).
    In particular, for all \fpp \(M \colon \Rscr^n \to \vect\), we have \(\decbars(M) \leq (2^n - 1)\, \bettibarsrank(M)\).
\end{lemma}
\begin{proof}
    For the first claim, let \(\Ical \subseteq \Rscr^n\) be the support of \(\Lsf_{a,b}\).
    It is observed in the proof of \cite[Theorem~5.4]{botnan-oppermann-oudot} that $\Ical$ is the union of at most \(n\) non-empty, right open rectangles $R_1, \dots, R_k$ ($k \leq n$), which are all downward-closed subsets of \(\Ical\).
    Then, \cite[Lemma~5.5]{botnan-oppermann-oudot} implies that there exists a rank exact sequence starting with $0 \to \kbb_\Ical$ and such that the rest of the summands are the modules form $\kbb_{\cap_{i \in S} R_i}$ for $\varnothing \neq S \subseteq \{1, \dots, k\}$, each one appearing exactly once.
    This implies that the minimal rank decomposition by rectangles of $\kbb_\Ical = \Lsf_{a,b}$ has as many rectangles as there are non-empty subsets of $\{1, \dots, k\}$, so that \(\decbars(\Lsf_{a,b}) \leq \sum_{1 \leq i \leq k} {k \choose i} = 2^k - 1 \leq 2^n - 1\).

    For the second claim, note that the minimal rank decomposition by rectangles of \(M\) can be computed by considering the rank exact decomposition of \(M\), taking the minimal rank decomposition by rectangles of each of the bars taking their sign into account, and then canceling, with multiplicity, the bars that appear both as positive and negative.
\end{proof}

We are now ready to prove the final result of this section.

\begin{proposition}
    \label{proposition:polynomial-bound}
    For all \fpp \(M \colon \Rscr^n \to \vect\), we have
    \[
        \bettibarsrank(M) \leq \big(4 \cdot \bettibars(M)\big)^{(2n+1)^{2n-1}},
    \]
    and there exists a family of \fpp modules $M_k : \Rscr^2 \to \vect$ for $k \in \Nbb$ such that $\bettibars(M_k) \leq 4k$ and $\bettibarsrank(M_k) \geq k^2/3$.
\end{proposition}
\begin{proof}
    For the upper bound, we use \cref{lemma:induction} and \cref{lemma:explicit-bound-degree-0}, by instantiating \(i = 2n-2\), using the global dimension of the rank exact structure (\cref{theorem:gldim-rank-Rn}), and \(\alpha = 2n +1\).

    For the lower bound, we use \cref{lemma:lower-bound-bars-mrd}.
    For this, we let \(M_k \coloneqq \kbb_\Ical \colon \Rscr^n \to \vect\) as in the statement of the result, taking \(\ell = k\), $g_i = (i-1,-i+1)$ for $1 \leq i \leq k$, and $c_i = (k+i, k+2-i)$ for $1 \leq i \leq k$.
    We thus have \(\bettibars(M_k) \leq 4k\) and \(k^2 \leq \decbars(M_k)\) by \cref{lemma:lower-bound-bars-mrd}.
    To conclude the proof, we combine the last inequality with the inequality of \cref{lemma:mrd-hook}.
\end{proof}

The next example shows how the upper bound of \cref{proposition:polynomial-bound} works for the module of \cref{example:rank-exact-resolution}.
In particular, it shows that the upper bound can be as low as linear.

\begin{example}
    Let $M \coloneqq \kbb_\Ical : \Rscr^2 \to \vect$ with $\Ical = \big([0,2)\times [0,1)\big) \cup \big([0,1)\times [0,2)\big)$, as in \cref{example:rank-exact-resolution}.
    An argument like the one in \cref{lemma:lower-bound-bars-mrd} (see \cref{lemma:koszul-upset-R2,figure:lower-bound-bars-mrd}) shows that $\beta_0(M) = \{[\Psf_{(0,0)}]\}$, $\beta_1 = \{[\Psf_{0,2},\Psf_{2,0},\Psf_{1,1}]\}$, and $\beta_2 = \{[\Psf_{1,2}],[\Psf_{2,1}]\}$. 
    In particular, $\bettibars(M) = 6$.
    Meanwhile, $\bettibarsrank(M) = 5$, by \cref{example:rank-exact-resolution}.
    This shows that the upper bound of \cref{proposition:polynomial-bound} can be very coarse in general.
    Indeed, following the analysis in this example, one can show that, if $k \geq 2$ and $M_k \coloneqq \kbb_{\Ical_k} : \Rscr^2 \to \vect$ with
    \[
        \Ical_k =
            \big([0,1)\times [0,k)\big)
            \cup
            \big([0,2)\times [0,k-1)\big)
            \cup
            \cdots
            \cup
            \big([0,k-1)\times [0,2)\big)
            \cup
            \big([0,k)\times [0,1)\big),
    \]
    then $\bettibars(M_k) = 2k+2$ and $\bettibarsrank(M_k) = 2k+1$.
    This means that there exist infinite families of modules with $\bettibars$ not bounded and with $\bettibarsrank$ linear in $\bettibars$.
\end{example}

We conclude this section stating a conjecture that says that, in the two parameter case, we expect the size of the Betti numbers relative to the rank exact structure to be at most cubic in the size of the usual Betti numbers.

\begin{conjecture}
    \label{conjecture:cubic-bound}
    Given \(M \colon \Rscr^2 \to \vect\) \fpp, we have $\bettibarsrank(M) \in O\left( \bettibars(M)^{3}\right)$, in the sense that $\bettibarsrank(M)$ can be bounded above by $\bettibars(M)^3$ times a constant independent of $M$.
\end{conjecture}

\section{Bottleneck stability of the rank exact decomposition and universality}

In \cref{section:bottleneck-stability-rank-intro-signed-barcode}, we prove \cref{theorem:stability-rank}, the bottleneck stability result for the rank exact decomposition.
In \cref{section:universality}, we prove \cref{proposition:universality}, a universality result for the signed bottleneck dissimilarity restricted to signed barcodes of hook modules.
We also show that the triangle inequality precludes the existence of a discriminative \emph{and} ''rank-stable'' dissimilarity on signed barcodes of hook modules (\cref{prop:no-stable-tri}).

\subsection{Bottleneck stability of the rank exact decomposition}
\label{section:bottleneck-stability-rank-intro-signed-barcode}

\begin{theorem}
    \label{theorem:stability-rank}
    Let \(n \geq 1 \in \Nbb\).
    For all \fpp modules \(M,N \colon \Rscr^n \to \vect\), we have
    \[
        \widehat{d_B}(\, \barcrank_\spm(M)\,,\, \barcrank_\spm(N)\,) \leq (2n-1)^2 \cdot d_I(M,N).
    \]
\end{theorem}
\begin{proof}
    Let \(\epsilon \geq 0 \in \Rscr\) and assume given \fpp modules \(M,N \colon \Rscr^n \to \vect\) that are \(\epsilon\)-interleaved.
    Let \(P_\bullet \to M\) and \(Q_\bullet \to N\) be minimal rank projective resolutions.
    By \cref{theorem:gldim-rank-poset} these resolutions have length at most \(2n-2\).
    By \cref{lemma:relative-stability-resolutions}, we have that
    \(A' = P_0 \oplus Q_1[\epsilon] \oplus P_2[2\epsilon] \oplus Q_3[3\epsilon] \oplus \cdots\) and \(B' = Q_0 \oplus P_1[\epsilon] \oplus Q_2[2\epsilon] \oplus P_3[3\epsilon] \oplus \cdots\) are \(\epsilon\)-interleaved.
    Since \(P_\bullet\) and \(Q_\bullet\) have length at most \(2n-2\), it follows that \(A = P_0 \oplus Q_1 \oplus P_2 \oplus Q_3 \oplus \cdots\) is \((0;(2n-2)\epsilon)\)-interleaved with \(A'\) and that \(B = Q_0 \oplus P_1 \oplus Q_2 \oplus P_3 \oplus \cdots\) is \((0;(2n-2)\epsilon)\)-interleaved with \(B'\).
    By composing these interleavings, we see that \(A\) and \(B\) are \((2n-1)\epsilon\)-interleaved.
    To conclude, note that \(\barc(A) = \bettirank_{2\Nbb}(M) \cup \bettirank_{2\Nbb+1}(N)\) and \(\barc(B) = \bettirank_{2\Nbb}(N) \cup \bettirank_{2\Nbb+1}(M)\), so the result follows from \cref{proposition-stability-hooks}.
\end{proof}


\begin{example}
    \label{example:stability}
    For an illustration of this example, see \cref{figure:example-stability}.
    Let $M,N : \Rscr^2 \to \vect$, with $M \coloneq \Rsf_{(0,0),(5,5)}^o$ and $N \coloneqq \kbb_{\Ical}$, with $\Ical = \big([0,5)\times [0,4)\big) \cup \big([0,4)\times [0,5)\big)$.
    An argument similar to that in \cref{example:rank-exact-resolution} shows that the minimal rank projective resolutions and rank exact decompositions of $M$ and $N$ look follows:
    \begin{align*}
        0 \to \Lsf_{(0,0),(5,5)}
            \to
        \Lsf_{(0,0),(0,5)} \oplus \Lsf_{(0,0),(5,0)}
            \to
        &M\\
        0 \to \Lsf_{(0,0),(5,4)} \oplus \Lsf_{(0,0),(4,5)}
            \to
        \Lsf_{(0,0),(0,5)} \oplus \Lsf_{(0,0),(4,4)}\oplus \Lsf_{(0,0),(5,0)}
            \to
        &N,
    \end{align*}
    \begin{align*}
        \barcrank_\spm(M) &=
            \left( \left\{
                [\Lsf_{(0,0),(0,5)}],
                [\Lsf_{(0,0),(5,0)}]
            \right\}, \left\{
                [\Lsf_{(0,0),(5,5)}]
            \right\}\right),\\
        \barcrank_\spm(N) &=
        \left( \left\{
        [\Lsf_{(0,0),(0,5)}],
        [\Lsf_{(0,0),(4,4)}],
        [\Lsf_{(0,0),(5,0)}]
        \right\}, \left\{
            [\Lsf_{(0,0),(5,4)}],
            [\Lsf_{(0,0),(4,5)}]
            \right\}\right).
    \end{align*}

    The modules $M$ and $N$ are $1$-interleaved, so \cref{theorem:stability-rank} implies that there exists $3$-matching between
    $
    \left( \left\{
    [\Lsf_{(0,0),(0,5)}],
    [\Lsf_{(0,0),(5,0)}],
    [\Lsf_{(0,0),(5,4)}],
    [\Lsf_{(0,0),(4,5)}]
    \right\}\right)
    $
    and
    $
    \left( \left\{
    [\Lsf_{(0,0),(0,5)}],
    [\Lsf_{(0,0),(4,4)}],
    [\Lsf_{(0,0),(5,0)}],
    [\Lsf_{(0,0),(5,5)}]
    \right\}\right).
    $
    Indeed, as illustrated in \cref{figure:example-stability}, there even exists a $1$-matching.
\end{example}

\subsection{Universality of the signed bottleneck dissimilarity}
\label{section:universality}

We start by giving a remark, showing that the signed bottleneck dissimilarity does not satisfy the triangle inequality.

\begin{remark}
    \label{remark:no-triangle-inequality}
    The same construction as in \cite[Example~4.1]{oudot-scoccola} shows that the dissimilarity \(\widehat{d_B}\) on \(\Rscr^2\)-barcodes does not satisfy the triangle inequality, even when restricted to signed barcodes arising as the rank exact decomposition of a \fpp \(\Rscr^2\)-persistence module.
    Although the example deals with the Betti signed barcode, that is, the signed barcode associated to the usual exact structure on \fpp \(\Rscr^2\)-persistence modules, it applies without any changes to the rank exact decomposition, since the (usual) minimal resolutions of the modules considered in \cite[Example~4.1]{oudot-scoccola} coincide with the minimal rank projective resolutions of the modules.
    Examples for \(n\geq 3\) can be constructed by considering the modules of \cite[Example~4.1]{oudot-scoccola} and taking the external tensor product with \(\Psf_0 \colon \Rscr^{n-2} \to \vect\), as in the beginning of the proof of \cref{proposition:instability-mrd-rectangles}.
\end{remark}

We now define the notions of balanced and \(\barcrank_{\spm}\)-stable dissimilarities on signed barcodes, and prove the universality result. 

\begin{definition}
    A dissimilarity function \(d\) on finite signed \(\Rscr^n\)-barcodes is \define{balanced} if
    \[
        d\left( (\Ccal_\sp, \Ccal_\sm \cup \Acal), (\Dcal_\sp, \Dcal_\sm) \right) = d\left( (\Ccal_\sp, \Ccal_\sm), (\Dcal_\sp \cup \Acal, \Dcal_\sm) \right)
    \]
    for all finite signed \(\Rscr^n\)-barcodes \(\Ccal_\spm\) and \(\Dcal_\spm\), and finite \(\Rscr^n\)-barcode \(\Acal\).
\end{definition}

\begin{definition}
    A dissimilarity function \(d\) on finite signed \(\Rscr^n\)-barcodes is \define{\(\barcrank_{\spm}\)-stable} if for every pair of \fpp modules \(M,N \colon \Rscr^n \to \vect\) we have \(d(\barcrank_\spm(M),\barcrank_\spm(N)) \leq d_I(M,N)\).
\end{definition}
It is straightforward to prove the existence of a most discriminative \(\barcrank_{\spm}\)-stable dissimilarity on signed $\Rscr^n$-barcodes.
\begin{proposition}
\label{proposition:universality-existence}
The collection of \(\barcrank_{\spm}\)-stable and balanced dissimilarity functions on finite signed \(\Rscr^n\)-barcodes has a maximum with respect to the pointwise order, denoted \(d_\msf\).
\end{proposition}
\begin{proof}
    Define \(d_\msf\) by
    \[
        d_\msf(\Ccal, \Dcal) = \sup \left\{ d(\Ccal, \Dcal) \; \colon \; \text{\(d\) is a \(\barcrank_{\spm}\)-stable and balanced dissimilarity on finite signed \(\Rscr^n\)-barcodes.}\right\},
    \]
    which is easily seen to be the largest \(\barcrank_{\spm}\)-stable and balanced dissimilarity on finite signed \(\Rscr^n\)-barcodes.
\end{proof}
We now show that for signed barcodes of hook modules, $\widehat{d_B}$ offers a good approximation of $d_\msf$.
\begin{proposition}
    \label{proposition:universality}
    When restricted to finite signed \(\Rscr^n\)-barcodes containing only hook modules, $\widehat{d_B}$ and $d_\msf$ satisfy the following bi-Lipschitz equivalence.
    \[
        \widehat{d_B}/(2n-1)^2 \leq d_\msf \leq \widehat{d_B}.
    \]
\end{proposition}

\begin{proof}
The first inequality is immediate from the fact that \(\widehat{d_B}/(2n-1)^2\) is \(\barcrank_{\spm}\)-stable (\cref{theorem:stability-rank}) and balanced, and the defining property of $d_\msf$ (\cref{proposition:universality-existence}). Note that this part of the bi-Lipschitz equivalence holds without the additional assumption that the signed barcodes contain only hook modules.

    To show the other part of the equivalence, Let $d$ be any \(\barcrank_{\spm}\)-stable and balanced dissimilarity function on finite signed \(\Rscr^n\)-barcodes. We now show that \(d(\Ccal_\spm,\Dcal_\spm) \leq \widehat{d_B}(\Ccal_\spm,\Dcal_\spm)\) for any two signed $\Rscr^n$-barcodes \(\Ccal_\spm\), \(\Dcal_\spm\) of hook modules. First, consider the modules \[M = \bigoplus_{[X] \in \Ccal_\sp \cup \Dcal_\sm} X\qquad \text{and}\qquad N = \bigoplus_{[Y] \in \Dcal_\sp \cup \Ccal_\sm} Y.\]
Observe that $\barcrank_\spm(M) = (\Ccal_\sp\cup \Dcal_\sm , \varnothing)$ and $\barcrank_\spm(N) = (\Dcal_\sp\cup \Ccal_\sm, \varnothing)$ by the assumption that the signed barcodes consist of hook modules.
Using this and the fact that \(d\) is balanced, we have \(d(\Ccal_\spm,\Dcal_\spm) = d(\barcrank_\spm(M), \barcrank_\spm(N))\), and by the \(\barcrank\)-stability of \(d\), we have \(d(\barcrank_\spm(M), \barcrank_\spm(N)) \leq d_I(M,N)\).
    As observed in \cref{section:matchings-bottleneck-distance-bottleneck-dissimilarity}, we have \(d_I(M,N) \leq d_B(\barc(M), \barc(N))\).
    Putting these together, we get 
    \(d(\Ccal_\spm,\Dcal_\spm) \leq d_B(\barc(M), \barc(N))\), and, by definition, we have \(d_B(\barc(M), \barc(N)) = \widehat{d_B}(\Ccal_\spm, \Dcal_\spm)\), concluding the proof of the universality result. 
\end{proof}

To conclude this section, we prove that the assumption of \(\barcrank_{\spm}\)-stability precludes the existence of discriminative metrics on signed barcodes.

\begin{proposition}
If \(d\) is a \(\barcrank_{\spm}\)-stable and balanced dissimilarity function on finite signed \(\Rscr^n\)-barcodes which satisfies the triangle inequality, then, for all \fpp \(M,N \colon \Rscr^n \to \vect\) such that \(d_I(M,N) < \infty\), we have \(d(\barcrank_\spm(M), \barcrank_\spm(N)) = 0\).
   \label{prop:no-stable-tri}\end{proposition}
\begin{proof}
Let \(d\) be a \(\barcrank_{\spm}\)-stable and balanced dissimilarity function on signed barcodes, and assume, further, that \(d\) satisfies the triangle inequality.
    By \cref{proposition:universality}, we have \(d \leq \widehat{d_B}\), so it is enough to show the following:  given \fpp \(M,N \colon \Rscr^n \to \vect\) with \(d_I(M,N) < \infty\) and \(\epsilon > 0\), there exists a signed barcode \(\Acal_\spm\) such that \(\widehat{d_B}(\barcrank_\spm(M),\Acal_\spm) + \widehat{d_B}(\Acal_\spm, \barcrank_\spm(N)) \leq \epsilon\).

We iteratively construct $\Acal_\spm$ and the following two $\epsilon/2$-matchings 
\[h_1: \barcrank(M)_\sp\cup \Acal_\sm \to \barcrank(M)_\sm\cup \Acal_\sp \qquad \text{and}\qquad h_2: \barcrank(N)_\sp\cup \Acal_\sm \to \barcrank(N)_\sm\cup \Acal_\sp.\]
Hooks that are at finite interleaving distance with the 0 module can be ''shrunk'' to 0 by interpolation. The remaining hooks must be paired up. 

\medskip 
\noindent\textit{Hooks with $j<\infty$.}
For $[L_{i,j}]$ with $j<\infty$ in either of the four barcodes of $\barcrank(M)_\spm$ and $\barcrank(N)_\spm$, add a copy of the following signed barcode to $\Acal_\spm$, \[ A_\spm = \left(\{[\Lsf_{i,j_{\ell-1}}], \dots, [\Lsf_{i,j_1}]\}, \{[\Lsf_{i,j_{\ell-1}}], \dots, [\Lsf_{i,j_1}]\}\right),\]
where  \(j = j_0, j_1, \dots, j_\ell = i\) is any sequence
such that \(\|j_k - j_{k+1}\|_\infty < \epsilon\) for all \(0 \leq k \leq \ell-1\). If $[L_{i,j}]\in \barcrank(M)_\sp$, then adjoin the following $\epsilon/2$-matching to $h_1$: 
\[h_1([L_{i,j}])=[L_{i,j_1}] \qquad h_1([L_{i, j_1}])= [L_{i,j_2}] \qquad \cdots\qquad h_1([L_{i,{j_{\ell-2}}}]) = [L_{i,{j_{\ell-1}}}], \]
and leave the remaining $[L_{i,j_{\ell-1}}]\in A_\sm$ unmatched. Adjoin the identity matching $A_\sm \to A_\sp$ to $h_2$. For $[L_{i,j}]$ in either of the other barcodes, we update $\Acal_\spm$ and the two matchings in the analogous way. 

\medskip
\noindent\textit{Hooks with $j=\infty$.}
For the infinite hook modules $[L_{i,\infty}]$, we note that, since  \(M\) and \(N\) are at finite interleaving distance, \cref{theorem:stability-rank} implies that there exists a  matching \(h\) between $\barcrank(M)_\sp\cup \barcrank(N)_\sm$ and $\barcrank(M)_\sm\cup \barcrank(N)_\sp$, and, in particular,  $h$ induces a perfect matching between the infinite hook modules. 

For a matched pair $[L_{i,\infty}]\in \barcrank(M)_\sp$ and $[L_{a,\infty}]\in \barcrank(M)_\sm$, add the following signed barcode to $\Acal_\spm$, 
\[A'_\spm = \left(\{[\Lsf_{i_0,\infty}], \dots, [\Lsf_{i_{\ell-1},\infty}]\}, \{[\Lsf_{i_1,\infty}], \dots, [\Lsf_{i_{\ell},\infty}]\}\right),\]
where \(i = i_0, i_1, \dots, i_\ell = a\) is any sequence such that \(\|i_k - i_{k+1}\|_\infty < \epsilon\) for all \(0 \leq k \leq \ell-1\). 
To $h_1$ we adjoin the identity matching $A'_\sm\cup \{[L_{i,\infty}]\}\to A'_\sp\cup\{[L_{a,\infty}]\}$, and to $h_2$ we adjoin the following $\epsilon/2$-matching
\[h_1([L_{i_0,\infty}])=[L_{i_1,\infty}] \qquad h_1([L_{i_1, \infty}])= [L_{i_2,\infty}] \qquad \cdots\qquad h_1([L_{i_{\ell-1},\infty}]) = [L_{i_{\ell},\infty}].\] 
The case $[L_{i,\infty}]\in \barcrank(N)_\sp$ and $[L_{a,\infty}]\in \barcrank(N)_\sm$ is symmetrical. 

For a matching of $[L_{i,\infty}]\in \barcrank(M)_\sp$ and $[L_{a,\infty}]\in \barcrank(N)_\sp$ in $h$, add the following signed barcode to $\Acal_\spm$, 
\[A''_\spm = \left(\{[\Lsf_{i_0,\infty}], \dots, [\Lsf_{i_{\ell -1},\infty}]\}, \{[\Lsf_{i_1,\infty}], \dots, [\Lsf_{i_{\ell-1},\infty}]\}\right),\]
 where \(\{i_j\}_{j=0}^\ell\) is as above. To $h_1$ we adjoin the identity matching $A''_\sm\cup \{[L_{i,\infty}]\}\to A''_\sp$, and to $h_2$ we adjoin the $\epsilon/2$-matching
   \[h_2([L_{i_1,\infty}])=[L_{i_0,\infty}] \qquad h_2([L_{i_2, \infty}])= [L_{i_1,\infty}] \qquad \cdots\qquad h_2([L_{i_{\ell-1},\infty}]) = [L_{i_{\ell-2},\infty}],\] 
   and $h_2([L_{a,\infty}]) = [L_{i_{\ell-1},\infty}]$.
   
For a matching of $[L_{i,\infty}]\in \barcrank(M)_\sm$ and $[L_{a,\infty}]\in \barcrank(N)_\sm$ in $h$, add the following signed barcode to $\Acal_\spm$, 
\[A'''_\spm = \left(\{[\Lsf_{i_1,\infty}], \dots, [\Lsf_{i_{\ell -1},\infty}]\}, \{[\Lsf_{i_0,\infty}], \dots, [\Lsf_{i_{\ell-1},\infty}]\}\right)\]
and update the matchings accordingly. This shows that 
\[\widehat{d_B}(\barcrank(M)_\spm, A_\spm) + \widehat{d_B}(\barcrank(N)_\spm, A_\spm) \leq \epsilon/2 + \epsilon/2 = \epsilon.\]
\end{proof}

\section{Other exact structures}
\label{sec:other-exact-structures}


So far we have focused exclusively on the rank exact structure, and a natural question is to what extent our analysis extends to other exact structures. As mentioned in the introduction, we have been following the roadmap outlined in~\cite[Theorem~8.4]{oudot-scoccola}, which requires the following two ingredients in order to guarantee bottleneck stability for $\Ecal$-signed barcodes derived from a given  exact structure~$\Ecal$:
\begin{enumerate}
\item\label{item:stab-proj} a stability result for the $\Ecal$-projectives in the usual (unsigned) bottleneck distance; 
\item\label{item:gldim} a finite upper bound on $\gldimEcal$.
\end{enumerate}
Then, the constant factor in the stability bound for $\Ecal$-signed barcodes is equal to the one for the $\Ecal$-projectives, multiplied by $1$ plus the upper bound on~$\gldimEcal$.
Hence the importance of deriving bounds that are as tight as possible in ingredients~\ref{item:stab-proj} and~\ref{item:gldim}. 

In the case of the usual exact structure studied in~\cite{oudot-scoccola}, the two ingredients were already available in the literature, in the form of Bjerkevik's bottleneck stability result for free modules on the one hand, of Hilbert's syzygy theorem on the other hand.
Bjerkevik's result is known to be tight when the number of parameters is \(2\) or \(4\), while Hilbert's syzygy theorem is tight for any number of parameters.

In the case of the rank exact structure, \cref{proposition-stability-hooks} in this paper provides the first ingredient, with a bound that is tight at least in the two-parameter setting (as per \cref{remark:projective-stability-tight}), while \cref{theorem:gldim-rank-Rn} provides the second ingredient with a tight bound.

Investigating exact structures associated to classes of interval modules other than the hook modules is a natural next step.
Some of these exact structures were studied recently, and upper bounds on their global dimensions were obtained~\cite{asashiba-escolar-nakashima-yoshiwaki,blanchette-brustle-hanson}.
However, these bounds grow with the size of the indexing poset, thus not giving a finite upper bound in the case of \fpp $\Rscr^n$-persistence modules.
This makes them not directly usable in our approach for proving bottleneck stability.
The fact that the bounds diverge does not seem to be an artifact of the proofs in~\cite{asashiba-escolar-nakashima-yoshiwaki,blanchette-brustle-hanson}, as the global dimension of an exact structure on a category or persistence modules seems to generally increase with the number of degrees of freedom required to specify an indecomposable relative projective.
Indeed, as our \cref{theorem:global-dimension-upset} shows, $\gldimEcal$ provably diverges to infinity in the case where \(\Ecal = \Eupset\), the exact structure that has as indecomposable projectives the interval modules with support an upset.
The rest of the section is devoted to the proof of this result.
After giving the result, we give \cref{remark:connection-miller}, which comments on the relationship between $\Eupset$-projective resolutions and the upset resolutions considered in \cite{miller}.

\bigskip

Let \(\Sscr\) be a poset.
An \define{upset module} is an interval module \(M \in \modcat_{\Sscr}\) such that \(M \cong \kbb_{U}\) for \(U \subseteq \Sscr\) an upset of \(\Sscr\).
A \(\kbb\Sscr\)-module is \define{upset-decomposable} if it is isomorphic to a direct sum of upset modules.

Assume now that $\Sscr$ is finite; the following considerations follow from \cite[Proposition~1.10]{auslander-solberg}.
The poset $\Sscr$ being finite implies that there exist finitely many isomorphism types of upset modules; moreover, \(\kbb\Sscr\)-projective modules are upset-decomposable, so there exists an exact structure \(\Eupset\) on \(\modcat_{\kbb\Sscr}\) such that the \(\Eupset\)-projective modules are precisely the upset-decomposable modules.
Also, the exact structure $\Eupset$ has enough projectives in the sense that, for every \(\Sscr\)-persistence module \(M\), there exists a \(\Eupset\)-exact sequence \(0 \to K \to P \to M \to 0\) with \(P\) a \(\Eupset\)-projective module.
We call \(\Eupset\) the \define{limit exact structure};
the name and notation come from the following characterization of the $\Eupset$-exact sequences:
a sequence \(0 \to L \to M \to N \to 0\) is $\Eupset$-exact precisely when, for every upward closed subposet \(U \subseteq \Pscr\), the induced sequence of vector spaces \(0 \to \varprojlim_U L \to \varprojlim_U M \to \varprojlim_U N \to 0\) is exact.
To prove this characterization of the $\Eupset$-exact sequences, one uses the natural isomorphism $\lim_U M \cong \hom(\kbb_U, M)$, and the fact that any $\Eupset$-projective module decomposes as a direct sum of modules of the form $\kbb_U$, by the above considerations.

In some cases, the limit exact structure also exists when $\Sscr$ is not finite; this is studied in \cite{blanchette-brustle-hanson-2}.
For example, let $\Rscr_{\geq 0} \subseteq \Rscr$ be the subposet of non-negative numbers.
Then, the limit exact structure on \(\vect^{\Rscr_{\geq 0}^n}_\fpp\) exists and has enough projectives; this follows from, e.g., \cite[Corollary~7.22]{blanchette-brustle-hanson-2}, see \cite[Example~7.23(1)]{blanchette-brustle-hanson-2} where this is explained.
In this case too, the indecomposable relative projectives are the upset modules.

Before giving the main result of this section, we need a technical lemma.

\begin{lemma}
    \label{lemma:bound-above-hasse-diagram}
    Let \(m \geq 2\) and let \(\Sscr = [m]^2\).
    Let \(V \subseteq \Sscr\) be an upset, let \(C\) be the set of all upsets of \(\Sscr\) that are contained in \(V\), and let \(D \subseteq C\setminus\{V\}\) consist of the maximal elements of \(C\setminus\{V\}\) with respect to inclusion.
    Then \(|D| \leq m\).
\end{lemma}
\begin{proof}
    Let \(B\) be the set of minimal elements of \(V\).
    Start by observing that \(|B| \leq m\), which follows from the fact that the map \(B \to [m]\) that maps \((x,y)\) to \(x\) must be injective, since the elements of \(B\) are all incomparable.

    Now, if \(V' \subsetneq V\) is a maximal sub-upset, then there exists \(b \in B\) such that \(b \notin V'\), since otherwise \(B \subseteq V'\) and thus \(V' = V\).
    Let \(U^b = \{a \in \Sscr \colon a \nleq b\}\), which is an upset.
    Note that \(U^b \cap V \subsetneq V\) is an upset.
    We claim that \(V' \subseteq U^b \cap V \subsetneq V\) and that \(U^b \cap V \subsetneq V\) is a maximal sub-upset.
    The fact that \(V' \subseteq U^b \cap V\) is clear, since \(b \notin V'\).
    To show that \(U^b \cap V \subsetneq V\) is maximal, assume we have \(U^b \cap V \subsetneq W \subseteq V\).
    In that case, we must have \(a \in W\) such that \(a \notin U^b\), and thus \(a \leq b\).
    It follows that \(b \in W\).
    Since for all \(b' \in B \setminus \{b\}\) we have \(b' \in U^b\), it follows that \(B \subseteq W\), and thus \(W = V\).
    This shows that \(U^b \cap V \subsetneq V\) is a maximal sub-upset, and, since \(V' \subseteq U^b \cap V\) and \(V' \subsetneq V\) is also a maximal sub-upset, we must have \(V' = U^b \cap V\).

    The above implies that the function from \(B\) to maximal sub-upsets of \(V\) that maps \(b\) to \(U^b \cap V\) is surjective, which proves the result.
\end{proof}

\begin{theorem}
    \label{theorem:global-dimension-upset}
    If \(m \geq 3\), we have \(\gldimupset\left(\vect^{[m]^2}\right) = m-2\).
%
\end{theorem}
\begin{proof}
    For notational convenience, let $\Sscr = [m]^2$.
    We start with a general construction.
    Consider the right \(\kbb \Sscr\)-module
    \[
        \Ubb = \bigoplus_{\substack{U \subseteq \Sscr,\\\text{upset}}} \kbb_{U}.
    \]
    Let \(\Pscr\) denote the opposite of the poset of upsets of \(\Sscr\), so that the elements of \(\Pscr\) are given by the upsets of \(\Sscr\), and we have \(U \leq V \in \Pscr\) exactly when \(V \subseteq U \subseteq \Sscr\).
    The poset \(\Pscr\) is a lattice, where \(U \wedge V = U \cup V\) and \(U \vee V = U \cap V\).
    A routine check shows that, as algebras, we have \(\End_{\kbb \Sscr}(\Ubb) \cong \kbb \Pscr\), and thus \(\Ubb\) is a \((\kbb \Pscr,\kbb \Sscr)\)-bimodule.
    This implies that we have a functor \(\Hom_{\kbb \Sscr}(\Ubb, -) \colon \modcat_{\kbb \Sscr} \to \modcat_{\kbb \Pscr}\).
    Following the same idea as in the proof of \cref{lemma.equivalence_X_and_L}, one shows that, for all \(M \in \modcat_{\kbb \Sscr}\), we have \(\pdimupset(M) = \pdim_{\kbb \Pscr}(\Hom(\Ubb, M))\).

    We now show that \(\gldim(\kbb\Pscr) = m\), starting with \(\gldim(\kbb\Pscr) \geq m\).
    Consider the set \(B = \{(a,b) \colon a + b = m-1\} \subseteq \Sscr\) and the upset \(V\) generated by \(B\).
    We claim that \(\pdim_{\kbb \Pscr}(\Ssf_V) = m\).
    To see this, we use \cref{prop.gldim_for_kP/e}, with \(\Qscr \subseteq \Pscr\) being the empty poset.
    We must thus show that, in \(\Pscr\), we have \(|V^+| = m\) and that, for every proper subset \(A \subsetneq V^+\) we have \(\vee A < \vee V^+\).
    Note that \(B = \{(0,m-1), (1,m-2), \dots, (i, m-1-i), \dots, (m-2,1), (m-1,0)\}\).
    Given \(i \in [m]\), let \(B^i\) be given by removing the \(i\)th element of \(B\), that is, let \(B^i = B \setminus \{(i, m-1-i)\}\).
    Let \(V^i\) be the upset generated by \(B^i\).
    We have \(V^+ = \{V^i\}_{i \in [m]}\), so that \(|V^+| = m\).
    To conclude the proof that \(\gldim(\kbb\Pscr) \geq m\), we must show that, if \(B^i \notin A \subsetneq V^+\), then \(\vee A < \vee V^+\).
    In terms of sets, we must show that
    \[
        \bigcap_{j \in [m]} V^j \subsetneq \bigcap_{j \in [m]\setminus \{i\}} V^j.
    \]
    This follows from the fact that \((i,m-1-i) \in \bigcap_{j \in [m]\setminus \{i\}} V^j\), but \((i,m-1-i) \notin V^i\).

    We now show that \(\gldim(\kbb \Pscr) \leq m\).
    By a standard result (see, e.g., \cite[Corollary~1.14,~Section~7]{mcconnell-robson}), it is enough to show that, for every simple module \(\Ssf_V \in \modcat_{\kbb \Pscr}\) corresponding to an upset \(V \subseteq \Sscr\), we have \(\pdim_{\kbb \Pscr}(\Ssf_V) \leq m\).
    Using \cref{prop.gldim_for_kP/e} with \(\Qscr \subseteq \Pscr\) being the empty set, it is thus enough to show that, for every upset \(V \subseteq \Sscr\), seen as an element \(V \in \Pscr\), we have \(|V^+| \leq m\), which is the content of \cref{lemma:bound-above-hasse-diagram}.

    We now show that \(\gldimupset(\modcat_{\kbb \Sscr}) \leq m-2\).
    Let \(i \in \Sscr\) and let \(U \subseteq \Sscr\) be an upset.
    Since any map \(\kbb_U \to \Rsf_{0,i}\) factors through the surjective map \(\Psf_0 \to \Rsf_{0,i}\), it follows that \(0 \to \kbb_{i^\wedge} \to \Psf_0 \to \Rsf_{0,i}\) is a minimal \(\Eupset\)-projective resolution, and thus that \(\pdimupset(I) = 1\) for any non-zero \(\kbb\Sscr\)-injective module \(I\).

    As in the proof of the upper bound of \cref{theorem:gldim-rank-poset}(2), given \(M \in \modcat_{\kbb \Sscr}\), we consider an exact sequence 
    \(0 \to M \to I \to J\) with $I$ and $J$ injective, and the induced exact sequence
    \[
        0 \to \Hom(\Ubb,M) \to \Hom(\Ubb, I) \to \Hom(\Ubb, J) \to \coker\left(\Hom(\Ubb, I) \to \Hom(\Ubb, J)\right) \to 0.
    \]
    By the previous considerations, we have \(\pdimupset(M) \leq \max\{1,\gldim(\kbb\Pscr)-2\} = m-2\).

    Finally, we show that \(\gldimupset(\modcat_{\kbb \Sscr}) \geq m-2\).
    Using that \(\gldim(\kbb\Pscr) = m\), let \(N \in \kbb \Pscr\) be such that \(\pdim_{\kbb \Pscr}(N) = m\).
    Then, following the proof of the lower bound in \cref{theorem:gldim-rank-poset}(2), we can consider a minimal \(\kbb \Pscr\)-projective resolution \(P_\bullet \to N\), which must be of the form \(\Hom(\Ubb,X_\bullet) \to N\), for a chain complex of \(\kbb\Sscr\)-modules \(X_\bullet\) of length \(m\).
    We can then consider the \(\kbb\Sscr\)-module \(\ker(X_1 \to X_0)\), and note that, by the fact that \(\Hom(\Ubb,-)\) is left exact, we have \(\Hom(\Ubb, \ker(X_1 \to X_0)) \cong \ker(P_1 \to P_0)\), so that \(\pdimupset(\ker(X_1 \to X_0)) = m-2\).
\end{proof}

\begin{corollary}
    \label{corollary:infty-dim-upset-exact}
    We have \(\gldimupset\left(\vect^{\Rscr_{\geq 0}^2}_{\textup{\fpp}}\right) = \infty\).
\end{corollary}
\begin{proof}
    Towards a contradiction, assume that there exists \(d \in \Nbb\) such that every \fpp \(N \colon \Rscr_{\geq 0}^2 \to \vect\) admits an \(\Eupset\)-projective resolution of length at most \(d\).
    Let \(m = d + 3\), and let \(M \colon [m]^2 \to \vect\) be such that \(\pdimupset(M) = m - 2 = d + 1\), which must exist by \cref{theorem:global-dimension-upset}.

    Consider the monotonic injection \(\iota \colon [m]^2 \to \Rscr_{\geq 0}^2\) mapping \((a,b)\) to \((a,b)\).
    Let \(\iota_! \colon \modcat_{\kbb [m]^2} \to \vect^{\Rscr_{\geq 0}^2}_\fpp\) denote left Kan extension along \(\iota\).
    It is easy to see that any minimal $\Eupset$-projective resolution $P_\bullet \to N$ has the property that $\iota_!(P_\bullet) \to \iota_!(N)$ is a minimal $\Eupset$-projective resolution of $\iota_!(N)$; see, e.g., \cite[Example~7.24(1)]{blanchette-brustle-hanson-2}.
    But this resolution is of length $m-2 = d + 1$, contradicting the fact that every \(N \colon \Rscr_{\geq 0}^2 \to \vect\) admits an \(\Eupset\)-projective resolution of length at most \(d\).
\end{proof}

\begin{remark}
    \label{remark:connection-miller}
    As a consequence of \cref{theorem:global-dimension-upset}, for every module \(M \colon [m]^2 \to \vect\), there exists an exact sequence of finite length \(0 \to P_\ell \to \cdots \to P_0 \to M \to 0\) with all of the modules \(P_i\) being upset-decomposable modules.
    Finite exact sequences \(0 \to P_\ell \to \cdots \to P_0 \to M \to 0\) in \(\vect^{[m]^2}\) with the property that all of the modules \(P_i\) are upset-decomposable
    are \define{finite upset resolutions} as per \cite{miller}.
    Thus, in the special case where the indexing poset is \([m]^2\), \cref{theorem:global-dimension-upset} refines \cite[Theorem~6.12(6)]{miller}, since it shows the existence of finite resolutions by upset-decomposable modules, which are not only exact but \(\Eupset\)-exact, a stronger condition.
    This is interesting, since, to the best of our knowledge, there is no notion of minimality for upset resolutions in the sense of \cite{miller} that would imply that minimal upset resolutions are unique up to isomorphism, even when the indexing set is a finite grid $[m]^2$.
    Meanwhile, every module $[m]^2 \to \vect$ admits a minimal $\Eupset$-projective resolution, and, in this case, minimality does imply uniqueness up to isomorphism, as this is true for general exact structures (see \cref{section:exact-structures}).
\end{remark}

\appendix

\section{Proofs of results from previous work}
\label{section:appendix}

\rankexactstructurefacts*

\begin{proof}
    Statement~(1) for the case of $\Pscr$ a finite poset follows from \cite[Theorems~4.4~and~4.8]{botnan-oppermann-oudot}.
    Statement~(1) for the case of \(\vect_{\fpp}^{\Rscr^n}\) follows from \cite[Theorem~4.4~and~Proposition~4.14]{botnan-oppermann-oudot}.
    In both cases, one uses the fact that the existence of finite resolutions implies the existence of finite minimal resolutions.

    \medskip

    Before proving statement (2), we prove an intermediate result.
    We claim that there exists \(m_1, \dots, m_n \in \Nbb\), a monotonic injective function \(f \colon [m_1] \times \cdots \times [m_n] \to \Rscr^n\), and a module \(M' \colon [m_1] \times \cdots \times [m_n] \to \vect\), such that \(M\) is isomorphic to the left Kan extension of \(M'\) along \(f\), and such that the image of $f$ is the sublattice of $\Rscr^n$ generated by the grades of the indecomposable projectives involved in a minimal projective presentation of $M$.
    This is standard and is implicit in the proof of \cite[Proposition~4.14]{botnan-oppermann-oudot}, but let us give a proof for completeness.

    Let \(Q \to P\) be a minimal presentation of $M$, so in particular \(Q \cong \bigoplus_{j \in J} \Psf_j\), \(P \cong \bigoplus_{i \in I} \Psf_i\), \(I\) and \(J\) are finite multisets of elements of \(\Rscr^n\), and \(M\) is isomorphic to the cokernel of the morphism \(Q \to P\).

    Consider the sets \(A_1, \dots, A_n \subseteq \Rscr\) with $A_k$ the set of all possible $k$-coordinates of elements in \(I\) or \(J\).
    Formally, let $A_k = \{ r \in \Rscr \colon \exists a \in I \cup J, r = a_k\} \subseteq \Rscr$.
    Moreover, by endowing each \(A_k\) with the linear order induced by that of \(\Rscr\), there is a unique monotonic bijection \(\alpha_k : [m_k] \to A_k\), where \(m_k = |A_k|\).
    It is straightforward to see that $A_1 \times \cdots \times A_k \subseteq \Rscr^n$ is the sublattice of $\Rscr^n$ generated by the grades of the indecomposable projectives appearing in $P$ and $Q$.
    Let $f = \alpha_1 \times \cdots \times \alpha_n : [m_1] \times \cdots \times [m_n] \to A_1 \times \cdots \times A_n \subseteq \Rscr^n$.
    Let us denote $C \coloneqq [m_1] \times \cdots \times [m_n]$ and $A \coloneqq A_1 \times \cdots \times A_n$.
    We get $C$-persistence modules \(M', P', Q' \colon C \to \vect\) by restricting \(M\), \(P\), and \(Q\), respectively, along $f$.
    To prove the claim, it is sufficient to prove that $M$ is isomorphic to the left Kan extension of $M'$ along $f$.


    Let \(U = \{b \in \Rscr^n \colon \exists a \in A, a \leq b\}\).
    Consider the monotonic function \(g \colon U \to A\) which maps \(b \in U\) to \(\vee \{a \in A \colon a \leq b\}\).
    It is easy to see that, for any \(N \colon C \to \vect\), the restriction of \(\Lan_f N\) to \(U\) is isomorphic to the composite \(N \circ \alpha^{-1} \circ g\), whereas the restriction to the complement \(U^c\) is identically \(0\).
    This implies that \(\Lan_f\) is exact, and that \(\Lan_f(Q') \cong Q\) and \(\Lan_f(P')\cong P\).
    In turn, this implies that \(\Lan_f(M') \cong M\), as required.

    We now prove statement (2).
    Let \(M'\) be as in the previous claim.
    From statement (1) it follows that \(M'\) admits a finite minimal $\Erank$-projective resolution.
    By \cite[Lemmas~4.12~and~4.13]{botnan-oppermann-oudot}, \(\Lan_f\) preserves minimal rank projective resolutions, concluding the proof.
\end{proof}

The following is a persistent version of Schanuel's lemma for relative projective covers, which we use to prove the persistent version of Schanuel's lemma for relative projective resolutions.

\begin{lemma}
    \label{lemma:relative-persistent-schanuel}
    Let \(M,N \colon \Rscr^n \to \vect\) and let \(0 \to K \to P \xrightarrow{\alpha} M \to 0\) and \(0 \to L \to Q \xrightarrow{\gamma} N \to 0\) be \(\Ecal\)-exact sequences.
    Assume that \(P\) and \(Q\) are \(\Ecal\)-projective.
    If \(M\) and \(N\) are \(\epsilon\)-interleaved, then \(P \oplus L[\epsilon]\) and \(K[\epsilon] \oplus Q\) are \(\epsilon\)-interleaved.
\end{lemma}

\begin{proof}
    Let \(f \colon M \to N[\epsilon]\) and \(g \colon N \to M[\epsilon]\) form an \(\epsilon\)-interleaving.
    Define the submodule \(X \subseteq P \oplus Q[\epsilon]\) such that \(X(r) = \{ (a,b) \in P(r) \oplus Q(r+\epsilon) \colon f_r(\alpha_r(a)) = \gamma_{r+\epsilon}(b)\}\).
    Diagrammatically, we have the following pullback square:
    \[
        \begin{tikzpicture}
            \matrix (m) [matrix of math nodes,row sep=2em,column sep=4em,minimum width=2em,nodes={text height=1.75ex,text depth=0.25ex}]
            { X           & P                                     \\
              Q[\epsilon] & N[\epsilon]. \\};
            \path[line width=0.75pt, -{>[width=8pt]}]
            (m-1-1) edge [above] (m-1-2)
            (m-1-1) edge (m-2-1)
            (m-2-1) edge [above] node {\(\gamma[\epsilon]\)} (m-2-2)
            (m-1-2) edge [right] node {\(f \circ \alpha\)} (m-2-2)
            ;
        \end{tikzpicture}
    \]
    The kernel of the bottom horizontal map is \(L[\epsilon]\), so the kernel of the top horizontal map is isomorphic to $L[\epsilon]$ as well.
    Moreover, in any exact structure deflations are closed under pullback, so the top horizontal map is part of an \(\Ecal\)-exact sequence \( 0 \to L[\epsilon] \to X \to P \to 0\).
    Since \(P\), the last term in the sequence, is \(\Ecal\)-projective, the sequence splits, and thus \(X \cong L[\epsilon] \oplus P\).
    We define, analogously, a persistence module \(Y\) with \(Y(r) = \{ (a,b) \in P(r+\epsilon) \oplus Q(r) \colon \alpha_{r+\epsilon}(a) = g_r(\gamma_r(b))\}\).
    By a symmetric argument, it follows that \(Y \cong K[\epsilon] \oplus Q\).

    We now prove that \(X\) and \(Y\) are \(\epsilon\)-interleaved.
    Consider, for each \(r \in \Rscr^n\), the morphism \(X(r) \to Y(r+\epsilon)\) that maps \((a,b) \in P(r) \oplus Q(r+\epsilon)\) to \((\phi^{P}_{r,r+2\epsilon}(a),b) \in P(r+2\epsilon) \oplus Q(r+\epsilon)\).
    This is well defined since, if \((a,b) \in X(r)\), then \(f_r(\alpha_r(a)) = \gamma_{r+\epsilon}(b)\) and thus
    \begin{align*}
        \alpha_{r+2\epsilon}(\phi^{P}_{r,r+2\epsilon}(a)) & = \phi^{M}_{r,r+2\epsilon}(\alpha_{r}(a))                                    \\
                                                          & = g_{r+\epsilon}(f_r(\alpha_r(a))) = g_{r+\epsilon}(\gamma_{r+\epsilon}(b)).
    \end{align*}
    These morphisms assemble into a morphism \(X \to Y[\epsilon]\), and there is an analogous morphism \(Y \to X[\epsilon]\).
    It is easy to see that these morphisms form an \(\epsilon\)-interleaving.
\end{proof}

\relativestabilityresolutions*

\begin{proof}
    We may assume that the \(\Ecal\)-projective resolutions have length at most \(\ell \in \Nbb\).
    We proceed by induction on \(\ell\).
    If \(\ell = 0\), there is nothing to prove.
    Let \(\ell \geq 1\) and consider the \(\Ecal\)-exact sequences
    \(0 \to \ker \alpha \to P_0 \xrightarrow{\alpha} M \to 0\) and
    \(0 \to \ker \gamma \to Q_0 \xrightarrow{\gamma} N \to 0\).
    By \cref{lemma:relative-persistent-schanuel}, we have that
    \((\ker \gamma)[\epsilon] \oplus P_0\) and \((\ker \alpha)[\epsilon] \oplus Q_0\) are \(\epsilon\)-interleaved.
    We can then use the inductive hypothesis on the \(\Ecal\)-projective resolutions of length at most $\ell-1$
    \begin{equation*}\label{inductive-step-equation-schanuel} \begin{split}
            0 &\to P_\ell[\epsilon] \to \cdots \to P_2[\epsilon] \to P_1[\epsilon] \oplus Q_0 \to (\ker \alpha)[\epsilon] \oplus Q_0,\\
            0 &\to Q_\ell[\epsilon] \to \cdots \to Q_2[\epsilon] \to Q_1[\epsilon] \oplus P_0 \to (\ker \gamma)[\epsilon] \oplus P_0,
        \end{split} \end{equation*}
    of \((\ker \alpha)[\epsilon] \oplus Q_0\) and \((\ker \gamma)[\epsilon] \oplus P_0\), respectively, concluding the proof.
\end{proof}

\persistencemodulesaremodules*
\begin{proof}
    We start by defining an additive functor $F : \Vect^\Pscr \to \Modcat_{\kbb \Pscr}$.
    Given a functor $M : \Pscr \to \Vect$, consider the vector space $F(M) \coloneqq \bigoplus_{i \in \Pscr} M(i)$.
    Given $[i,j] \in \kbb \Pscr$ and $x \in M(k)$, define the action $x \cdot [i,j]$ to be $0$ if $i\neq k$ and $\phi_{i,j}^M(x)$ if $i=k$.
    It is clear that this extends to an action making $F(M)$ into a right $\kbb \Pscr$-module.

    Given a natural transformation $f : M \to N$ with $M,N : \Pscr \to \Vect$, define a linear map $F(f) : F(M) \to F(N)$ by linearly extending the rule that maps $x \in M(i)$ to $f_i(x) \in N(i)$; it is clear that this is a morphism of $\kbb\Pscr$-modules.
    The fact that $F$ is additive follows immediately from its definition.

    \medskip

    We now define a functor $G : \Modcat_{\kbb \Pscr} \to \Vect^\Pscr$.
    Given a right $\kbb \Pscr$-module $A$, consider the functor $G(A) : \Pscr \to \Vect$ defined on objects by $G(A)(i) \coloneqq A \cdot [i,i] = \{x \cdot [i,i] : x \in A\}$.
    Given $i \leq j \in \Pscr$, let $\phi^{G(A)}_{i,j}$ be defined by mapping $x \cdot [i,i]$ to $x \cdot [i,j]$; this is well defined since, if $x \cdot [i,i] = y \cdot [i,i]$, then $x \cdot [i,j] = x \cdot [i,i] \cdot [i,j] = y \cdot [i,i] \cdot [i,j] = y \cdot [i,j]$.
    It is clear that these definitions make $G(A)$ a functor.

    Given a morphism of $\kbb \Pscr$-modules $g : A \to B$, define a natural transformation $G(g)$ with component corresponding to $i \in \Pscr$ given by mapping $x \cdot [i,i]$ to $g(x) \cdot [i,i]$.
    The naturality of $G(g)$ follows directly from the fact that $g$ respects the action of $\kbb \Pscr$.
    The fact that $G$ is additive follows immediately from its definition.

    \medskip

    We now show that the composite $G \circ F$ is naturally isomorphic to the identity.
    This proves that $F$ is faithful.
    If $M : \Pscr \to \Vect$, then, by definition of $F$ and $G$,
    \[
        G(F(M))(i) = \left(\bigoplus_{j \in \Pscr} M(j)\right) \cdot [i,i] = M(i),
    \]
    since $[i,i]$ annihilates all components of $\bigoplus_{j \in \Pscr} M(j)$ except for the $i$th one, where it acts as the identity.
    This gives a transformation $G(F(M)) \to M$, with all components being isomorphisms, and which is easily checked to be natural.

    \medskip

    To see that $F$ is full, note that any morphism $F(M) \to F(N)$ is determined by its action on the vector spaces $M(i) = F(M)\cdot[i,i] \to F(N)\cdot[i,i] = N(i)$.

    \medskip

    Finally, if $\Pscr$ is finite and $M : \Pscr \to \vect$, it is clear that $F$ restricts to a functor \(F' : \vect^\Pscr \to \modcat_{\kbb \Pscr}\), and that $G$ restricts to a functor \(G' : \modcat_{\kbb \Pscr} \to \vect^\Pscr\).
    To conclude, we prove that the composite $F' \circ G'$ is naturally isomorphic to the identity.
    Since $\Pscr$ is finite, the element $\sum_{i \in \Pscr} [i,i] \in \kbb \Pscr$ is well defined, and is in fact the unit of the algebra.
    This implies that
    \begin{align*}
        F'(G'(A)) = \bigoplus_{i \in \Pscr} G'(A)(i) = \bigoplus_{i \in \Pscr} A \cdot [i,i] \cong A.
    \end{align*}
    concluding the proof.
\end{proof}

\bounddimensionles*
\begin{proof}
    By convention, let $X_i = 0$ if $i \notin \{0, \dots, \ell\}$.
    Denote the morphisms in the long exact sequence by $d_i : X_i \to X_{i-1}$, extending them by zero to $i \notin \{0, \dots, \ell\}$.
    By exactness, we have short exact sequences as follows
    \[
        0 \to \coker(d_{i+2}) \to X_i \to \ker(d_{i-1}) \to 0,
    \]
    which implies that $\dim(X_i) \leq \dim(\coker(d_{i+2})) + \dim(\ker(d_{i-1})) \leq \dim(X_{i+1}) + \dim(X_{i-1})$.
    Let $x \in \{0,1,2\}$.
    We then get the desired result by summing the above inequality over $i \equiv x (\mod\, 3)$.
\end{proof}

\bibliographystyle{alpha}
\bibliography{biblio}

\end{document}